\documentclass[a4paper,12pt]{amsart}

\title[Log splitting and loop theorems]{Splitting and loop theorems in logarithmic Gromov--Witten theory}

\usepackage[dash,dot]{dashundergaps}

\usepackage{amsmath, amssymb, mathrsfs, amsthm, shorttoc, stmaryrd}
\usepackage{mathtools} 
\usepackage{hyperref}
\usepackage[dvipsnames]{xcolor}
\usepackage{subcaption}
\hypersetup{
colorlinks,
linkcolor={black},
citecolor={blue!50!black},
urlcolor={blue!80!black}
}

\usepackage{aliascnt}

\usepackage[all]{xy}

\usepackage{bbm}

\usepackage{tabularx}
\usepackage{longtable}
\numberwithin{equation}{subsection}

\usepackage{enumitem}

\usepackage{cleveref}
\crefformat{equation}{(#2#1#3)}
\let\oref\ref
\let\tilde\widetilde
\AtBeginDocument{\renewcommand{\ref}[1]{\Cref{#1}}}

\usepackage{tikz}
\usetikzlibrary{matrix, calc, arrows, patterns}

\usetikzlibrary{decorations.pathmorphing, decorations.pathreplacing, knots, calligraphy}
\usetikzlibrary{arrows.meta,calc,decorations.markings}

\usepackage{tikz-cd} 

\newcommand{\scr}[1]{\mathscr #1}

\newcommand{\OO}{\mathscr O}

\newcommand{\bra}[1]{{\left[#1\right]}}
\newcommand{\pra}[1]{{\left(#1\right)}}

\renewcommand{\frak}[1]{\mathfrak{#1}}

\newcommand{\Cl}[1]{C^{\rm log}_{#1}}

\newcommand{\lccx}[1]{\mathbb{L}^{\rm log}_{#1}}

\newcommand{\lkah}[1]{\Omega^{\rm log}_{#1}}

\newcommand{\WR}{{\scr R}}

\newcommand{\longsimeq}{\overset{\sim}{\longrightarrow}}

\newcommand{\on}[1]{\operatorname{#1}}
\newcommand{\bb}[1]{{\mathbb{#1}}}

\newcommand{\ca}[1]{{\mathcal{#1}}}
\newcommand{\bd}[1]{{\mathbf{#1}}}

\newcommand{\ul}[1]{{\underline{#1}}}

\newcommand{\Span}[1]{\left<#1\right>}

\newcommand*{\cal}[1]{\mathcal{#1}}

\newcommand*{\af}[1]{\ca A_{#1}}

\newcommand*{\Aff}{\mathbb{A}}

\def\Alt{\operatorname{Alt}}

\newcommand{\sharedthm}[2]{
  \newaliascnt{#1}{definition}%
  \newtheorem{#1}[#1]{#2}
  \aliascntresetthe{#1}%
}

\theoremstyle{definition}
\newtheorem{definition}{Definition}[subsection]   
\sharedthm{condition}{Condition}
\sharedthm{problem}{Problem}
\sharedthm{situation}{Situation}
\sharedthm{construction}{Construction}
\sharedthm{question}{Question}

\theoremstyle{plain}
\sharedthm{conjecture}{Conjecture}
\sharedthm{proposition}{Proposition}
\sharedthm{lemma}{Lemma}
\sharedthm{choice}{Choice}
\sharedthm{theorem}{Theorem}
\sharedthm{corollary}{Corollary}
\sharedthm{claim}{Claim}

\theoremstyle{remark}
\sharedthm{remark}{Remark}
\sharedthm{aside}{Aside}
\sharedthm{example}{Example}
\sharedthm{warning}{Warning}

\crefname{condition}{Condition}{Conditions}
\crefname{problem}{Problem}{Problems}
\crefname{situation}{Situation}{Situations}
\crefname{construction}{Construction}{Constructions}
\crefname{question}{Question}{Questions}
\crefname{conjecture}{Conjecture}{Conjectures}
\crefname{choice}{Choice}{Choices}
\crefname{claim}{Claim}{Claims}
\crefname{aside}{Aside}{Asides}
\crefname{warning}{Warning}{Warnings}

\usepackage{letltxmacro}
\LetLtxMacro{\phiorig}{\phi}
\renewcommand{\phi}{\varphi}

\newcommand{\bas}{\mathsf{bas}}

\newcommand{\ZZ}{{\mathbb Z}}

\newcommand{\RR}{{\mathbb R}}

\newcounter{nootje}
\usepackage{comment}

\newcommand*{\pb}{\ar[dr, phantom, very near start, "{\ulcorner}"]}

\newcommand*{\lpb}{\ar[dr, phantom, very near start, "{\ulcorner \ell}"]}
\newcommand*{\lpbstrict}{\ar[dr, phantom, very near start, "{\ulcorner \ell s}"]}

\renewcommand{\tilde}[1]{\widetilde{#1}}
\renewcommand{\hat}[1]{\widehat{#1}}

\newcommand{\ccx}[1]{{\bb L}_{#1}}

\usepackage{xparse}
\NewDocumentCommand{\bl}{ O{\lambda} m }{{{\left({#2}\right)}^{\sim}_{#1}}}

\tikzset{
  symbol/.style={
    draw=none,
    every to/.append style={
      edge node={node [sloped, allow upside down, auto=false]{$#1$}}}
  }
}

\newcommand{\ev}{\mathsf{ev}}
\newcommand{\vir}{\mathsf{vir}}
\DeclareMathOperator{\Log}{Log}

\DeclareMathOperator{\CH}{CH}
\DeclareMathOperator{\LogCH}{LogCH}
\DeclareMathOperator{\LogGW}{LogGW}

\DeclareMathOperator{\DR}{DR}
\DeclareMathOperator{\LogDR}{LogDR}
\DeclareMathOperator{\LogDDR}{LogDDR}
\DeclareMathOperator{\LogDRL}{LogDRL}
\DeclareMathOperator{\LogDDRL}{LogDDRL}

\renewcommand{\angle}[1]{\hspace{-2pt}\left\langle #1 \right\rangle}

\DeclareMathOperator{\Sym}{Sym}
\DeclareMathOperator{\Bl}{Bl}

\newcommand{\colim}{\operatornamewithlimits{colim}}

\newcommand{\Mbar}{\overline{\M}}
\newcommand{\Mfrak}{\mathfrak{M}}

\newcommand{\isom}{\stackrel{\sim}{\longrightarrow}}

\newcommand{\cat}[1]{\bd{#1}}
\renewcommand{\log}{{\mathsf {log}}}

\newcommand{\trop}{{\mathsf {trop}}}

\newcommand{\NN}{\mathbb N}

\newcommand{\M}{{\mathsf {M}}}
\newcommand{\gp}{{\mathsf {gp}}}
\newcommand{\g}[1]{#1^{\mathsf {gp}}}
\newcommand{\punc}{{\mathsf {punc}}}
\newcommand{\st}{{\mathsf {st}}}
\newcommand{\gl}{{\mathsf{gl}}}

\newcommand{\sPP}{{\mathsf{sPP}}}
\newcommand{\ghost}{\overline{{\mathsf {M}}}}

\newcommand{\Mpt}{\bb M}
\newcommand{\Mptst}{\Mpt^\st}

\newcommand{\tensor}{\otimes}

\newcommand{\LogSch}{\operatorname{LogSch}}

\DeclareMathOperator{\PP}{PP}

\newcommand{\A}{\mathbb{A}}

\newcommand{\E}{\mathbb{E}}
\newcommand{\F}{\mathbb{F}}
\newcommand{\G}{\mathbb{G}}

\newcommand{\LL}{\mathbb{L}}
\newcommand{\N}{\mathbb{N}}
\renewcommand{\P}{\mathbb{P}}
\newcommand{\Q}{\mathbb{Q}}
\newcommand{\R}{\mathbb{R}}

\newcommand{\Z}{\mathbb{Z}}
\newcommand{\Acal}{\mathcal{A}}
\newcommand{\Bcal}{\mathcal{B}}

\newcommand{\Mcal}{\mathcal{M}}

\newcommand{\Ocal}{\mathcal{O}}

\newcommand{\Rb}{\mathbf{R}}

\newcommand*{\GG}{\mathbb G}

\newcommand{\ol}{\overline}

\newcommand{\pt}{\{*\}}

\newcommand{\prodvfc}{\mathsf{prod}}

\newcommand{\CO}{\mathsf{CO}}
\newcommand{\Disc}{\mathsf{Disc}}

\newcommand{\Mptm}{\Mpt^\mu}
\newcommand{\Mptmt}{\Mpt^{\mu, t}}
\usetikzlibrary{calc}

\author{Leo Herr}
\author{David Holmes}
\author{Pim Spelier}

\date{\today}

\begin{document}
\begin{abstract}
We define a new theory of logarithmic Gromov--Witten invariants allowing negative contact orders using pierced logarithmic curves. We prove that these classes satisfy logarithmic birational invariance and satisfy loop and splitting theorems generalising those for classical Gromov--Witten invariants. 
\end{abstract}
 
\maketitle

\tableofcontents

\section{Introduction}
\subsection{Background}
Gromov--Witten invariants count curves in a target algebraic variety; for example, there are 620 rational curves of degree 4 in $\bb P^2$ through 11 points in general position. Log Gromov--Witten invariants are a generalisation which count curves with tangency conditions; for example there are 100 rational curves of degree 3 in $\bb P^2$ tangent to two fixed lines and passing through 6 points in general position \cite{ernstrom1996recursive}. More interesting than the exact numbers are the many structural properties satisfied by these invariants.

Gromov--Witten invariants form a fundamental recursive algebraic structure called a cohomological field theory \cite{Kontsevich1994Gromov-Witten-c,Behrend1997GromovWitten,Pandharipande2018CohFTCalculations}. For example, these play a key role in Givental--Teleman reconstruction \cite{giventalreconstruction,Teleman2012} which, for a fixed target, relate the Gromov--Witten invariants of low genus curves to the invariants of high-genus curves via the \emph{splitting} and \emph{loop} axioms written out below. 

Traditional log Gromov-Witten invariants (including those of \cite{BNR}, the most refined invariants available prior to the present work) do not satisfy the splitting and loop axioms. They do however satisfy a different set of structural properties, known as \emph{degeneration formulae}: one can deform a complicated target until it breaks up into simpler pieces, and then reconstruct the Gromov--Witten invariants of the complicated target from the log Gromov--Witten invariants of the simple pieces \cite{Li2001Stable-morphism,Li2002A-degeneration-,ranganathan2019logarithmic,MaulikRanganathan}.

The main goal of this work is to develop a new theory of log Gromov--Witten invariants, which do satisfy the splitting and loop axioms, making it possible to reconstruct high-genus log Gromov-Witten invariants from low-genus ones, and with the eventual goal of allowing us to combine these techniques with degeneration formulae. 

The splitting and loop axioms in classical Gromov--Witten theory arise from the study of gluing maps for moduli spaces of prestable curves: 
\begin{equation}\label{eq:gl}
\gl\colon \frak M_{g_1, n_1 + 1} \times \frak M_{g_2, n_2 + 1} \to \frak M_{g_1 + g_2, n_1 + n_2} \text{ and }\gl\colon \frak M_{g, n+2} \to \frak M_{g+1, n}, 
\end{equation}
where the first ``splitting'' map glues together the last markings on the curves of genera $g_1$ and $g_2$, and the second ``loop'' map glues together the last two markings on the curve of genus $g$. Suppressing decorations, these maps fit into commutative diagrams 
\begin{equation}\label{dia:glueing}
	\begin{tikzcd}[column sep = tiny]
		& {\Mbar_{g+1,n}(X)}\ar[dl] & D  \ar[l]\ar[rr]\ar[dl]\ar[d]&&  \Mbar_{g,n+2}(X) \ar[d]\\
		\Mfrak_{g+1, n} &  \Mfrak_{g, n + 2}\ar[l, "\gl"] & X\ar[rr, "\Delta"] && {X \times X}
	\end{tikzcd}
\end{equation}
(and similarly for the splitting map) where the right vertical map is given by evaluation at the last two markings, and both squares are pullbacks. The splitting `axiom' (really a theorem) relates the pullbacks of the virtual classes in this diagram along the maps $\gl$ and $\Delta_X$:
\[\gl^![\Mbar_{g_1 + g_2, n_1 + n_2}(X)]^{\vir} = \Delta^!\left([\Mbar_{g_1, n_1 + 1}(X)]^{\vir} \boxtimes [\Mbar_{g_2, n_2 + 1}(X)]^{\vir}\right), \]
and the loop axiom is analogous:
\[\gl^![\Mbar_{g+1, n}(X)]^{\vir} = \Delta^![\Mbar_{g, n + 2}(X)]^{\vir}. \]
A basic obstruction to generalising these axioms to log Gromov--Witten theory is that in the latter situation every leg has a contact order with the boundary of $X$, and if two legs are to be glued into an edge then their contact orders must sum to zero; in particular, unless both of them are trivial, we must allow for the possibility of negative contact orders. 

Negative contact order (or negative tangency) to a boundary divisor does not make sense in classical algebraic geometry, but is possible in principle in log geometry. It means that we must allow certain elements to lie in the groupification of a monoid instead of in the monoid itself \cite[\S 1.1]{Abramovich2020Punctured-logar}. On the other hand, defining log Gromov--Witten invariants in a way which allows negative contact orders, and proving splitting and loop axioms for these invariants, is not straightforward. 

Spaces of relative and logarithmic stable maps with negative contact orders were constructed in \cite{structuresingenuszerorelativeGWtheory} and \cite{Abramovich2020Punctured-logar}, and the latter defines a Gromov--Witten correspondence which takes a tropical virtual class to an algebraic one. Spaces of tropical stable maps with negative contact orders are generally highly singular and non-equidimensional, so defining a tropical virtual class (and hence an algebraic virtual class) is not straightforward; but it is unavoidable if one aims to define log Gromov-Witten invariants. 

One approach (adopted in \cite{Abramovich2020Punctured-logar}) is to use tropical data to pick out irreducible components of the tropical moduli space, and to use the fundamental class of such a component as a virtual fundamental class. Another definition was proposed in \cite{BNR} building on idea from \cite{fan2020structures}, settling the question of how logarithmic and orbifold invariants are related in genus 0, but which does not satisfy splitting or loop axioms, nor the birational invariance satisfied by classical log Gromov--Witten invariants \cite{Abramovich2018Birational-inva} (except in some special cases; see \cite{johnston2026birational}).

\subsection{Pierced log curves}

The difficulties faced by both these theories is most easily seen on the tropical level. Suppose one wants to glue two legs (markings) of a tropical curve (in essence, a metric graph) together to produce an edge of a higher genus tropical curve. This new edge will have some finite length, and the data of this length must come from somewhere. It may be supplied as some auxiliary data (as in \cite{Huszar2019Clutching-and-gluing}), or it may be determined by requiring that the legs of the original tropical curve have finite length (these lengths are then summed to yield the length of the new edge). This latter approach is the one adopted in \cite{Holmes2023LogarithmicCohomologicalFT}, where two of us introduced \emph{pierced} log curves. These are a variant of Kato's log curves \cite{Kato2000Log-smooth-defo} whose underlying curve is still a prestable curve. They have the same log structure at nodes and at unmarked points, but a different log structure at the markings which in particular assigns a length to each leg; see \ref{def:piercedcurve} for details. In particular, we have glueing maps for pierced curves:
\begin{equation}\label{eq:gl_pie}
\gl\colon \Mpt_{g_1, n_1 + 1} \times \Mpt_{g_2, n_2 + 1} \to \Mpt_{g_1 + g_2, n_1 + n_2} \text{ and }\gl\colon \Mpt_{g, n+2} \to \Mpt_{g+1, n}, 
\end{equation}

We use our theory of pierced curves to define a stack of pierced maps $\Mpt_\Lambda(X)$ to a log target $X$, as well as a tropical analogue $\Mpt^t_\Lambda(X)$. Here $\Lambda$ is a collection of data which captures the genus, number of markings, contact orders, and curve class. 

The marked points of a pierced log curve are actual log morphisms, unlike for Kato's log curves. This means that the `evaluation at a marked point' maps $\ev_i\colon \Mbar_{g,n}(X) \to X$ lift to the stack of pierced maps (as do their factorisations via the inclusions of suitable strata of $X$), yielding maps 
$\ev_i\colon \Mpt_\Lambda(X) \to X$, and allowing us to avoid working with the evaluation stacks of \cite{Abramovich2020Punctured-logar}. 

Putting this together, the key diagram \ref{dia:glueing} makes sense for pierced log curves; this will allow for splitting and loop axioms as described below. 

\subsection{Introducing pierced Gromov--Witten invariants}

Our first main result is that the space $\Mpt_{\Lambda}(X)$ of stable pierced log maps is proper. This is essential for defining Gromov--Witten invariants, since it allows for pushing forward to the moduli stack of stable curves, and more generally for integration. 

\begin{proposition}[\ref{prop:mpieralgebraicproper}]
The moduli space $\Mpt_{\Lambda}(X)$ is an algebraic stack with log structure. The moduli space $\Mpt_{\Lambda}(X)$ is proper if $X$ is log smooth and projective.
\end{proposition}
The proof runs essentially by reduction to the analogous result for punctured maps in \cite{Abramovich2020Punctured-logar}. 

Our next main task is to equip this space with a virtual fundamental class. A virtual pullback for the map $\Mpt_{\Lambda}(X)\to \Mpt_{\Lambda}^t(X^t)$ to the corresponding tropical moduli space\footnote{Throughout the paper, we use a superscript ${}^t$ to denote tropicalisation; see \ref{sec:log_targets}. } is constructed following ideas of \cite{Abramovich2020Punctured-logar}. It then remains to construct a virtual class on the tropical moduli space $\Mpt_{\Lambda}^t(X^t)$, which is not completely straightforward since this space is neither smooth nor equidimensional. Inspired by \cite{BNR}, we construct a class by writing down a homological piecewise polynomial function on $\Mpt_{\Lambda}^t(X^t)$
in the sense of \cite{RPSS_log_taut}, which yields an element of Chow homology. The precise definition of our polynomial is in \ref{subsec:tropvfc}, but the essential idea is that the polynomial measures the distance from the \emph{tip} of a leg to the boundary of the cone in whose interior it lies; this is possible with pierced curves as legs have finite length. 

With these tools in hand, we can define pierced log Gromov--Witten invariants as
\[
\int_{[\Mpt_\Lambda(X)]^\vir} \prod_{i=1}^n \ev_i^* \gamma_i
\]
where $\gamma_i$ are classes on (strata of blowups of) $X$; descendant invariants are defined similarly.

\subsection{Main results}

With this new virtual class, we are able to prove a birational invariance statement analogous to that proven for classical log Gromov--Witten invariants by Abramovich and Wise \cite{Abramovich2018Birational-inva}. Such birational invariance fails in general for the class constructed in \cite{BNR}, though it does hold in some special situations, see \cite{johnston2026birational}.

\begin{theorem}[\ref{thm:log_bir_inv}] Let $X$ be a target as in \ref{sec:log_targets} and let $\tilde X \to X$ be a log blowup with both $X$ and $\tilde X$ smooth and log smooth. There are finitely many discrete data $\tilde \Lambda$ for $\tilde X$ which are balanced and compatible with given discrete data $\Lambda$ for $X$, and we have 
\begin{equation}
\sum_{\tilde \Lambda}f_*[\Mpt_{\tilde{\Lambda}}(\tilde{X})]^\vir = [\Mpt_{\Lambda}(X)]^\vir. 
\end{equation}
\end{theorem}
The proof of this result is essentially tropical, and depends on a careful comparison of the balancing conditions for $X$ and $\tilde X$. \Cref{ex:birational} shows that this sum is necessary; the discrete data for $X$ really can break into pieces when pulled back to $\tilde X$, and several different discrete data for $\tilde X$ can all make non-zero contributions.

Next we move on to the splitting and loop axioms, which are the technical heart of this work. We fix a log target $X$, which is again assumed smooth and log smooth. We fix a gluing map as in \ref{eq:gl_pie}, and we fix discrete data $\Lambda$ for the target of the gluing map. There is then a finite set $\gl^*\Lambda$ of discrete data $\Xi$ for the source of the gluing map which are compatible with $\Lambda$. We work throughout in logarithmic intersection theory and Chow groups, see \ref{sec:log_intersection_theory} for details. Our first result is
\begin{theorem}[\ref{thm:loggluingclean}]
For $\gl$ either the splitting or loop gluing map we have 
\[s_+s_- \gl^! [\Mpt_{\Lambda}(X)]^\vir = \sum_{\Xi \in \gl^* \Lambda} c_+^{\Xi} c_-^{\Xi}\Delta_X^! [\Mpt_{\Xi}(X)]^\vir\]
in log Chow homology. 
\end{theorem}
Here $\Delta_X\colon X \to X \times X$ is the diagonal of $X$, $c_\pm^\Xi$ are certain integers computed from the tropical data, and $s_\pm$ are piecewise linear functions on the moduli space, which for each fixed $\Xi$ are pulled back from $X$. The space $\Mpt_{\Xi}(X)$ is either a space of stable maps on lower genus curves (in the non-separating case), or a product of two such spaces (in the separating case). To obtain this result we first prove a combinatorial analogue (\ref{thm:tropicalgluingsep}) where the precise structure of our tropical virtual class plays a central role. We then lift this combinatorial statement to a logarithmic one using a new theory of logarithmic Gysin pullbacks developed in \ref{sec:log_vir_pull}. 

The classes in logarithmic Chow homology induced by $s_+$ and $s_-$ may be nilpotent, so this gluing formula does not always determine $\gl^! [\Mpt_{\Lambda}(X)]^\vir$ from the lower genus invariants. Our final two results allow us to rectify this in certain situations. First, if $X$ admits a suitable torus action (for example, if $X$ is toric) then we can lift $s_+$ and $s_-$ to an equivariant Chow ring in which they are not zero divisors. 

\begin{theorem}[\ref{thm:loggluingequiv}]
Suppose that $X$ admits a suitable action of a torus $T$. Then 
\[\gl^! [\Mpt_{\Lambda}(X)]^\vir = Q(\sum_{\Xi \in \gl^* \Lambda} \frac{1}{s_+s_- }c_1^{\Xi} c_2^{\Xi}\Delta_X^{T,!} [\Mpt_{\Xi}(X)/T]^\vir)\]
in $\bigoplus_{\Xi \in \gl^* \Lambda} \LogCH(\Mpt_{\Lambda | \Xi}(X))$, 
where 
\[Q\colon \LogCH(\Mpt_{\Lambda | \Xi}(X)/T) \to \LogCH(\Mpt_{\Lambda | \Xi}(X))\] 
is the pullback map. 
\end{theorem}
In the special case where $X = \P^1$, this calculation is the starting point of \cite{Spelier2025SplittingFormulaLogDR}, where the non-separating gluing pullback $\gl^* \LogDR_g(A)$ of the logarithmic double ramification cycle is computed explicitly.

Finally, in the separating case, we show that one can find explicit formulas in Chow groups of blowups. The proof combines techniques from the birational invariance result, with the gluing statement.

\begin{theorem}[\ref{thm:loggluingseparating}]
There is a log alteration $\hat{X} \to X$ for which we have \begin{align*}\gl^! [\Mptm_\Lambda(X)]^\vir &= \sum_{\Xi \in \gl^* \Lambda} \sum_{i=1}^{m_\Xi} c^{\Xi;i}_+ c^{\Xi;i}_- \pi_{\Xi;i,*}\Delta_{E_{\Xi;i}}^! [\Mpt_{\Xi;i}]^{\vir} \\ 
&= \sum_{\Xi \in \gl^* \Lambda} \sum_{i=1}^{m_\Xi} c^{\Xi;i}_+c^{\Xi;i}_- \pi_{\Xi;i,*}\Delta_{E_{\Xi;i}}^! \sum_{\Xi' \in \pi^* (\Xi;i)} \pi_{\Xi',*} [\Mptm_{\Xi'}(\hat X)]^\vir.\end{align*}
inside $\CH_*(\gl^*\Mptm_\Lambda(X))$. Here $\pi_{\Xi;i}, \pi_{\Xi'}$ are certain log blowups defined in \ref{sec:glue:logsep}, and $E_{\Xi;i}$ are strata inside $\hat{X}$, and $c_\pm^{\Xi;i} \in \Z_{>0}$.
\end{theorem}

This also leads to a relation between log Gromov--Witten classes (see \ref{def:logGWinvariants}). So far we have worked with Chow groups and intersection theory, but we could just as well have worked with Borel--Moore homology throughout; or one can simply apply the cycle class map to both sides of the above theorems. Note that a Gysin map for Borel--Moore homology associated to any perfect obstruction theory is constructed in \cite{KapranovVasserot}. For the next result we assume the existence of Kunneth decompositions the diagonals of strata; this will occasionally hold in Chow, but much more often in Borel--Moore homology (which in this setting coincides with singular homology).

\begin{theorem}[\ref{thm:loggluingkunneth}]
For each $E_{\Xi;i}$ write the K\"unneth decomposition of its diagonal as $\Delta_{E_{\Xi;i}} = \sum_{k} \delta_{\Xi;i;k+} \boxtimes \delta_{\Xi;i;k;-}$. Then we have the equality
\begin{multline*}
    \gl^! \LogGW(X;\Lambda, \gamma) =\\  \sum_{\Xi \in \gl^* \Lambda} \sum_{i=1}^{m_\Xi}  c_+^{\Xi;i} c_-^{\Xi;i} \sum_k\sum_{\Xi' \in \pi^*(\Xi;i)} \LogGW(\hat{X};\Xi_+', \gamma_+ \cup \delta_{\Xi;i;k+}) \boxtimes \\\LogGW(\hat{X};\Xi_-', \gamma_- \cup \delta_{\Xi;i;k-}).
\end{multline*}
\end{theorem}

We give an example of this theorem in \ref{sec:example_calculation}, where we compute the gluing pullback of a $g = 2, n = 4$ log Gromov--Witten invariant of $X = \P^2$ along a separating gluing map. We choose discrete data $\Lambda = (g = 2, n = 4, \beta = kH, (a_i)_{i=1}^4)$ with contact orders $a = ((k,0,0),(0,k,0),(0,0,k),(0,0,0))$ and insertions
\[
\gamma=\bigl([\mathrm{pt}]_{D_0},\,[\mathrm{pt}]_{D_1},\,1,\,[\mathrm{pt}]_{\bb P^2}\bigr).
\]
This is a classical logarithmic Gromov--Witten invariant, but evaluating the formula for $\gl^! \LogGW(X;\Lambda, \gamma)$ requires computing pierced invariants. We arrive at the following result.
\begin{proposition}[\ref{prop:theexample}]
The degree of the dimension $0$ class $\gl^! \LogGW(X;\Lambda,\gamma)$ is
\[
	\frac{-k^3}{24}\binom{k+2}{5}.
\]
\end{proposition}

\subsection{Comparison to previous constructions}
In \ref{sec:punctured_comparison}, we construct a map from our space of pierced maps to the space of punctured maps from \cite{Abramovich2020Punctured-logar} and prove that it is proper. In fact, the space of pierced maps is a closed substack of a log blowup of the space of punctured maps (with additional log structure -- see \ref{thm:closedinblowup}).

In \ref{subsec:comppunctinvariants}, we compare our pierced Gromov--Witten invariants to the \emph{refined punctured invariants} constructed in \cite{BNR}. Recall that our tropical virtual class is defined by a certain homological piecewise-polynomial function, leading to an algebraic virtual class $[\Mpt(X)]^\vir$. In this section, we write down a different homological piecewise-polynomial function, producing a different virtual class which we denote $[\Mpt(X)]^\prodvfc$, and show in \ref{prop:comparisonbnr} that $[\Mpt(X)]^\prodvfc$ is equal (up to certain integer factors) to the class constructed in \cite{BNR}, explaining how to realise their class in our framework. Note that $[\Mpt(X)]^\prodvfc$ does not satisfy birational invariance, nor the loop and splitting theorems. On the other hand, $[\Mpt(X)]^\prodvfc$ can be obtained from $[\Mpt(X)]^\vir$ by capping with a suitable class (see \ref{prop:trop_prod_vir_divisibility}), hence $[\Mpt(X)]^\vir$ refines the refined virtual classes of \cite{BNR}.

We remark that if all contact orders are non-negative then the refined punctured virtual class of \cite{BNR} and our virtual classes $[\Mpt(X)]^\vir$ and $[\Mpt(X)]^\prodvfc$ all coincide, and recover the classical logarithmic Gromov--Witten invariants with non-negative tangencies of \cite{Abramovich2014Stable-logaritmic-maps-II,Gross2013Logarithmic-gro}.

Splitting for punctured Gromov--Witten invariants is discussed in \cite[\S 5]{Abramovich2020Punctured-logar} and in \cite{Wu_splitting}. However, in their theory no glueing maps are available, so it is not possible to prove analogues of our splitting theorem. Instead, they study essentially the punctured analogue of the right-hand square in the diagram \ref{dia:glueing} (where evaluation maps take values in an appropriate evaluation stack), making the punctured analogue of the fibre product `$D$' very explicit (including on the virtual level). This is very useful for understanding how the invariants of $X$ relate to the invariants of its strata, but to us do not appear applicable to the problem of relating high- and low-genus invariants. 

In the case of a smooth pair $(X,D)$, the failure of the loop axiom for the invariants of \cite{BNR} (equivalently, by
\ref{prop:comparisonbnr}, for the class $[\Mpt(X)]^\prodvfc$) has been studied
from the orbifold side by You \cite{You2021Gromov-Witten-i}. The
relative invariants with negative contact orders of
\cite{fan2020structures,Fan2019Higher-genus-re} are extracted from the orbifold
Gromov--Witten invariants of the root stacks $X_{D,r}$ for $r$ large: a marking of
\emph{small} age $k/r$ (with $k$ bounded) corresponds to positive contact order
$k$, and a marking of \emph{large} age $k/r$ (with $r - k$ bounded) to negative
contact order $-(r-k)$. Fan, Wu and You show that the resulting theory is a
partial cohomological field theory, satisfying the splitting axiom but not the
loop axiom. You constructs a modified loop axiom in which an explicit correction term built from
mid-age invariants is added to the usual right hand side.

\subsection{Acknowledgements}

We are grateful to Dan Abramovich, Luca Battistella, Navid Nabijou, Dhruv Ranganathan, and Samuel Johnston for helpful discussions during the preparation of this paper.
We are particularly grateful to Sam Molcho for sharing a preliminary version of his ongoing work on logarithmic intersection theory, and to Dhruv Ranganathan for many helpful comments on a preliminary version of this paper.

Research of L.H. was supported by grant VI.Vidi.193.006 of the Dutch Research Council (NWO). D.H. was funded by the European Union (ERC, EAGL, 101169685) and by grant VI.Vidi.193.006 of the Dutch Research Council (NWO). Views and opinions expressed are however those of the authors only and do not necessarily reflect those of the European Union or the European Research Council. Neither the European Union nor the granting authority can be held responsible for them. 

AI (Claude Opus, ChatGPT) was used to generate the tikz code for some figures and for proof-reading. It found citations to the literature in \ref{lem:Psimapstackypointcase} in \ref{sec:push_hpp} to replace a human-written argument, but everything in the current proof was written by the authors. In \ref{sec:example_calculation}, it searched small $g,n$ to find cases amenable for computation and computed the example for $k=3$.

\section{Background and conventions}
\subsection{Log structures}
We work with log structures in the sense of Fontaine-Illusie-Kato; see \cite{Ogus2018Lectures-on-log}. In particular a log scheme $X$ carries a sheaf of monoids $\M_X$ with a map $\alpha\colon \M_X \to \ca O_X$, and the quotient $\M_X/\ca O_X^\times$ is denoted $\ghost_X$. We work always with integral saturated log schemes. If $C/S$ is a pierced curve (see \ref{def:piercedcurve}) we will assume that $S$ is fs (i.e.\ admits local charts by integral \emph{finitely generated} monoids), but we do not assume this for $C$.

\subsection{Logarithmic targets}
\label{sec:log_targets}

The target for our Gromov--Witten theory will be an algebraic stack $X$ with log structure, log smooth over the spectrum of a fixed ground field $k$ with trivial log structure. 
We spell out our hypotheses on $X$ for each section, but the reader may take $X$ to be a smooth, log smooth, projective, geometrically connected log scheme.

We fix once and for all a tropicalisation $X \to X^t$ of $X$: this is a strict map to an Artin fan or equivalently to a cone complex, and is automatically smooth since $X$ is log smooth. We do \emph{not} allow $X^t$ to be a cone space or cone stack; in the case where $X$ has normal crossings, this corresponds to them being strict normal crossings. This is not a fundamental restriction, but does simplify some definitions. If we want to emphasise the Artin fan we write $\ca A_X$, if we want to emphasise the cone complex we write $\Sigma_X$. 

In \ref{sec:punctured_comparison,subsec:comppunctinvariants,sec:bir_inv,sec:gluing} we assume $X$ is locally of finite type and separated.

From \ref{subsec:tropvfc} onwards, we will assume that $X \to X^t$ is surjective with connected fibres, and that $X$ is quasi-compact. This implies that $X^t$ has finitely many cones.

In \ref{subsec:comppunctinvariants} where we make the comparison to \cite{Abramovich2020Punctured-logar,BNR}, we assume moreover that $X$ is proper and smooth (so normal crossings), as that is the context in which their constructions are given.
We assume in \ref{sec:glue:log} that, if an irreducible curve mapping to $X$ has vanishing class in homology, it is contracted. This is automatic if $X$  $\bb Q$-factorial or projective, since in either case we can find a Cartier divisor on $X$ which meets $C$ at finitely many points (in the $\bb Q$-factorial case we find such a divisor on an affine patch and then take Zariski closure and a suitable positive multiple; in the projective case we use an ample divisor).

\subsection{Logarithmic intersection theory}\label{sec:log_intersection_theory}

If $X$ is an algebraic stack locally of finite type over a field, we write $\CH_*(X)$ for the Chow group as in \cite{Kresch1999Cycle-groups-fo}, but taken always with \emph{rational} coefficients. We write $\CH^*(X)$ for the corresponding Chow cohomology ring as in \cite{BSS-I}. 

If $X$ is an algebraic stack with log structure, locally of finite type over a field, we write $\LogCH_*(X)$ for the colimit of the groups $\CH_*(\tilde X)$ as $\tilde X$ runs over a cofinal system of locally free log alterations of $X$. Beware that this group is sometimes denoted ``$\LogCH^*(X)$'' in the literature. We use $\LogCH^*(X)$ instead to refer to the resulting bivariant theory \cite{Barrott2019Logarithmic-Cho,sammolchointersectiontheorypreprint}. 

We describe the transition maps in this colimit: 
If $\tilde X_1 \to \tilde X_2$ are sufficiently fine locally free log alterations, then there is a pullback square \cite[Appendix B]{drink}
\[
\begin{tikzcd}
    \tilde X_1 \ar[r] \ar[d] & \ca A_{\tilde X_1}\ar[d]\\
    \tilde X_2 \ar[r] & \ca A_{\tilde X_2} 
\end{tikzcd}
\]
for some Artin fans $\ca A_{\tilde X_i}$. 
Local freeness implies that the Artin fans $\ca A_{\tilde X_i}$ are smooth. The map between them is also of DM type because it is locally a log alteration, so it is lci and yields a Gysin pullback $\CH_*(\tilde X_2 ) \to \CH_*(\tilde X_1)$; see \cite{Barrott2019Logarithmic-Cho} and \ref{sec:log_vir_pull}. 

In \ref{sec:log_vir_pull} we define a general notion of log perfect obstruction theories and log virtual pullbacks, which will be used throughout.

\subsection{Cone stacks with boundary}
In this paper we work with tropicalisations of log stacks. For a log smooth stack $Y$ one often takes the tropicalisation to be its Artin fan $\Acal_Y$ \cite{Abramovich2016Skeletons-and-f}, which we equivalently think of as a cone stack $\Sigma_Y$ \cite{Cavalieri2020A-Moduli-Stack} by the equivalence of categories between Artin fans and cone stacks. For more general stacks, such as idealised log smooth stacks (e.g., a stratum in a log smooth stack) the notion of \emph{idealised} Artin fan and cone stack \emph{with boundary}, as defined in \cite[Definitions 41, 46]{RPSS_log_taut} is more useful. These correspond to (unions of) strata in Artin fans/cone stacks.

We first recall an example of classical cone stacks and Artin fans.

\begin{example}
Recall that a classical tropical curve is a metrised graph of genus $g$, with $n$ infinite legs. The moduli space of stable tropical curves forms a cone stack $\Mbar_{g,n}^t$. Explicitly, it is the colimit $\colim_{\Gamma} \R_{\geq 0}^{E(\Gamma)}$. The colimit is taken over all stable decorated graphs of genus $g$ with $n$ markings, with the maps being contractions (including graph automorphisms). This is the tropicalisation of the moduli space of curves $\Mbar_{g,n}$. The Artin fan corresponding to this cone stack is the colimit $\colim_{\Gamma} [\A^1/\G_m]^{E(\Gamma)}$.
\end{example}

We give the definition of cone stacks with boundary and idealised Artin fans from \cite{RPSS_log_taut}. 

\begin{definition}
A \emph{cone stack with boundary} is a pair $(\Sigma,\Delta)$ of a cone stack $\Sigma$ and a subcone stack $\Delta$, the boundary. We write $\Sigma^\circ = \Sigma \setminus \Delta$ for the interior. 

An \emph{idealised Artin fan} is a reduced closed embedding $\Bcal \subset \Acal$ inside an Artin fan.
\end{definition}

\begin{proposition}
The category of cone stacks with boundary is equivalent to the category of idealised Artin fans, by sending a pair $(\Sigma,\Delta)$ to $\Acal_{\Sigma} \setminus \Acal_{\Delta} \subset \Acal_{\Sigma}$.
\end{proposition}

\begin{example}
The cone stack with boundary $(\R_{\geq 0}, 0)$ may be denoted as $\R_{> 0}$. On the Artin fan side, this corresponds to the closed embedding $B\G_m \subset [\A^1/\G_m]$.
\end{example}

The main example of cone stacks with boundary will be the moduli space of tropical pierced maps constructed in \ref{sec:tropicalmaps}. Pierced curves are introduced in more detail in \ref{subsec:pierced_curves}. Tropically, a pierced curve corresponds to a classical tropical curve decorated with non-zero lengths for each leg. As cone stacks with boundary, the moduli space of tropical pierced curves $\Mpt_{g,n}^t$ is $\Mfrak_{g,n}^t \times \R_{>0}^n$.

\subsection{Homological piecewise polynomial functions}\label{sec:hPP}

Classically, the tropicalisation $\Acal_X$ of a log smooth scheme $X$ (e.g., the moduli space of curves) has been a helpful tool for understanding the intersection theory of $X$. Namely, we have a pullback map $\LogCH^*(\Acal_X) \to \LogCH^*(X)$, and $\LogCH^*(\Acal_X)$ is purely combinatorial: it is the ring of piecewise polynomials on $\Sigma_X$, the cone stack corresponding to $\Acal_X$. A similar theory for the homology of cone stacks with boundary is developed in \cite{RPSS_log_taut}.

The key definition is that of homological piecewise polynomials.

\begin{definition}[\protect{\cite[Definition~57]{RPSS_log_taut}}]
\label{def:prrshomPP}
Let $(\Sigma,\Sigma^\circ,\Delta)$ be a cone stack with boundary. We denote the group of (strict) \emph{homological} piecewise polynomials $\PP_*(\Sigma,\Delta)$ as
\begin{align*}
	\sPP_*(\Sigma,\Delta) &= \{f \in \sPP^*(\Sigma) : f|_{\Delta} = 0 \} \subset \sPP^*(\Sigma)\\
	\PP_*(\Sigma,\Delta) &= \{f \in \PP^*(\Sigma) : f|_{\Delta} = 0 \} \subset \PP^*(\Sigma)
\end{align*}
i.e., piecewise polynomials with $\Q$-coefficients vanishing on the boundary $\Delta$. These carry a grading given by the negative degree of the corresponding polynomial.
\end{definition}

This gives exactly a combinatorial description of the homological (log) Chow group of the corresponding idealised Artin fan.

\begin{theorem}[\protect{\cite[Theorem~59]{RPSS_log_taut}, Appendix \oref{sec:push_hpp}}]
\label{thm:prrshomPP}
Let $(\Sigma,\Sigma^\circ,\Delta)$ be a finite cone stack with boundary, and let $\Bcal_{\Sigma^\circ}$ denote the corresponding idealised Artin fan.
We have morphisms of graded groups \begin{align*}\Psi: \sPP_*(\Sigma,\Delta) &\to \CH_*(\Bcal_{\Sigma^\circ}) \\ \Psi: \PP_*(\Sigma,\Delta) &\to \LogCH_*(\Bcal_{\Sigma^\circ}).\end{align*}
The first morphism is injective, and is an isomorphism if $\Sigma$ is locally free; the second morphism is always an isomorphism. 
\end{theorem}

\begin{example}
Consider the cone stack with boundary $\R_{>0}^2$ with coordinates $x,y$, corresponding to the idealised Artin fan $\Bcal = B\G_m^2$. Then $\Psi(xy) = [\Bcal]$, and $\Psi(\min(x,y))$ is the fundamental class of the log blowup $\tilde{\Bcal}$. 
This log blowup is the $\P^1$-bundle over $\Bcal$ representing isomorphisms between its two universal line bundles. 
Equivalently, $\tilde{\Bcal}/\Bcal$ is the quotient by $\G_m^2$ of the map $E \to (\pt,\N^2)$, where $E$ is the exceptional divisor in the blowup $\Bl_{(0,0)} \A^2$. This example is worked out in more detail in \cite[Example~62]{RPSS_log_taut}.
\end{example}

\section{The space of logarithmic pierced maps}

\subsection{Pierced curves}
\label{subsec:pierced_curves}
In \cite{Holmes2023LogarithmicCohomologicalFT}, we defined the moduli space of \emph{pointed curves}, and through those defined \emph{pierced} curves. These pierced curves allow for negative tangencies and gluing, and are the main actors in this paper. In this section, we give a direct definition of the moduli space of pierced curves.

\begin{definition}
\label{def:piercedcurve}
A pierced curve is a tuple $(C/S, p_1, \dots, p_n)$ where $C \to S$ is a morphism of log schemes and $p_i\colon S \to C$ are sections, such that 
\begin{enumerate}
\item The underlying schemes give a marked prestable curve (in particular the $p_i$ go through the smooth locus);
\item
On each geometric fibre, at each $p_i$ there is a (necessarily unique)
element $t_i \in \ghost_{C, p_i}$ whose preimage in $\M_{C, p_i}$ maps to a uniformiser in $\ca O_{C, p_i}$; 
\item The induced map $\ghost_S \oplus \bb N \to \ghost_{C, p_i}; (0,1) \mapsto t_i$ becomes an isomorphism after groupification; we call the second projection the \emph{slope} of an element of $\ghost_{C, p_i}^\gp$. The \emph{length} of the marking is $\ell_i \coloneqq p_i^*t_i \in \ghost_S$. 
\item After replacing $\M_{C, t_i}$ by the submonoid of elements of non-negative slope, we have a log curve (and the sections are still log maps). 
\item All elements in $\M_{C, t_i}$ that have negative slope map to $0 \in \ca O_C$;
\item \label{enum:item:maxext} $\ghost_{C, p_i} = \{(a, n)\in \ghost_S \oplus \bb Z : \alpha(a + n\ell_i)= 0 \in \ca O_S/\ca O_S^\times \text{ or } n=0\}$. 
\end{enumerate}

We let $\Mpt_{g,n}$ denote the moduli stack of pierced curves.

\end{definition}

\begin{remark}
For $(C/S,p_1,\cdots,p_n)$ a pierced curve, there is an associated pointed curve $(C^\text{pt}/S)$, with monoid given by elements of non-negative slope along the legs. This shows that the moduli stack of pierced curves is isomorphic to the moduli stack of pointed curves.
In particular, it is an algebraic stack with log structure. For maps to a target log scheme $X$, the story differs: the moduli space of pointed maps is empty if there is a negative tangency (since by definition the monoids of the pointed curves only have elements with non-negative slopes).

Condition (\oref{enum:item:maxext}) is equivalent to saying that $\ghost_{C,p_i}$ is the maximal submonoid of $\ghost_{C^\text{pt},p_i}^\gp$ such that the section $p_i$ extends. 
\end{remark}
\begin{warning}
We warn the reader that a pierced curve $C$ is not fs. It is integral and saturated; however, it is not finitely generated, as the condition $(a,m) = (0,0)$ or $a + m\ell > 0$ gives a non-finitely generated monoid; see \ref{ex:puncturingA1}. 
\end{warning}

\begin{example}
\label{ex:puncturingA1}
Let $S$ be a log point and $C/S$ a pierced curve whose underlying curve is smooth, and with a single marked point $p$. Then $C$ is strict over $S$ away from $p$ and we have 
\[\ghost_{C,p} = \{(a,n) \in \ghost_S \oplus \Z : (a,n) = 0 \text{ or } a + n \ell > 0\}\]
with the map $\ghost_{C,p} \to \ghost_S$ uniquely determined by the image $\ell$ of $(0,1)$. In particular $\ghost_{C,p}$ is never finitely generated unless $\ghost_S = 0$. This is depicted in \ref{fig:piercedghostsheaf}.

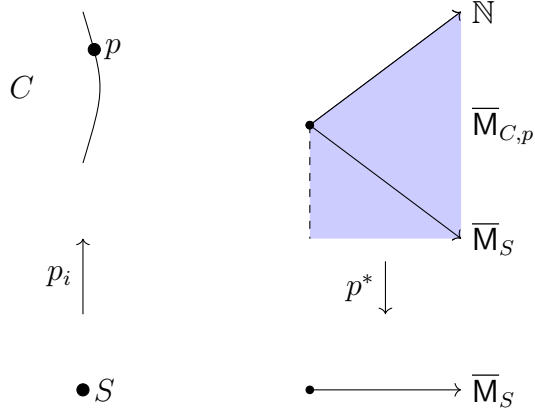
\begin{figure}
	\centering
	\begin{tikzpicture}
	\draw (0, 0) ..controls(.3, 1).. (0, 2);
	\filldraw (.15, 1.5) circle (.08);
	\node[right] at (.15, 1.5){$p$};
	\node[left] at (-.5, 1){$C$};
	\filldraw (0, -3) circle (.08);
	\node[right] at (0, -3){$S$};
	\draw[->] (0, -2) -- node[midway,left] {$p_i$} (0, -1);
	\begin{scope}[shift={(3, 0)}]
	\draw[->] (0, .5) to (2, 2);
	\draw[->] (0, .5) to (2, -1);
	\node[right] at (2, 2){$\NN$};
	\node[right] at (2, -1){$\ghost_S$};
	\filldraw (0, .5) circle (.05);
	\fill[blue!20] (0, .5) to (2, 2) to (2, -1) to (0, -1) to (0, .5);
	\draw[->] (0, .5) to (2, 2);
	\draw[->] (0, .5) to (2, -1);
	\draw[dashed] (0, .5) to (0, -1);
	\node[right] at (2, .5){$\ghost_{C, p}$};
	\filldraw (0, .5) circle (.05);
	\draw[->] (1, -1.3) -- node[midway,left] {$p^*$} (1, -2);
	\draw[->] (0, -3) to (2, -3);
	\filldraw (0, -3) circle (.05);
	\node[right] at (2, -3){$\ghost_S$};
	\end{scope}
	\end{tikzpicture}
	\caption{An example of the log structure on a pierced curve $C$ over a standard log point $S$ at a marked point $p \in C$.
	}
	\label{fig:piercedghostsheaf}
\end{figure}
\end{example}

\subsection{Pierced maps}

Let $X$ be an algebraic stack with log structure. We first define explicitly the discrete data associated to a map from a pierced curve $C \to X$. As in the classical case, this will include the genus, the number of markings, and the curve class, but it will also contain possibly negative \emph{contact orders} (tangencies).

\begin{definition}
\label{def:contact_orders}
A \emph{contact order} for $X$ is a pair $(\sigma, a)$ where $\sigma$ is a cone in the tropicalisation $X^t$ and $a \in \sigma^{\gp}(\bb Z)$, such that there is no proper subcone $\tau \subset \sigma$ with $a \in \tau^\gp(\Z) \subset \sigma^\gp(\Z)$. We denote the set of contact orders by $\CO(X)$.
\end{definition}

For $X$ strict normal crossings with boundary divisors $D_1,\dots,D_k$, there is a natural map $\CO(X) \to \Z^k$, recording the contact order to the $i$'th boundary divisor. The preimage of $a \in \Z^k$ is formed by the cones $\sigma$ corresponding to the connected components in $\bigcap_{i : a_i \ne 0} D_i$. In particular, if any intersection of boundary divisors is connected, then $\CO(X) \to \Z^k$ is injective.

\begin{remark}
The condition in \ref{def:contact_orders} that no proper subcone $\tau$ should exist means for example that if we want to constrain a point $p_i$ to lie in one of the $D_j$ then we can only do so by requiring some non-zero contact order there; we cannot simply require the point to land there and allow the contact order to be arbitrary. This restriction is necessary in order for our virtual classes constructed in \ref{subsec:tropvfc,sec:piercedlogGW} to be well-defined; more precisely, for the piecewise polynomial functions defining the tropical virtual classes to vanish on the boundary.

Of course, one can add in this constraint by multiplying with the insertion $\ev_i^* D_j$.
\end{remark}

\begin{definition}
\label{def:discrete-data}
\emph{Discrete data} for $X$ is a tuple 
 \[\Lambda = (g,n,\beta, (\sigma_i, a_i)_{1 \le i \le n})\] 
where $g$ and $n$ are non-negative integers, $\beta\colon \on{Pic}(X) \to \bb Z$ is an algebraic group homomorphism, and each $(\sigma_i, a_i)$ is a contact order.  We denote the set of discrete data by $\Disc(X)$.
\end{definition}

\begin{definition}\label{def:has_disc_data}
Let $X$ be an algebraic log stack and $\Lambda = (g,n,\beta,(\sigma_i, a_i)_{1 \le i \le n})$ discrete data. Let $(C/S,p_1,\dots,p_n) \in \Mpt_{g,n}$ be a log pierced curve of genus $g$, and $f: C \to X$ a map. We say $f$ has discrete data $\Lambda$ if both of the following conditions hold:
\begin{enumerate}
	\item the map $\on{deg} \circ f^*\colon \on{Pic}(X) \to \bb Z$ is equal to $\beta$;
	\item $f(p_i)$ lies in the closure of the stratum associated to $\sigma_i$, and the composite
\begin{equation}
\ghost_{X, f(p_i)} \xrightarrow{f} \ghost_{C, p_i} \xrightarrow{\mathrm{slope}} \bb Z
\end{equation}
is the given element $a_i$ of $\sigma_i^{\gp}(\bb Z)$. 
\end{enumerate}
\end{definition}

\begin{remark}
For $X$ a smooth scheme, through the embedding $\on{NS}(X) \to H^2(X,\Z)$ obtained from the exponential sequence and the duality between $H^2(X,\Z)$ and $H_2(X,\Z)$, this definition of the curve class $\beta$ is (up to torsion in $H_2(X,\Z)$) the same as the usual definition of $f_*[C] = \beta \in H_2(X,\Z)$. The definition we give has the advantage of being easily compatible with tropicalisation; in fact, in \ref{subsec:balancing} the usual balancing equations for log maps will be framed purely in terms of compatibility of the degree maps $\on{Pic} X \to \bb Z$ and $\on{Pic} X^t \to \bb Z$.
\end{remark}

\begin{definition}
Let $X$ be an algebraic stack with log structure, and let $\Lambda$ be discrete data on $X$. Write
$
	\Mpt_\Lambda(X)
$
for the moduli space parametrising tuples
\[
(C/S,p_1,\dots,p_n, f\colon C \to X) 
\]
where $(C/S, p_1, \dots, p_n)$ is a log pierced curve, and $f$ has discrete data $\Lambda$, and where $\ul{f}: \ul{C} \to \ul{X}$ is stable.
\end{definition}

\begin{definition}
\label{def:evaluationmaps}
Denote $(C,(p_1,\dots,p_n),f: C \to X)/\Mpt_\Lambda(X)$ the universal log pierced map. For $i=1,\dots,n$, let $\ev_i = f \circ p_i : \Mpt_\Lambda(X) \to X$ denote the $i$-th evaluation map.
\end{definition}

\begin{remark}
Unlike in the punctured case, the evaluation maps are logarithmic maps, as the sections of the pierced curve are logarithmic. This difference is one of the key improvements which will later allow for gluing log stable maps and for birational invariance. 
\end{remark}

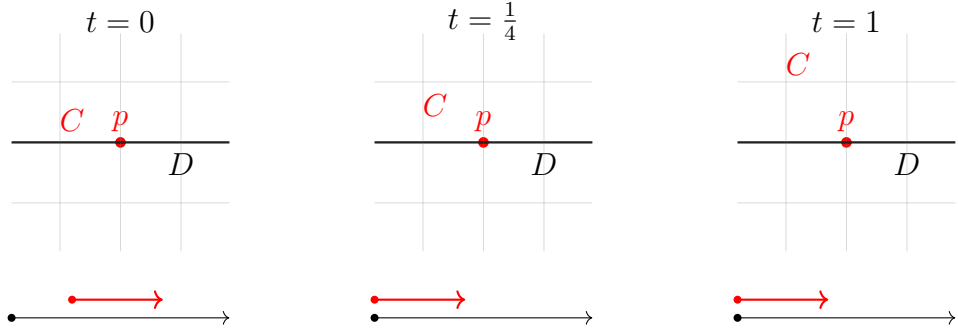
\begin{figure}
\centering
\begin{tikzpicture}[scale=.8]

            \pgfmathsetmacro{\ss}{3.6};

            \pgfmathsetmacro{\httrop}{-2.9};

            \pgfmathsetmacro{\dx}{6};

            \foreach \k in {0,.5,1}{

            \pgfmathsetmacro{\tad}{\k*12};
            \pgfmathsetmacro{\mul}{\k*\k};
            \pgfmathsetmacro{\mint}{min(\ss * .5, sqrt(\ss*.5/(\k*\k+0.001)))};

            \node[circle,fill=red,inner sep=0pt, minimum width=4pt] at (\tad,0) {};

            \node at (\tad,0) [above,red] {$p$};

            \node at (\tad+1,0) [below] {$D$};

            \node at (\tad-.8,\k*\k*.84+.1*\k) [above,red] {$C$};

            \draw[line width = .8] (-\ss*.5 + \tad,0) -- ++(\ss,0);

            \draw[color=red,parametric, domain={-\mint:\mint}, samples=100, line width = .95] plot function{t + \tad,\k*\k*t*t};

            \draw[->] (-\ss*.5 + \tad,\httrop) -- ++(\ss,0);

            \node[circle,fill,inner sep=0pt, minimum width=3pt] at (-\ss*.5 + \tad,\httrop) {};
			
			}

            \pgfmathsetmacro{\ell}{1.5};

			\draw[->, red, line width  = .8] (-\ss*.5 + 1,\httrop+.3) -- ++(\ell,0);
			\draw[->, red, line width  = .8] (-\ss*.5 + \dx,\httrop+.3) -- ++(\ell,0);
			\draw[->, red, line width  = .8] (-\ss*.5 + 2*\dx,\httrop+.3) -- ++(\ell,0);

			\node[circle,fill=red,inner sep=0pt, minimum width=3pt] at (-\ss*.5 + 1,\httrop+.3) {};
			\node[circle,fill=red,inner sep=0pt, minimum width=3pt] at (-\ss*.5 + \dx,\httrop+.3) {};
			\node[circle,fill=red,inner sep=0pt, minimum width=3pt] at (-\ss*.5 + 2*\dx,\httrop+.3) {};

			\node at (0,2) {$t = 0$};
			\node at (\dx,2) {$t = \frac14$};
			\node at (2*\dx,2) {$t = 1$};

            \draw[very thin,color=gray, opacity=.3] (-1.8,-1.8) grid (1.8,1.8);

            \draw[very thin,color=gray, opacity=.3] (-1.8+\dx,-1.8) grid (1.8+\dx,1.8);

            \draw[very thin,color=gray, opacity=.3] (-1.8+2*\dx,-1.8) grid (1.8+2*\dx,1.8);

        \end{tikzpicture}  
  \caption{A depiction of the degeneration $x \mapsto tx^k$ from \ref{ex:piercedmapdegen}, and its tropicalisation. For $t \ne 0$ the tropicalisation is always given by the PL function mapping the vertex of the dual graph to the vertex in $\Sigma_X$, and with slope $k$ on the unique leg. For $t = 0$ the entire curve lands inside the divisor $D$. Tropically, this implies the vertex of the dual graph maps to the interior of the ray. The slope on the unique leg is still $k$.}
  \label{fig:piercedmapdegen}
\end{figure}

\begin{example}
\label{ex:piercedmapdegen}
This example and its tropicalisation are pictured in \ref{fig:piercedmapdegen}.
We consider the pierced map that on the level of underlying schemes is given by the following diagram
\[\begin{tikzcd}
C = \P_S^1 \ar[r,"f"] \ar[d] 		& X = \P^1 \times \P^1		\\
S = \A_t^1 &
\end{tikzcd}\]
where $f$ sends $(x,t)$ to $(x,tx^k)$ for a fixed $k \in \N$, and $C/S$ has a unique section $p$, mapping $S$ to $0$. We give $X$ the log structure given by the divisor $D = \P^1 \times 0$. We let $\alpha \in \ghost_X(X)$ denote the PL function corresponding to $D$. We give $S$ the log structure pulled back from ${s = 0} \subset \A_{t,s}^2$, and let $\ol{t},\ol{s} \in \ghost_S(S)$ denote the corresponding PL functions. Note that $\ol{t}$ vanishes outside of $0 \in S$. We let $C/S$ be the unique pierced curve with length $\ell = \ol{s} \in \ghost_S$. We give $f$ a log structure by sending $\alpha$ to $\ol{t} + k \ol{s}$. 

Equivalently, $f^\trop$ is given on the leg $[0,\ell]$ by $x \mapsto \ol{t} + k x$. The evaluation map $\ev: S \to X$ is constant $(0,0)$ on the level of underlying schemes, and tropically is given by \begin{align*}\ev^\trop: \R_{\geq 0} \cdot \ol{s} \oplus \R_{\geq 0} \cdot \ol{t} &\to \R_{\geq 0} \cdot \alpha \\ (\ol{s},\ol{t}) &\mapsto \ol{t} + k \ol{s}.\end{align*}

We can read off from $f^\trop$ that the tangency of $p$ to $D$ in this family is constant $k$. For $t \ne 0$ this can be seen from the algebraic picture; for $t = 0$ this is encoded in the logarithmic structure.
\end{example}

\begin{figure}
\centering
\begin{tikzpicture}[scale=.8]

           \draw[->] (0,0) -- ++(7,0);
           \pgfmathsetmacro{\vx}{3}
           \draw[->, red, line width  = .8] (\vx,.3) -- node[midway,above] {$\ell_3$} ++(2,0);
           \draw[->, red, line width  = .8] (\vx,.3) -- node[pos=0.75,right] {$\ell_2$} ++(0,2);
           \draw[->, red, line width  = .8] (\vx,.3) -- node[midway,above] {$\ell_1$} ++(-1.3,0);

           \node[circle,fill=red,inner sep=0pt, minimum width=3pt] at (\vx,.3) {};

           \node[circle,fill,inner sep=0pt, minimum width=3pt] at (0,0) {};

           \draw [decorate, decoration = {calligraphic brace, mirror}, line width = .6] (0.05,-0.2) --  node[midway,below] {$t$} (1.65,-0.2);

           \draw [decorate, decoration = {calligraphic brace, mirror}, line width = .6] (1.75,-0.2) --  node[midway,below] {$2\ell_1$} (\vx-.05,-0.2);

           \draw [decorate, decoration = {calligraphic brace, mirror}, line width = .6] (\vx+.05,-0.2) --  node[midway,below] {$2\ell_3$} (\vx+1.95,-0.2);

        \end{tikzpicture}  
  \caption{The tropicalisation of the map $C \to \P^1$, where $C$ is the universal curve over $\Mpt_{\Lambda}(\P^1)$ described in \ref{ex:piercedmapnegative}. The contact orders (slopes on the legs) are $(-2,0,2)$. The lengths of the line segments are given in red for the curve, and in black for $\Sigma_X$.}
  \label{fig:piercedmapnegative}
\end{figure}
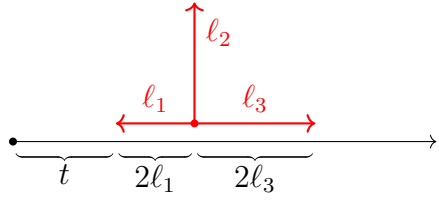

\begin{example}
\label{ex:piercedmapnegative}
Let $X = \P^1$ and consider the discrete data $\Lambda$ with $g = 0, n = 3, \beta = 0$ and contact orders $(\R_{\geq 0}, -2),(0,0),(\R_{\geq 0},2)$. Then $\Mpt_\Lambda(X)$ is a point with log structure of rank $4$. This rank $4$ monoid is $\N \cdot t \oplus \bigoplus_{i=1}^3 \N \cdot \ell_i$, with the tropical curve and tropical map pictured in \ref{fig:piercedmapnegative}. We write $\ul{(C,p_1,p_2,p_3)} = (\P_x^1,0,1,\infty)$ for the universal curve $C/\Mpt_\Lambda(X)$. Recall that a map $T\to \P^1$ is the same as an element of $M_{T}^\gp$ that is locally comparable to $1$. Under this identification, the universal map $f: C \to \P^1$ is given by $t \cdot \ell_1^2 \cdot x^{-2}$. The evaluation map $\ev_i$ is given by $t \cdot \ell_1^2 \cdot \ell_i^{a_i}$ where $a = (-2,0,2)$.
\end{example}

We have defined $\Mpt_\Lambda(X)$ as a category fibred in groupoids over $\LogSch$. It remains to show that it is actually a proper algebraic stack with log structure. It is possible to prove this directly using a similar argument as in the punctured case \cite{Abramovich2020Punctured-logar}. In the next section, we will show our moduli space has a proper, representable map to the space of punctured maps $\Mcal_\Lambda^\punc(X)$ \cite{Abramovich2020Punctured-logar}, from which properness and representability will follow directly.

\subsection{Comparison to punctured maps}\label{sec:punctured_comparison}

We now compare to \cite{Abramovich2020Punctured-logar}. 
A family $(C/S, p_1,\cdots,p_n, f : C \to X)$ of pierced maps factors through a unique pre-stable map from a punctured curve 
\begin{equation}\label{eqn:pierpunfactn}
	C \to \check C \to X,
\end{equation}
where the characteristic monoid of $\check C$ at marked points is the image of
\[f^* \bar \M_X \oplus \bar \M_S \oplus \NN \longrightarrow \g{\bar \M}_S \oplus \ZZ.\]

\begin{definition}	
Define a morphism
\[\Phi: \Mpt_{\Lambda}(X) \to \Mcal_{\Lambda}^\punc(X)\]
by sending pierced maps to the factorisation
\[(C/S, p_i, f : C \to X) \mapsto (\check C/S, \check C \to X)\]
from \ref{eqn:pierpunfactn}. 
\end{definition}

This map forgets the data of the logarithmic section, but there is a commutative square
\begin{equation}
\label{eq:mpiertompundiag}
\begin{tikzcd}
\Mpt_\Lambda(X) \ar[r] \ar[d]     &\Mcal_\Lambda^\punc(X) \ar[d, "\Phi"]    \\
\scr Q^n \ar[r]     &\ul{B\GG_m^n}
\end{tikzcd}
\end{equation}
Here $\scr Q$ is the algebraic stack with log structure parametrising positive elements of $\ghost$, with underlying algebraic stack $\ul{B\G_m}$. The first vertical map sends $(C/S,p_1,\cdots,p_n,f)$ to the vector of nonzero sections $(\ell_1,\dots,\ell_n) \in \ghost_{S,>0}^n$, and the second vertical map sends $(C/S,p_1,\cdots,p_n,f)$ to the cotangent line bundle $\bb L_i$ at $p_i$. The commutativity follows from \cite[Lemma~4.3]{Holmes2023LogarithmicCohomologicalFT}.

\begin{definition}
\label{def:mpuncq}
Write ${{\Mcal}^\punc_\Lambda(X)}_\scr Q$ for the pullback $\Mcal^\punc_\Lambda(X) \times_{B\GG_m^n} \scr Q^n$, and write \[\Phi_\scr Q : \Mpt_\Lambda(X) \to {\Mcal}_\Lambda^\punc(X)_\scr Q\] for the map induced by \ref{eq:mpiertompundiag}.
\end{definition}

On $S$-points the map $\Phi_\scr Q: \Mpt_\Lambda(X) \to \Mcal^\punc_\Lambda(X)_\scr Q$ is given by
\[
(C/S,p_i: S \to C, f: C \to X) \mapsto (\check C/S,p_i: S \to C, f: \check C \to X).
\]
On both sides, $p_1,\dots,p_n$ are sections of log schemes. 
The map $\Phi_{\scr Q}$ is a log monomorphism.

\begin{theorem}
\label{thm:closedinblowup}
The map $\Phi_\scr Q: \Mpt_\Lambda(X) \to \Mcal_\Lambda^\punc(X)_\scr Q$ factors as a closed immersion followed by a log blowup.
\end{theorem}
\begin{proof}
An element $(\check C/S,p_i: S \to C, f: \check C \to X)$ of $\Mcal^\punc(X)_\scr Q$ lies in the image of $\Phi_{\scr Q}$ if and only if $\ghost_{\check C,p_i} \subset \ghost_{C,p_i}$ as subsets of $\ghost^\gp_{C}$ for all $i$. For $m \in \ghost_{\check C,p_i}$, if $m$ has non-negative slope then $m \in \ghost_{C,p_i}$. If $m = (a,n)$ with $a \in \ghost_S, n \in \Z_{< 0}$ then we have $m \in \ghost_{C,p_i}$ if and only if $\alpha(a + n\ell) = 0$. This is a closed immersion inside a log blowup, where the log blowup is the log blowup making $a + n \ell$ and $0$ comparable, and then $\alpha(a + n\ell) = 0$ cuts out a closed condition. 
\end{proof}

In the strict normal crossing case, we can write this down even more explicitly using some of the tropical language of \ref{sec:tropicalmaps} and some language from \cite{BNR}. We will do so in \ref{subsec:comppunctinvariants}, in order to compare pierced invariants and punctured invariants.

\begin{example}
If there are no negative tangencies, then $\Mpt_\Lambda(X) = \Mcal^\punc_\Lambda(X)_{\scr Q}$.
\end{example}

\begin{example}[continues \ref{ex:piercedmapnegative}]
For $X = \P^1$ and $g = 0, n = 3, \beta = 0$ and contact orders given by $(\R_{\geq 0},-2),(0,0),(\R_{\geq 0},2)$ all three spaces
\[\Mpt_\Lambda(X) \to \Mcal^{\punc}_\Lambda(X)_{\scr Q} \to \Mcal^{\punc}_\Lambda(X)\]
have a point as underlying algebraic stack. The maps on the monoids are given by
\[
	\N \cdot t \oplus \bigoplus_{i=1}^3 \N \cdot \ell_i \leftarrow  \N \cdot \alpha \oplus \bigoplus_{i=1}^3 \N \cdot \ell_i \leftarrow \N \cdot \alpha
\]
where $\alpha$ is mapped to $t + 2 \ell_i$.
\end{example}

As a corollary, we obtain the following proposition.

\begin{proposition}
\label{prop:mpieralgebraicproper}
The moduli space $\Mpt_{\Lambda}(X)$ is an algebraic stack with log structure. It is proper if $X$ is log smooth and projective.
\end{proposition}

We will see that something stronger holds: the map $\Phi$ is an isomorphism on underlying stacks. We will show this using the theory of minimal objects by Gillam \cite{Gillam2012Logarithmic-sta}. We remark that in the strict normal crossing case, this will also follow from \ref{eq:bnrcomparison}.

We determine the minimal objects of $\Mpt_\Lambda(X)$, and show the map $\Phi: \Mpt_\Lambda(X) \to \Mcal^\punc_\Lambda(X)$ induces a bijection on minimal objects.

\begin{construction}
\label{constr:minmonoid}
Let $(C/S,p_1,\dots,p_n: S \to C, f: C \to X) \in \Mpt_{\Lambda}(X)$ where $\ul{S}$ is a geometric point. Let $P = f^* \ghost_X$. Let $V$ be the set of vertices and $E$ the set of edges. For $h$ a half edge, write $a_h\colon P_h \to \Z$ for the slope along $h$. If $h$ is an internal half-edge, which we think of as a directed edge $e_h = (h, i(h))$ from a vertex $v_h$ to a vertex $v_{i(h)}$, we write $\phi_h\colon P_{e_h} \to P_{v_h}$ for the generisation map.
For a leg $i$ incident to a vertex $v$ we let $\phi_i\colon P_{p_i} \to P_v$ denote the generisation map. Consider the monoid \[\ghost = \prod_{v \in V} P_v \times \prod_{e \in E} \bb N \times \prod_{i=1}^n \bb N \times \prod_{i=1}^n P_{p_i}\] where a general element is written \[((\alpha_v)_{v \in V},(\ell_e)_{e \in E}, (\ell_i)_{i=1}^n, (t_i)_{i=1}^n);\] we think of an element of one of the $P_v$ as an element of $\ghost$ by setting all the other entries to $0$, and similarly for the other terms. Let $R$ be the subgroup of $\ghost^\gp$ generated by elements
\[
-\phi_h(m) + \phi_{i(h)}(m) - a_e(m) \ell_e
\]
for an oriented edge $e = (h, i(h))$ and $m \in P_e$, and by elements
\[
\phi_i(m) + a_{h_i}(m) \ell_i - t_i(m)
\]
for $i \in \{1,\dots,n\}$ and $m \in P_{p_i}$. 
We define \[\ghost_{S,\bas} \coloneqq \ghost/R,\] the quotient of $\ghost$ by the equivalence relation induced by $R$ in the category of fs monoids. We call this the \emph{basic monoid}. Note that there is a natural map $\ghost_{S,\bas} \to \ghost_{S}$.
\end{construction}

\begin{definition}
Let $(C/S,p_1,\dots,p_n\colon S \to C, f\colon C \to X) \in \Mpt_{\Lambda}(X)$ be a pierced map. We say $C$ is \emph{basic} if at every strict geometric point $s$ of $S$, the natural map $\ghost_{s,\bas} \to \ghost_s$ from \ref{constr:minmonoid} is an isomorphism.
\end{definition}

\begin{example}[continues \ref{ex:piercedmapnegative}]
In this case we have $\ghost = \N \cdot \alpha \times \prod_{i=1}^3 \N \cdot \ell_i \times \prod_{i=1}^3 \N \cdot t_i$. The quotient is $\ghost_\bas = \N \cdot t_1 \times \prod_{i=1}^3 \N \cdot \ell_i$, with the quotient map sending $\alpha,t_2$ to $t_1 + 2\ell_1$ and $t_3$ to $t_1 + 2\ell_1 + 2\ell_3$.
\end{example}

One can check the following standard proposition.
\begin{proposition}
\label{prop:minimaleqbasic}
Let $(C/S,p_1,\dots,p_n\colon S \to C, f\colon C \to X) \in \Mpt_{\Lambda}(X)$ be a pierced map. It is minimal in the sense of \cite{Gillam2012Logarithmic-sta} if and only if it is basic.
\end{proposition}

From this, the isomorphism on underlying stacks follows.

\begin{proposition}
The map $\Phi: \Mpt_{\Lambda}(X) \to \Mcal^\punc_{\Lambda}(X)$ is an isomorphism on underlying stacks.
\end{proposition}
\begin{proof}
Clearly the map $\Mcal^\punc_{\Lambda}(X)_{\scr Q} \to \Mcal^\punc_{\Lambda}(X)$ is an isomorphism on underlying stacks, so it remains to show the same for $\Mpt_\Lambda(X) \to \Mcal^\punc_{\Lambda}(X)_{\scr Q}$. By \cite[Descent Lemma]{Gillam2012Logarithmic-sta} this is the case if and only if minimal objects in $\Mcal^\punc_{\Lambda}(X)_{\scr Q}$ have a unique lift to minimal objects in $\Mpt_\Lambda(X)$.

Let $C/S$ be a minimal object in $\Mcal^\punc_{\Lambda}(X)_{\scr Q}$, and let $C'/S' \in \Mpt_\Lambda(X)$ be a minimal object over it. We use the notation of \ref{constr:minmonoid}. By the description of minimal monoids in \cite[Proposition~2.32]{Abramovich2020Punctured-logar} and by \ref{constr:minmonoid}, we have
\begin{align*}
	\ghost_S &= \left(\prod_{v \in V} P_v \times \prod_{e \in E} \bb N \times \prod_{i=1}^n \bb N\right)/Q \\
	\ghost_{S'} &= \left(\prod_{v \in V} P_v \times \prod_{e \in E} \bb N \times \prod_{i=1}^n \bb N \times \prod_{i=1}^n P_{p_i}\right)/R.
\end{align*}
where $R$ is as in \ref{constr:minmonoid} and \[Q = R \cap \left(\prod_{v \in V} P_v \times \prod_{e \in E} \bb N \times \prod_{i=1}^n \bb N\right).\]
We see we have $\ghost_{S'}^\gp = \ghost_S^\gp$; the newly added variables $t_i(m) \in \ghost_{S'}$ for $1 \leq i \leq n, m\in P_{p_i}$ are equal to $\phi_i(m) + a_{h_i}(m) \ell_i \in \ghost_S^\gp$, and in fact we obtain the explicit description \[\ghost_{S'} = \ghost_S[\{t_i(m) | 1 \leq i \leq n, m \in P_{p_i}\} \subset \ghost_S^\gp].\]

We have seen that $\ghost_{S'}$ is uniquely determined by $\ghost_S$. Next, we show the map $\M_{S} \to \Ocal_S$ extends uniquely to a map $\M_{S'} \to \Ocal_S$. From the definition of a pierced curve (\ref{def:piercedcurve}(\oref{enum:item:maxext})), we see that these newly added elements $t_i(m)$ must be sent to $0$ in $\Ocal_S/\Ocal_S^\times$, and hence the maps $\ghost_{S'} \to \Ocal_S, \M_S \to \Ocal_S$ have a unique lift $\M_{S'} \to \Ocal_S$, sending $\M_{S'} \setminus \M_S$ to $0$.

Similarly we see $\ul{C}$ carries a unique log structure $C$ that makes $(C/S,p_1,\dots,p_n,f)$ a pierced map. 
\end{proof}

\section{The space of tropical pierced maps}
\label{sec:tropicalmaps}

\subsection{Tropical curves and maps to a cone stack}

Let $\sigma$ be a cone with dual monoid $M$. A tropical pierced curve over $\sigma$ is a graph with edge and leg lengths in $M_{>0}$, and with vertices decorated by genera. More formally, we make the following definition. 
\begin{definition}
An $n$-marked \emph{tropical pierced curve} over $\sigma$ is a tuple $(V,H,r,\iota,\ell,g)$ where $V$ and $H$ are finite sets with $V$ non-empty (the \emph{vertices} and \emph{half-edges}), $r\colon H \to V$ is a map, $\iota$ is a fixed-point-free involution, and $g\colon V \to \bb Z_{\ge 0}$, and $\ell\colon H \to M_{>0}$ satisfies $\ell(h) = \ell(\iota(h))$ for all $h \in H$. An \emph{edge} of $\Gamma$ is a pair of half-edges interchanged by $\iota$ with both halves sent by $r$ to elements of $V$; a \emph{leg} is a pair of half-edges interchanged by $\iota$ where exactly one of the halves $h$ has $r(h) \notin V$; the \emph{tip} of the leg is the image of that $h$ under $r$. We require the following: 
\begin{enumerate}
\item
For every $1 \le i \le n$ there is exactly one half edge $h$ with $r(h) = i$;
\item the corresponding topological graph is connected. 
\end{enumerate}
The \emph{genus} of the tropical curve is the sum of the first Betti number of the graph with the sum of the genera $g(v)$ of the vertices. 
\end{definition}
This differs from the unpierced theory, where only edges have lengths and legs do not.

To a tropical pierced curve $\Gamma$ is naturally associated a cone complex, essentially by taking the `cone over' $\Gamma$. Let us be a bit more precise. For each vertex and leg-tip we take a copy of $\sigma$. To an edge or leg of length $m \in M_{>0}$ corresponds the cone over $\sigma$ given by taking duals of 
\begin{equation}
 M \to \{(a, b) \in M^2 : m \mid a-b\}. 
\end{equation}
Face maps are duals of the projections from the edge/leg monoid to the vertex/tip monoid. 

A map from a tropical pierced curve $\Gamma$ to a cone stack $T$ is just a map from the cone over $\Gamma$ to $T$. An example is illustrated in \ref{fig:piercedmapnegative}. 

\begin{definition}
Given a cone complex $T$, we define the CSWB $\Mpt^t(T)$ first by saying that the underlying cone stack is the fibred category over $\cat{RPC}$ whose objects over a cone $\sigma$ are pairs of a tropical pierced curve $\Gamma$ over $\sigma$ and a map of cone stacks $\Gamma \to T$. The boundary is then the locus where either some leg has length $0$, or the tip of a leg meets a proper face of the minimal cone containing the image of the leg.
\end{definition}
	
In the special case where the base monoid $M$ is $\bb R_{\ge 0}$ and the boundary is $0$, these metric graphs really look like metric graphs, and the maps really look like maps from a graph to a cone. We encourage the reader to keep this case in mind for geometric intuition, see again \ref{fig:piercedmapnegative}. 

\begin{definition}
Given a log stack $X$, which we always assume comes together with a strict map $X \to X^t$ to a cone stack, we define the CSWB $\Mpt^t(X, X^t)$ whose objects over a cone $\sigma$ are tuples 
\begin{equation}
(\Gamma, f\colon \Gamma \to X^t, \beth\colon V \to \mathsf{Eff}_1(X) \subset \on{Hom}(\on{Pic}(X), \bb Z))
\end{equation}
of a tropical curve $\Gamma$ over $\sigma$, a tropical map $\Gamma \to X^t$, and an assignment of an effective curve class on $X$ to each vertex of $\Gamma$. 
\end{definition}

\subsection{Slopes and balancing}
\label{subsec:balancing}

There is a natural map 
\begin{equation}
\{(a, b) \in M^2 : m \mid a-b\}^\gp \to \bb Z; (a, b) \mapsto \frac{a-b}{m}. 
\end{equation}
Given a half-edge $h$ of $\Gamma$ of length $m$, we call the map $M_h \to \bb Z$ defined by this formula the \emph{slope}. Given a map $\Gamma \to T$ and a PL function on $T$, we define the slope of the PL function along a half-edge in $\Gamma$ by pulling back and then applying the slope map just defined. 

\begin{definition}\label{def:balanced}
A point $(\Gamma, f, \beth)$ of $\Mpt^t(X, X^t)$ is \emph{balanced} at a vertex $v$ if the triangle
\begin{equation}
\begin{tikzcd}
  PL(X^t) \arrow[d] \arrow[r, "\Sigma_v"] & \bb Z \\
\on{Pic}(X) \ar[ur, swap, "\beth(v)"] & 
\end{tikzcd}
\end{equation}
commutes, where the map $\Sigma_v$ takes the sum of outgoing slopes at $v$. It is \emph{balanced} if it is so at all vertices. 
\end{definition}

\begin{lemma}
Let $(\Gamma, f, \beth)$ be a point of $\Mpt^t(X, X^t)$. If it is realisable by a pierced map $(C,f) \in \Mpt(X)$, then it is balanced.
\end{lemma}
\begin{proof}
We fix a vertex $v$ and a PL function $\alpha$ on the star of $\sigma_v$. 
Then 
\begin{equation}
\beth(v)(\ca O(\alpha)) = \deg \ca O(\alpha)|_{C_v} = \deg \ca O_{C_v}(\alpha|_{C_v}) = \sum_{h \mapsto v} s_h(\alpha) = \Sigma_v(\alpha). \qedhere
\end{equation}
\end{proof}

\subsection{Stability}
\begin{definition}
Let $(\Gamma, f, \beth)$ be a point of $\Mpt^t(X, X^t)$. We call it \emph{unstable} at a vertex $v$ if $v$ has genus 0, valence $< 3$, or genus 1 and valence 0, and has vanishing curve class. It is \emph{stable} if there are no unstable vertices. Write 
$\Mpt^{t,st}(X, X^t) \subseteq \Mpt^{t}(X, X^t)$
for the substack of stable maps. 
\end{definition}

Assuming it is also balanced, this implies that the sum of outgoing slopes vanishes for every PL function, which is automatic if the vertex and its two half-edges lie on a straight line within a single cone.

\begin{lemma}
\label{lem:stabilisation}
There is a stabilisation map 
\begin{equation}
\Mpt^t(X, X^t) \to \Mpt^{t}(X, X^t)
\end{equation}
which takes balanced objects to balanced objects in the sense of \ref{def:balanced}. 
\end{lemma}
\begin{proof}
The stabilisation map on curves is constructed in \cite{Holmes2023LogarithmicCohomologicalFT}, and naturally extends to maps from them. For the fact that it sends balanced to balanced, there are two cases. For rational bridges it's clear as we remove a vertex and don't change anything else at other vertices. For a rational tail the slope going into it must be 0, hence removing this does not affect balancing at the vertex it is attached to. 
\end{proof}

\subsection{Contact orders and discrete data}

Recall from \ref{def:contact_orders} the definition of a contact order for $T$ viewed as a stack with log structure; concretely, a contact order for $T$ is a pair $(\sigma, a)$ where $\sigma$ is a cone of $T$ and $a \in \sigma^\gp$ is an integral vector, such that $a$ does not lie in $\tau^\gp$ for $\tau$ any proper face of $\sigma$. We define compatibility as in \ref{def:has_disc_data}: 

\begin{definition}
Given an object $(\Gamma, f, \beth)$ of $\Mpt^t(X, X^t)$ over a cone $\sigma$, and $1 \le i \le n$, and a contact order $(\sigma, a)$, we say the $i$th marking of $(\Gamma, f, \beth)$ is \emph{compatible with $(\sigma, a)$} if the tip of the $i$th leg lies in the interior of a cone which contains $\sigma$, and if the slope of $f$ at the $i$th marking is equal to $a$. 
\end{definition}

\begin{remark}
For every $(\Gamma, f, \beth)$ and every $i$ there is a unique contact order such that the $i$th marking is compatible with that contact order. 
\end{remark}

Recall from \ref{def:discrete-data} that a discrete datum for $X \to X^t$ is a tuple 
	\[\Lambda = (g,n,\beta, (\sigma_i, a_i)_{1 \le i \le n})\] 
   where $g$ and $n$ are non-negative integers, $\beta\colon \on{Pic}(X) \to \bb Z$ is a group homomorphism, and each $(\sigma_i, a_i)$ is a contact order.  

\begin{definition}
We say a point $(\Gamma, f, \beth)$ of $\Mpt^t(X, X^t)$ is \emph{compatible} with the discrete data if it has genus $g$ and $n$ markings, if $\beta = \sum_v \beth(v)$, and if for each $i$, the $i$th marking is compatible with $(\sigma_i, a_i)$. We write 
\begin{equation}
\Mpt^t_\Lambda(X, X^t)
\end{equation}
for the set of objects $(\Gamma, f, \beth)$ of $\Mpt^t(X, X^t)$ which are compatible with $\Lambda$. 
\end{definition}

\subsection{Tropical virtual fundamental class}
\label{subsec:tropvfc}

In the presence of negative tangencies, the space $\Mpt^t_\Lambda(X,X^t)$ need not be smooth, nor even equidimensional. Nevertheless, its logarithmic Chow group is well understood in terms of homological piecewise polynomials \cite{RPSS_log_taut} (see \ref{sec:hPP} for a recap).

To define the tropical virtual fundamental class, we will need two kinds of piecewise linear functions: the length $\ell_i$ of a leg, and a piecewise linear function $r_i$, which is roughly the maximal length by which the leg $i$ can be extended before it would reach a wall. We now give a formal definition of $r_i$ not assuming strict normal crossing, and afterwards give a more explicit description and an example in the strict normal crossing case.

\begin{figure}
    \centering 
    \begin{tikzpicture}[scale=1.2]
    \draw[->] (0,0) -- (2.5,0);
    \draw[->] (0,0) -- (0,2.5);
    \fill (2,0) circle (2pt);
    \draw[->, thick] (2,0) -- (0.56,1.44);
    \draw[dashed] (.56, 1.44) to (0, 2);
    \draw (0.56,0) -- (0.56,-0.08);
    \node[below] at (0.56,-0.08) {$r_1$};
    \draw (0,1.44) -- (-0.08,1.44);
    \node[left] at (-0.08,1.44) {$\ell$};
    \draw (2,0) -- (2,-0.08);
    \node[below] at (2,-0.08) {$r_1+\ell$};
    \node[right] at (2.5, 0){$D_1$};
    \node[left] at (0, 2.5){$D_2$};
    \begin{scope}[shift = {(5, 0)}]
        \fill[blue!20] (0, 0) rectangle (2, 2);
        \draw[->, dashed] (0,0) -- (2.1,0);
        \draw[->, dashed] (0,0) -- (0,2.1);
        \node[below] at (2, 0){$r_1$};
        \node[left] at (0, 2){$\ell$};
    \end{scope}
\end{tikzpicture}
\caption{The spaces in \ref{ex:simpletropicalex}. 
Left: The tropical type in $\R_{\geq 0}^2$. The lengths $\ell, r_1$ are the length of the edge in the target and the remaining edge length until the dashed line meets the boundary edge. The contact order with $D_2$ is positive, so the minimum of \ref{lem:riminrij} is $r_1 = r_{11}$. 
Right: The tropical moduli space of this type is the CSWB $\R_{> 0}^2$.}
\label{fig:simpletropicalex}
\end{figure}
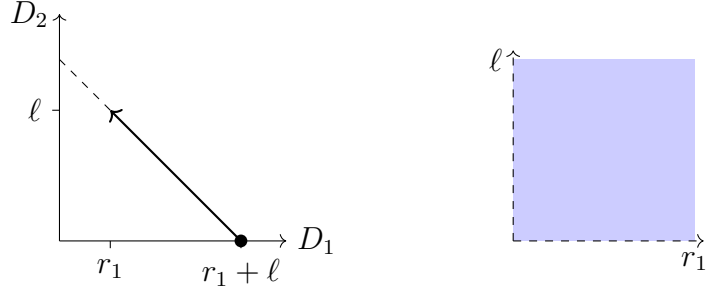

\begin{construction}\label{const:ri}
	Fix an object $(\Gamma, f, \beth)$ of $\Mpt_\Lambda^t(X, X^t)$ over a cone $\tau$ and fix $1 \le i \le n$. We will define a PP function $r_i$ on $\tau$ (together with a minimal subdivision $\tilde \tau$ of $\tau$ on which $r_i$ is linear), which measures how far the $i$th leg can be extended before reaching the boundary of the star of $\sigma_i$ (its image cone). 
    
    First, if $a_i \in \sigma_i \subseteq \sigma_i^\gp$ then we set $r_i = 1$, a sPP function of degree 0. Now suppose that $a_i \notin \sigma_i$. Then $r_i$ will be PL on $\tau$ and is defined as follows. Let $\sigma'$ be the smallest cone containing $\sigma_i$ and which contains the image of $\tau$ under the map given by restricting $f$ to the tip of the $i$th leg. 
    
    Then there is a unique map $F\colon \tau \times \bb R_{\ge 0} \to {\sigma'}^\gp$ which on $\tau \times \{0\}$ coincides with the restriction of $f$ to the vertex to which the $i$th leg is attached, and whose restriction to the graph of $\ell_i$ inside $\tau \times \bb R_{\ge 0}$ coincides with the restriction of $f$ to the tip of the $i$th leg. Note that $F$ sends $(0, 1) \in \tau \times \RR_{\geq 0}$ to $a_i \in {\sigma'}^\gp$. 
    
    Let $\delta_{\sigma'}$ be the boundary complex of $\sigma'$, then the pullback of $\delta_{\sigma'}$ along $F$ is the graph of a PL function on $\tau$, which we define to be  $r_i + \ell_i$, 
    and the projection of this pullback to $\tau$ gives a canonical (minimal) subdivision on which $r_i$ becomes strict. 

    Define $r_i \coloneqq \max\{x : F(t,x) \in \sigma'\}-  \ell_i$, and the projection of the pullback of the boundary complex of $\sigma'$ along $F$ to $\tau$ gives a canonical (minimal) subdivision on which $r_i$ becomes strict. 

	The function $r_i$ and subdivisions $\tilde \tau$ assemble to a global PP function $r_i$ on $\Mpt_\Lambda^t(X, X^t)$, together with a global subdivision 
	\begin{equation}\label{eqn:multiplerisubdivs}
	\Mpt_\Lambda^{t,i}(X, X^t) \to \Mpt_\Lambda^{t}(X, X^t)	
	\end{equation}
	on which it is linear. 
\end{construction}

\begin{example}\label{ex:simpletropicalex}
    Let $X = \P^2$ with boundary divisor given by two of the invariant divisors $D = D_1 + D_2$ instead of the full toric boundary as in \ref{fig:simpletropicalex}. Its tropicalization is then $X^t = \R_{\geq 0}^2$. Consider the tropical moduli space $\Mpt^t_{\Lambda}(X, X^t)$ of genus-zero maps of degree one with one marked point of contact order $1$ along edge $D_2$ and $-1$ along $D_1$, ignoring balancing for simplicity. 

    The function $r_1$ recording the length by which one can extend the edge before meeting the boundary is already piecewise linear here. The map on schemes must send the source curve into $D_1$ with image a line, so it is an isomorphism to $D_1$. The remaining data in a stable map consists of the pair of lengths $\ell$ and $r_1$, both of which must be positive. This characterizes the tropical moduli space $\Mpt^t_\Lambda(X, X^t)$ as the CSWB $\R_{>0}^2$. 
\end{example}

\begin{definition}
\label{def:Mmu}
Write 
\[\Mptmt_\Lambda(X, X^t) \to \Mpt^t_\Lambda(X, X^t)\]
for the log blowup given by superimposing all the subdivisions \ref{eqn:multiplerisubdivs} as $i$ varies from $1$ to $n$. 
\end{definition}

In the SNC case we can express $r_i$ in terms of simpler functions $r_{i,j}$, which are also important for the comparison to \cite{BNR}.

\begin{definition}
\label{def:rij}
Suppose that $X$ is SNC with boundary divisor $D = \cup_{1 \le j \le k} D_j$. 
Let $1 \leq i \leq n$ and $1 \leq j \leq k$, with $a_{ij}$ the corresponding component of the contact order. If $a_{ij} \ge 0$ we set $r_{ij} = 1$, a PP function of degree $0$. On the other hand, suppose $a_{ij} <0$. Fix an object $(\Gamma, f, \beth)$ of $\Mpt_\Lambda^t(X, X^t)$ over a cone $\tau$, and let $F\colon \tau \times \bb R_{\ge 0} \to \bb R^k$ be the unique map whose restriction to $\tau \times \{0\}$ coincides with the restriction of $f$ to the vertex to which the $i$th leg is attached, and whose restriction to the graph of $\ell_i$ inside $\tau \times \bb R_{\ge 0}$ coincides with the restriction of $f$ to the tip of the $i$th leg. Then the pullback of the $j$th coordinate hyperplane inside $\bb Q^k$ is the graph of a linear function on $\tau$ which we define to be $r_{ij} + \ell_i$. 
\end{definition}
The value $r_{ij}$ can be interpreted as the maximum distance leg $i$ can lengthen by before it crosses wall $j$.
\begin{lemma}\label{lem:riminrij}
In the SNC case, given $1 \le i \le n$ we have 
\begin{equation}
r_i = \min_j r_{ij}
\end{equation}
where we interpret the minimum of a PL function $r$ and the degree-0 PP function $1$ to be $r$. 
\end{lemma}

\begin{example}
Suppose that $X^t = \bb R_{\ge 0}^2$, $n=2$, and we have contact orders $a_1 = (1,1)$ and $a_2 = (-1,-1)$. We can define a map $\tau = \bb R_{\ge 0}\Span{x,y,\ell_1, \ell_2} \to \Mpt^t(X^t)$ by sending a point $(x,y,\ell_1, \ell_2)$ to the graph where leg $i$ has length $\ell_i$, and the tip of leg 2 has coordinates $(x,y)$ (see \ref{fig:cone_r_i}). Then we have 
\begin{equation}
\begin{split}
&r_{11} = r_{12} = r_1 = 1\\
&r_{21} = x, r_{22} = y , r_2 = \min(x,y), 
\end{split}
\end{equation}
in particular $r_2$ is not strict. 
\end{example}

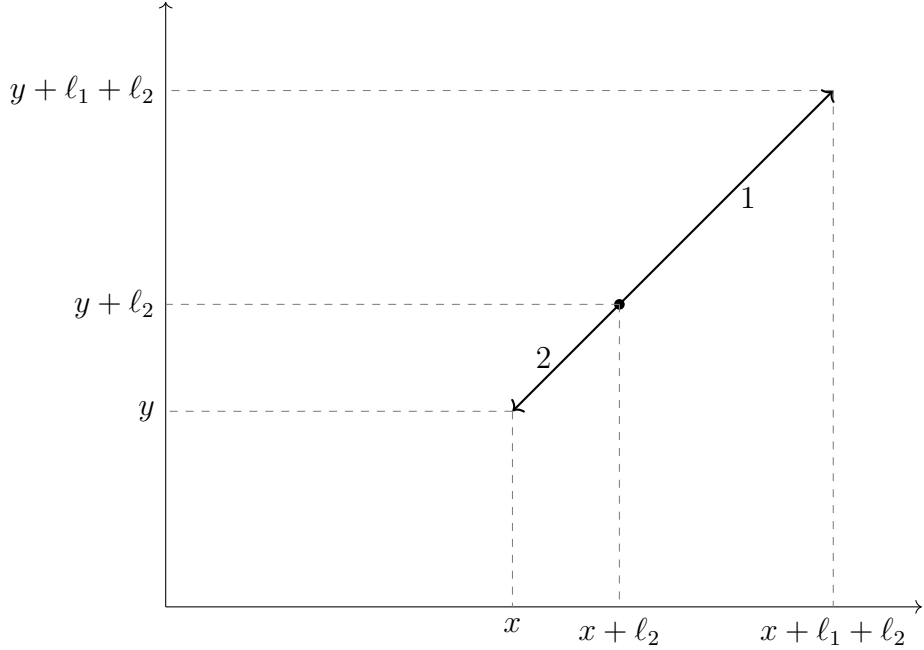
\begin{figure}
\begin{tikzpicture}
  \draw[->] (0,0) -- (10,0);
  \draw[->] (0,0) -- (0,8);
  
  \coordinate (dot) at (6,4);
  \coordinate (arrow1tip) at ($(dot) + (2.828,2.828)$);
  \coordinate (arrow2tip) at ($(dot) + (-1.414,-1.414)$);
  
  \fill (dot) circle (2pt);
  
  \draw[->, thick] (dot) -- (arrow2tip);
  
  \draw[->, thick] (dot) -- (arrow1tip);
  
  \node at ($(dot) + (1.7,1.414)$) {1};
  \node at ($(dot) + (-1,-0.707)$) {2};
  
  \draw[dashed, gray] (arrow2tip) -- (arrow2tip |- 0,0);
  \draw[dashed, gray] (arrow2tip) -- (0,0 |- arrow2tip);
  \node[below] at (arrow2tip |- 0,0) {$x$};
  \node[left] at (0,0 |- arrow2tip) {$y$};
  
  \draw[dashed, gray] (dot) -- (dot |- 0,0);
  \draw[dashed, gray] (dot) -- (0,0 |- dot);
  \node[below] at (dot |- 0,0) {$x + \ell_2$};
  \node[left] at (0,0 |- dot) {$y + \ell_2$};
  
  \draw[dashed, gray] (arrow1tip) -- (arrow1tip |- 0,0);
  \draw[dashed, gray] (arrow1tip) -- (0,0 |- arrow1tip);
  \node[below] at (arrow1tip |- 0,0) {$x + \ell_1 + \ell_2$};
  \node[left] at (0,0 |- arrow1tip) {$y + \ell_1 + \ell_2$};
\end{tikzpicture}
\caption{A cone of $\Mpt^t(X^t)$}
\label{fig:cone_r_i}
\end{figure}

\begin{definition}\label{def:cpct_notation}
We define $\Mpt^t_\Lambda(X^t)$ to be the cone complex
\begin{equation}
\Mpt^t_\Lambda(X^t) \coloneqq \Mpt^{t, \on{bal}, \st}_\Lambda(X, X^t)
\end{equation}
of stable, balanced objects (as in \ref{def:balanced}). Let $\Mptmt_\Lambda(X^t)$ be the analogous subdivision
\begin{equation}
\Mptmt_\Lambda(X^t) \coloneqq \Mpt^{\mu, t, \on{bal}, \st}_\Lambda(X, X^t) \longrightarrow \Mpt^t_\Lambda(X^t). 
\end{equation}
from \ref{def:Mmu}. 

We write $\Mpt_\Lambda(X^t)$ and $\Mptm_\Lambda(X^t)$ for the pullbacks $\Mpt \times_{\Mpt^t} \Mpt_\Lambda^t(X^t)$ and $\Mpt \times_{\Mpt^t} \Mptmt_\Lambda(X^t)$.
\end{definition}

\begin{definition}
\label{def:vfc}
We define the virtual fundamental classes in $\LogCH_*(\Mpt^t_\Lambda(X^t))$:
\begin{align*}
	[\Mpt^t_\Lambda(X^t)]^\vir &= \prod_i \ell_i r_i ,\\
	[\Mpt^t_\Lambda(X^t)]^\prodvfc &= \prod_i \left(\ell_i \prod_{j : a_{ij} < 0} r_{ij}\right)
\end{align*}
where we identify homological piecewise polynomials with $\LogCH_*(\Mpt^t_\Lambda(X^t))$ via the isomorphism $\Psi$ from \ref{thm:prrshomPP}. The class $[\Mpt^t_\Lambda(X^t)]^\prodvfc$ requires $X$ to be SNC so that the $r_{i,j}$ are defined.
\end{definition}

We see that for each $i$ we have that $r_i$ divides $\prod_{j : a_{ij} < 0} r_{ij}$, and hence we obtain the following relation between these classes.

\begin{proposition}\label{prop:trop_prod_vir_divisibility}
Assume $X$ is SNC. There is a piecewise polynomial $\kappa$ such that \[[\Mpt_\Lambda^t(X^t)]^\prodvfc = \kappa \cdot [\Mpt_\Lambda^t(X^t)]^\vir.\]
\end{proposition}
    
Both these PP functions are easily seen to vanish on the boundary of the CSWB. 

\begin{remark}
    The class $[\Mpt^t_\Lambda(X^t)]^\vir$ can be seen as living in $\CH(\Mptmt_\Lambda(X^t))$, and the class $[\Mpt^t_\Lambda(X^t)]^\prodvfc$ can be seen as living in $\CH(\Mpt^t_\Lambda(X^t))$, but we place them both in $\LogCH_*(\Mpt^t_\Lambda(X^t))$ to lighten the notation. 
\end{remark}

\section{Pierced logarithmic Gromov--Witten invariants}
\label{sec:piercedlogGW}

In this section we construct pierced logarithmic Gromov--Witten invariants, using the work on perfect obstruction theories in \cite[Section~4.1]{Abramovich2020Punctured-logar}. We recap some of their constructions. As we will compare our invariants with the invariants both from \cite{Abramovich2020Punctured-logar} and \cite{BNR}, we will discuss the most general version of the construction.

We start with a commutative diagram of algebraic stacks with log structure
\[
\begin{tikzcd}
 & V \arrow[d] \\
 Y \arrow[r] \arrow[d] & W \arrow[d] \\
S \arrow[r] & B
\end{tikzcd}
\]
with $Y/S$ flat and relatively Gorenstein and $V/W$ strict. In the application, we will take $B$ a point with trivial log structure, $Y/S$ the universal curve over $\Mpt_\Lambda(X^t)$, and $V/W$ the map $X/X^t$. Then one can consider the moduli space $M$ of lifts of $Y \to W$ to $Y \to V$, parametrising commutative diagrams
\[
\begin{tikzcd}
Y_T \arrow[rr] \arrow[d] \arrow[dr] & & V \arrow[d] \\
T \arrow[dr] & Y \arrow[r] \arrow[d] & W \arrow[d] \\
  & S \arrow[r] & B.
\end{tikzcd}
\]
The moduli space $M$ admits as universal object the following diagram
\begin{equation}
\label{eq:potabstract}
\begin{tikzcd}
Y_M \arrow[rr, "f"] \arrow[d, "\pi"] \arrow[dr] & & V \arrow[d] \\
M \arrow[dr] & Y \arrow[r] \arrow[d] & W \arrow[d] \\
  & S \arrow[r] & B.
\end{tikzcd}
\end{equation}
Now we define $\E = \Rb \pi_*(\omega_\pi \tensor f^* \LL_{V/W} )$. By functoriality of the cotangent complex and the isomorphism $\LL_{Y_M/Y} \simeq \pi^*\LL_{M/S}$, we have a map 
\[f^* \LL_{V/W} \to \pi^* \LL_{M/S}.\] 
This induces a map $\E \to \LL_{M/S}$ by duality, which by \cite[Proposition~4.2]{Abramovich2020Punctured-logar} is a perfect obstruction theory.

Applying this to $Y/S$ the universal curve over $\Mpt_\Lambda(X^t)$, and $V/W$ the map $X/X^t$ we obtain a diagram
\begin{equation}
\label{eq:pot}
\begin{tikzcd}
Y_M \arrow[rr, "f"] \arrow[d, "\pi"] \arrow[dr] & & X \arrow[d] \\
M \arrow[dr] & Y \arrow[r] \arrow[d] & X^t \arrow[d] \\
  & \Mpt_\Lambda(X^t) \arrow[r] & B
\end{tikzcd}
\end{equation}
and perfect obstruction theory
\[
	\E \to \LL_{M/\Mpt_\Lambda(X^t)}.
\]
We briefly remark that although \cite{Abramovich2020Punctured-logar} deals only with punctured curves, not pierced curves, their theory still applies here, because the map $X/X^t$ is strict and hence this obstruction theory does not depend on the log structure at all.

This $M$ is not exactly $\Mpt_\Lambda(X)$, because the universal map $f: Y_M \to X$ need not have curve class $\beta$ nor be stable. But curve class is a discrete invariant, and stability is an open condition, hence $\Mpt_\Lambda(X) \to M$ is an open inclusion, and by restriction we have the same perfect obstruction theory
\[
	\E \to \LL_{\Mpt_\Lambda(X)/\Mpt_\Lambda(X^t)}
\] 

Note that the map $\Mpt_\Lambda(X^t) \to \Mpt_\Lambda^t(X^t)$ is smooth and strict, as it is the pullback along the smooth map $\Mpt \to \Mpt^t = \Mfrak^\trop \times B\G_m^n$. As in \cite[Section~4.2]{Abramovich2020Punctured-logar}, these are compatible, and we obtain a perfect obstruction theory
\[
	\F \to \LL_{\Mpt_\Lambda(X)/\Mpt^t_\Lambda(X^t)}.
\]
for the map $\trop: \Mpt_\Lambda(X)\to \Mpt^t_\Lambda(X^t)$.

A perfect obstruction theory induces a Gysin pullback on Chow classes. In analogy to the punctured case, one can use this to define a Gromov--Witten correspondence.

\begin{definition}
\label{def:piercedcorrespondence}
Fix discrete data $\Lambda$. We define the \emph{pierced logarithmic Gromov--Witten correspondence} to be the Gysin pullback
\[
	\trop^!: \LogCH_*(\Mpt_\Lambda^t(X^t)) \to \LogCH_*(\Mpt_\Lambda(X))
\]
induced by the perfect obstruction theory $\F$ defined above as in \ref{def:potpullback}.
\end{definition}

We in particular want to pull back the two tropical classes $[\Mptmt_\Lambda(X^t)]^\vir, [\Mpt^t_\Lambda(X^t)]^\prodvfc$ from \ref{def:vfc}. This gives the following definition.

\begin{definition}\label{def:virtual_classes}
We define pierced logarithmic Gromov--Witten classes as the pullback of classes in $\LogCH_*(\Mpt_\Lambda^t(X^t))$ from \ref{def:vfc}: 
\begin{align*}
	[\Mpt_\Lambda(X)]^\vir &= \trop^! [\Mpt^t_\Lambda(X^t)]^\vir,\\
	[\Mpt_\Lambda(X)]^\prodvfc &= \trop^! [\Mpt^t_\Lambda(X^t)]^\prodvfc. 
\end{align*} 
\end{definition}

Again, by \ref{prop:trop_prod_vir_divisibility} these admit a relation. 
\begin{proposition}\label{prop:log_prod_vir_divisibility}
For the piecewise polynomial $\kappa$ in \ref{prop:trop_prod_vir_divisibility} we have \[[\Mpt_\Lambda(X)]^\prodvfc = \kappa \cdot [\Mpt_\Lambda(X)]^\vir.\]
\end{proposition}

Additionally, we will sometimes use the following virtual fundamental class in Chow homology of a particular model.
\begin{definition}
We let $\Mptm_\Lambda(X)$ denote the subdivision $\Mpt_\Lambda(X) \times_{\Mpt^t_\Lambda(X^t)} \Mptmt_\Lambda(X^t)$.
\end{definition}

\begin{definition}
We write
\begin{align*}
[\Mptm_\Lambda(X)]^\prodvfc = \trop^! [\Mpt^t_\Lambda(X^t)]^\prodvfc
\end{align*}
for the virtual fundamental class in $\CH_*(\Mptm_\Lambda(X))$.
\end{definition}
Note that this makes sense by \ref{sec:push_hpp}, and induces the class $[\Mpt_\Lambda(X)]^\prodvfc$ in log Chow homology. We will clearly distinguish whether we mean the virtual fundamental class in the log Chow group or in the Chow group. For the most part we will use the log Chow group, but for comparisons we will sometimes use the Chow group.

One can obtain pierced logarithmic Gromov--Witten invariants by capping with classes pulled back from $X$ along evaluation maps, and pushing forward to the moduli space of stable log curves.

\begin{definition}\label{def:logGWinvariants}
Fix $\Lambda$ discrete data and $\gamma = (\gamma_i)_{i=1}^n$ with $\gamma_i \in \CH^*(X_{\sigma_i})$, where $X_{\sigma_i}$ denotes the stratum corresponding to the $i$'th contact order. Then we define the log Gromov--Witten class
\[
\LogGW(X;\Lambda,\gamma) = \pi_*\left([\Mpt_\Lambda(X)]^\vir \cap \prod_{i=1}^n \ev_i^* \gamma_i \right)
\]
in $\LogCH_*(\Mpt_{g,n}^\st)$, where $\pi: \Mpt_\Lambda(X) \to \Mpt_{g,n}^\st$. One can make the same definition of log Gromov--Witten classes in Borel--Moore homology. 
\end{definition}

In the case where $X$ is snc, then $X_{\sigma_i} = \bigcup_{j: a_{ij \ne 0}} D_j$.

\begin{example}
\label{ex:tropspacenonnegtangencies}
In the case with all non-negative tangencies the space $\Mpt^t_\Lambda(X^t)$ is integral of dimension $-n$; the condition that the endpoints of a leg must be sent to a positive element of the corresponding cone is empty. More precisely, as a conestack with boundary, we have $\Mpt^t_\Lambda(X^t) = \Mbar^{t}_\Lambda(X^t) \times \R_{>0}^n$, where $\Mbar^{t}_\Lambda(X^t)$ is the classical space of stable balanced tropical maps to $X$ (a cone stack with empty boundary), and $\R_{>0}^n$ corresponds to the leg lengths (still required to be positive). The corresponding idealised Artin fan is thus an Artin fan times $B\G_m^n$.

Then the $\vir$ and $\prodvfc$ classes are just the fundamental class $\trop^![\Mpt^t_\Lambda(X^t)]$, and this recovers the classical notion of logarithmic Gromov--Witten invariants with non-negative tangencies \cite{Abramovich2014Stable-logaritmic-maps-II,Gross2013Logarithmic-gro}. 
\end{example}

\subsection{Comparison to punctured invariants}
\label{subsec:comppunctinvariants}

Now take $X$ to be additionally separated, as in \ref{sec:punctured_comparison}. 
We compare pierced logarithmic Gromov--Witten invariants to the punctured logarithmic Gromov--Witten invariants from \cite{Abramovich2020Punctured-logar,BNR}. The paper \cite{Abramovich2020Punctured-logar} defines a relative perfect obstruction theory on \[\Mcal^\punc_{\Lambda}(X) \xrightarrow{\trop} \Mcal^{\punc,t}_\Lambda(X^t),\] and the paper \cite{BNR} uses this obstruction theory to define \emph{refined punctured} log Gromov--Witten invariants.

For the remainder of this section we assume, as \cite{BNR}, that $X$ is strict normal crossings. We fix some notation, rephrasing some of \cite{BNR} in our language. Let $\Lambda$ be discrete data for $X$. Let $k_P$ denote the total number of puncturings in $\Lambda$, i.e.\ pairs $(p_i,D_j)$ of markings and boundary divisors with $a_{ij} < 0$. For each such pair $(i,j)$ they define a \emph{puncturing offset} $t_{ij}: \Mcal^{\punc,t}_\Lambda(X^t) \to \R_{>0}$, sending a punctured map $f: \Gamma \to X^t$ to the distance between the vertex incident to leg $i$ and the wall corresponding to $D_j$.

In the language of (homological) piecewise polynomials, each $t_{ij}$ is a piecewise linear function on $\Mcal^{\punc,t}_\Lambda(X^t)$. Then we have an even more explicit description of the map $\pi^t$, given by the following tropical version of \ref{thm:closedinblowup}.

\begin{proposition}\label{prop:bnrcomparisongeom}
We have a fibre square
\begin{equation}
\label{eq:bnrcomparison}
\begin{tikzcd}
	\Mpt^t_\Lambda(X^t) \ar[r, "f"] \ar[d,"\pi^t"] & \prod_{i,j : a_{ij} < 0} \R_{>0} \cdot s_{ij} \times \prod_{i=1}^n \R_{>0} \cdot \ell_i \ar[d,"\phi"]  \\
	\Mcal^{\punc,t}_\Lambda(X^t) \ar[r, "g"] & \prod_{i,j : a_{ij} < 0} \R_{>0} \cdot t_{ij}
\end{tikzcd}
\end{equation}
where $\phi$ is given by $t_{ij} = s_{ij} - a_{ij} \ell_i$.
\end{proposition}

We furthermore note that $s_{ij}$ is exactly $-a_{ij} r_{ij}$ where $r_{ij}$ is the piecewise linear function defined in \ref{def:rij}.

Now we will construct virtual pullbacks between the (tropical) pierced spaces and the (tropical) punctured spaces. First, on the level of (tropical) moduli spaces we obtain the following commutative diagram
\begin{equation}
\label{eq:commdiagpiercedpunc}
\begin{tikzcd}
\Mpt_{\Lambda}(X) \arrow[r, "\pi"] \arrow[d, "\trop"] & \Mcal^\punc_{\Lambda}(X) \arrow[d, "\trop"] \\
\Mpt^t_\Lambda(X^t) \arrow[r, "\pi^t"] & \Mcal^{\punc,t}_\Lambda(X^t)
\end{tikzcd}
\end{equation}

The vertical maps have perfect obstruction theories $\trop^!$ by \ref{def:piercedcorrespondence} and \cite[Section~4.1]{Abramovich2020Punctured-logar}. We explain how to give the horizontal maps $\pi,\pi^t$ a perfect obstruction theory. We can refine the commutative diagram to
\begin{equation}
\label{eq:refinedcommdiagpiercedpunc}
\begin{tikzcd}
\Mpt_{\Lambda}(X) \arrow[d, "\trop_1"] \arrow[r, "\pi"] & \Mcal^\punc_{\Lambda}(X) \arrow[d, "\trop_1"] \\
\Mpt_\Lambda(X)^t \arrow[d, "\trop_2"] \arrow[r, "\pi'"] & \Mcal^{\punc}_\Lambda(X^t) \arrow[d, "\trop_2"]\\
\Mpt_{\Lambda}^t(X^t) \arrow[r, "\pi^t"] & \Mcal^{\punc,t}_\Lambda(X^t)
\end{tikzcd}
\end{equation}
where the top square is a pullback diagram. 

Again, the vertical maps have obstruction theories as noted above for the pierced spaces, and in \cite[Section~4.1]{Abramovich2020Punctured-logar} for the punctured moduli spaces. Further, the map $\pi^t$ is a pullback of the map $\phi$ from \ref{eq:commdiagpiercedpunc}, and naturally obtain a perfect obstruction theory. To put an obstruction theory on $\pi'$, we use that the second square is a pullback of the square
\begin{equation}
\label{eq:secondsquarepiercedpunc}
\begin{tikzcd}
\Mpt_{g,n} \prod_{i,j : a_{ij} < 0} \R_{>0} \cdot s_{ij} \arrow[r, "\pi"] \arrow[d, "\trop"] & \Mfrak_{g,n} \prod_{i,j : a_{ij} < 0} \R_{>0} \cdot t_{ij} \arrow[d, "\trop"] \\
\Mpt_{g,n}^t \times \prod_{i,j : a_{ij} < 0} \R_{>0} \cdot s_{ij} \arrow[r, "\pi^t"] & \Mfrak_{g,n}^{t} \times \prod_{i,j : a_{ij} < 0} \R_{>0} \cdot t_{ij}.
\end{tikzcd}
\end{equation}
We have seen this for the two vertical arrows, and in \ref{prop:bnrcomparisongeom} for the bottom horizontal arrow, so it follows for the top horizontal arrow. All maps in this square are lci, and induce virtual pullbacks on $\trop_2, \pi', \pi^t$. As the first square is also a pullback square, we also obtain a virtual pullback for $\pi$.

\begin{proposition}
\label{prop:comparisonacgs}
The commutative diagram \ref{eq:commdiagpiercedpunc} induces a commutative diagram
\[
\begin{tikzcd}
\LogCH_*(\Mpt_{\Lambda}(X)) & \arrow[l, "\pi^!", swap] \LogCH_*(\Mcal^\punc_{\Lambda}(X)) \\
\LogCH_*(\Mpt^t_\Lambda(X^t)) \arrow[u, "\trop^!"]  & \arrow[l, "\pi^{t,!}"] \LogCH_*(\Mcal^{\punc,t}_\Lambda(X^t)) \arrow[u,"\trop^!", swap]
\end{tikzcd}
\]
of virtual pullbacks.
\end{proposition}
\begin{proof}
It suffices to show that in both squares in \ref{eq:refinedcommdiagpiercedpunc} the virtual pullbacks commute. For the first square, this is just commuting of virtual pullbacks, as $\pi^!$ is a pullback of $\pi{',!}$. For the second square, we simply use it is a pullback of the square \ref{eq:secondsquarepiercedpunc}. All virtual pullbacks are given by lci pullbacks of the morphisms in \ref{eq:secondsquarepiercedpunc}, and hence we get commutativity by \ref{prop:potpullbackfunctorial}.
\end{proof}

Now we can additionally compare our virtual fundamental classes with the refined virtual fundamental class of $\Mcal^\punc_{\Lambda}(X)$ constructed by Battistella, Nabijou and Ranganathan in \cite{BNR}. With the comparisons between the moduli spaces and the obstruction theories above, it remains to rewrite their definitions in the language of homological piecewise polynomials.

\begin{remark}
We caution that we use different notation from \cite{BNR}. The moduli space $\Mcal^\punc_{\Lambda}(X)$ is there referred to as $\text{Punct}_\Lambda(X|D)$, and the tropical moduli space $\Mcal^{\punc,t}_\Lambda(X^t)$ is there simply called $V(T)$, where $T$ is a tropicalisation of $X$.
\end{remark}

\begin{proposition}
\label{prop:comparisonbnr}
Let $[\Mcal^\punc_{\Lambda}(X)]^\mathsf{ref}$ denote the refined virtual fundamental class of $\Mcal^\punc_{\Lambda}(X)$ defined in \cite[Definition~1.9]{BNR}. Then we have
\begin{align*}
    \pi^! [\Mcal^\punc_{\Lambda}(X)]^\mathsf{ref} &= [\Mpt_{\Lambda}(X)]^\prodvfc \prod_{i, j : a_{ij} < 0} |a_{ij}|\\
    &= \kappa\cdot  [\Mpt_{\Lambda}(X)]^\vir \prod_{i, j : a_{ij} < 0} |a_{ij}|.
\end{align*}
where $\kappa$ is the piecewise polynomial from \ref{prop:trop_prod_vir_divisibility}.
\end{proposition}
\begin{proof}
We recall how the refined virtual fundamental class $[\Mcal^\punc_{\Lambda}(X)]^\mathsf{ref}$ is defined in the language used in this paper. Let $\Acal$ denote the cone stack of the cone stack with boundary $\Mcal^{\punc,t}_\Lambda(X^t)$. It is observed in \cite[Section~1.3]{BNR} that there is a fibre square
\begin{equation}
\label{eq:bnrrefinedvfc}
\begin{tikzcd}
		\Mcal^{\punc,t}_\Lambda(X^t) \ar[r] \ar[d] & \Acal \ar[d]  \\
	\R_{>0}^{k_P} \ar[r,"i"] & \R_{\geq 0}^{k_P}
\end{tikzcd}
\end{equation}
of cone stacks with boundary, where the vertical maps are given by the $t_{ij}$. They then define $[\Mcal^{\punc,t}_\Lambda(X^t)]^\mathsf{ref} = i^![\Acal]$, and $[\Mcal^\punc_{\Lambda}(X)]^\mathsf{ref} = \trop^{\punc,!} [\Mcal^{\punc,t}_\Lambda(X^t)]^\mathsf{ref}$.

In the language of (homological) piecewise polynomials, each $t_{ij}$ is a piecewise linear function on $\Mcal^{\punc,t}_\Lambda(X^t)$. The fibre square \ref{eq:bnrrefinedvfc} shows that the product $\prod t_{ij}$ is a homological piecewise polynomial, and in fact under the isomorphism $\PP_*(\cdot) = \LogCH_*(\cdot)$ from \ref{thm:prrshomPP} we see $\prod t_{ij} = [\Mcal^{\punc,t}_\Lambda(X^t)]^\mathsf{ref}$, where the product is taken over $(i,j)$ with $a_{ij} < 0$.

Write $\pi: \Mpt^t_\Lambda(X^t) \to \Mcal^{\punc,t}_\Lambda(X^t)$ for the forgetful map, and consider the fibre square \ref{eq:bnrcomparison}. Note that $[\Mcal^{\punc,t}_\Lambda(X^t)]^\mathsf{ref} = g^! \prod t_{ij}$, and hence
\begin{align*}
	\pi^{t,!} [\Mcal^{\punc,t}_\Lambda(X^t)]^\mathsf{ref} &= f^! \phi^! \prod t_{ij} \\
	&= f^! \prod_{ij} s_{ij} \prod_i \ell_i \\
	&= \prod_{ij} s_{ij} \prod_i \ell_i	\\
	&= \prod_{ij} r_{ij} \prod_i \ell_i \prod_{i, j : a_{ij} < 0} |a_{ij}|	\\
	&= [\Mpt_\Lambda^t(X^t)]^\prodvfc \prod_{i, j : a_{ij} < 0} |a_{ij}|.
\end{align*}
The proposition follows immediately from \ref{prop:comparisonacgs} and \ref{prop:log_prod_vir_divisibility}.
\end{proof}

\section{Birational invariance}
\label{sec:bir_inv}

The goal of this section is to establish birational invariance for pieced Gromov--Witten invariants along log blowups of log smooth targets. We will first establish the result on a tropical level, and then lift this to the logarithmic level.

\subsection{Tropical birational invariance}\label{sec:trop_bir_inv}
Let $\tilde X^t \to X^t$ be a subdivision, with induced map $h\colon \tilde X \coloneqq X \times_{X^t} \tilde X^t \to X$. Composition gives a natural map
\begin{equation}
	\Mpt^{t}(\tilde X, \tilde X^t) \to \Mpt^{t}(X, X^t), (\Gamma, f , \beth) \mapsto (\Gamma, h \circ f, \beth \circ h^*)
\end{equation}
and there is a corresponding map 
\begin{equation}
\Mptmt(\tilde X, \tilde X^t) \to \Mptmt(X, X^t)
\end{equation}lying over this. Both these maps take balanced objects to balanced objects; by \ref{lem:stabilisation} it is enough to check this for the map $\Mpt^{t}(\tilde X, \tilde X^t) \to \Mpt^{t}(X, X^t)$, and every sPL function on $X$ pulls back to an sPL function on $\tilde X$, so balancing for $\tilde X$ implies balancing for $X$. 
Composing with stabilisation gives a map
\begin{equation}
	\Mpt^{\mu, t, bal, st}(\tilde X, \tilde X^t) \to \Mpt^{\mu, t, bal, st}(X, X^t). 
\end{equation}

\begin{definition}\label{def:compatible_disc_data}
	Suppose we are given discrete data $\Lambda$ for $X$ and $\tilde \Lambda$ for $\tilde X$. We say these are \emph{compatible} if the genera and number of markings are equal, if $\tilde \beta |_{\on{Pic}(X)} = \beta$, and if for every marking $p_i$ the image cone $\tilde \sigma_i$ maps into $\sigma_i$, if 
	$\tilde \sigma_i^\gp$ contains $a_i$, and if $a_i$ is equal to $\tilde a_i$. 
\end{definition}

Let $\Lambda$ and $\tilde \Lambda$ be compatible discrete data for $X$ and $\tilde X$. Recalling from \ref{def:cpct_notation} that we set $\Mptmt_\Lambda(X) = \Mpt^{t, \on{bal}, \st, \mu}_\Lambda(X, X^t)$, we have a natural map 
\begin{equation}
\Mptmt_{\tilde\Lambda}(\tilde X) \to \Mptmt_\Lambda(X). 
\end{equation}

We remark that this is generally not a subdivision; the map on tropicalisations is generally not surjective. Our main tropical birational invariance result follows; the log version is \ref{thm:log_bir_inv}. 

\begin{theorem}\label{thm:tropical_birat_inv}
Suppose that $X \to X^t$ is smooth with connected fibres, and that both $X^t$ and $\tilde X^t$ are smooth. 
Let $\pi$ be the natural map of CSWB
\begin{equation}
\pi \colon \bigsqcup_{\tilde \Lambda \in \pi^*\Lambda} \Mptmt_{\tilde \Lambda}(\tilde X) \to \Mptmt_\Lambda(X), 
\end{equation}
where $\pi^*\Lambda$ is the set of all discrete data for $\tilde X$ that are balanced and compatible with $\Lambda$. 
Then $\pi^*\Lambda$ is finite, $\pi$ is a pushable map, and 
\begin{equation}
\pi_*\left[\bigsqcup_{\tilde \Lambda \in \pi^*\Lambda} \Mptmt_{\tilde \Lambda}(\tilde X)\right]^\vir = \left[\Mptmt_\Lambda(X)\right]^\vir. 
\end{equation}
\end{theorem}
We remark that this theorem is not true if we replace the classes $[...]^\vir$ with the classes $[...]^\prodvfc$ (see \ref{def:virtual_classes}), giving one source of motivation for the definition of the former. For the remainder of this section we maintain the assumption that $X \to X^t$ is smooth with connected fibres, and that both $X^t$ and $\tilde X^t$ are smooth. 

Before the proof (which will occupy the remainder of this section), we show in a simple example how the birational invariance manifests. 

\begin{example}
\label{ex:birational}
Let $X$ be the blowup of $\bb P^2$ in a point, and let the boundary consist of irreducible divisors $D_1$, $D_2$ where $D_1$ is the exceptional divisor, and $D_2$ is the strict transform of a line through the centre of the blowup. Then $X^t = \R_{\geq 0}^2$, and we define discrete data $\Lambda$ by setting $g = 0$, $n = 1$, $\beta$ to be the class of the exceptional divisor, and contact order $(-1,1)$. 

The moduli space $M \coloneqq \Mptmt_{\Lambda}(X^t)$ is given by $\R_{> 0} r \oplus \R_{>0} \ell$, with the corresponding tropical curve given by a single vertex at $(\ell+r,0)$, with a leg with endpoint $(r,\ell)$, pictured in the first part of \ref{fig:birat_eg}. On this space the tropical virtual fundamental class is given by the homological piecewise polynomial $r \ell$.

Now let $\pi: \tilde{X}^t\to X^t$ be the subdivision in the ray $(1,1)$ (the blow up in the intersection point $D_1 \cap D_2$). Then there are two lifts of the discrete data $\Lambda_1, \Lambda_2$, and two corresponding tropical moduli spaces are $M_1 = \R_{> 0} r \oplus \R_{> 0} (\ell - r)/2$ and $M_2 = \R_{> 0} (r-\ell) \oplus \R_{ > 0 } \ell$, pictured in \ref{fig:birat_eg}. Note that in $\Lambda_1$ the contact order to the rays generated by $(0,1),(1,1)$ is $(2,-1)$, and in $\Lambda_2$, the contact order to the rays generated by $(1,1),(1,0)$ is $(1,-2)$. Hence the virtual fundamental class of $M_1$ is $r(\ell-r)/2$, where the $2$ comes from the change of integral structure, and the virtual fundamental class of $M_2$ is $\ell(r-\ell)/2$, where the $2$ comes from the contact order $-2$ with respect to the ray $(1,0)$ whose dual  cuts out the wall $\langle (1,1)\rangle$ being crossed.

All in all, we have $\pi_* [M_1]^\vir = \pi_* [M_2]^\vir = \frac12 [M]^\vir$, and hence $\pi_* [M_1]^\vir + \pi_* [M_2]^\vir = [M]^\vir$.
\end{example}

\begin{figure}[htbp]
\centering
\begin{subfigure}[b]{0.3\textwidth}
\centering
\begin{tikzpicture}[scale=1.2]
    \draw[->] (0,0) -- (2.5,0);
    \draw[->] (0,0) -- (0,2.5);
    \fill (2,0) circle (2pt);
    \draw[->, thick] (2,0) -- (0.56,1.44);
    \draw (0.56,0) -- (0.56,-0.08);
    \node[below] at (0.56,-0.08) {$r$};
    \draw (0,1.44) -- (-0.08,1.44);
    \node[left] at (-0.08,1.44) {$\ell$};
    \draw (2,0) -- (2,-0.08);
    \node[below] at (2,-0.08) {$r+\ell$};
\end{tikzpicture}
\caption*{$M$}
\end{subfigure}
\hfill
\begin{subfigure}[b]{0.3\textwidth}
\centering
\begin{tikzpicture}[scale=1.2]
    \draw[->] (0,0) -- (2.5,0);
    \draw[->] (0,0) -- (0,2.5);
    \draw[gray] (0,0) -- (2.2,2.2);
    \fill (2,0) circle (2pt);
    \fill (1,1) circle (2pt);
    \draw[->, thick] (2,0) -- (0.56,1.44);
    \draw (0.56,0) -- (0.56,-0.08);
    \node[below] at (0.56,-0.08) {$r$};
    \draw (0,1.44) -- (-0.08,1.44);
    \node[left] at (-0.08,1.44) {$\ell$};
    \draw (2,0) -- (2,-0.08);
    \node[below] at (2,-0.08) {$r+\ell$};
    \draw (1,0) -- (1,-0.08);
    \node[below] at (1,-0.08) {$\frac{r+\ell}{2}$};
    \draw (0,1) -- (-0.08,1);
    \node[left] at (-0.08,1) {$\frac{r+\ell}{2}$};
\end{tikzpicture}
\caption*{$M_1$}
\end{subfigure}
\hfill
\begin{subfigure}[b]{0.3\textwidth}
\centering
\begin{tikzpicture}[scale=1.2]
    \draw[->] (0,0) -- (2.5,0);
    \draw[->] (0,0) -- (0,2.5);
    \draw[gray] (0,0) -- (2.2,2.2);
    \fill (2,0) circle (2pt);
    \draw[->, thick] (2,0) -- (1.4,0.6);
    \draw (1.4,0) -- (1.4,-0.08);
    \node[below] at (1.4,-0.08) {$r$};
    \draw (0,0.6) -- (-0.08,0.6);
    \node[left] at (-0.08,0.6) {$\ell$};
    \draw (2,0) -- (2,-0.08);
    \node[below] at (2,-0.08) {$r+\ell$};
\end{tikzpicture}
\caption*{$M_2$}
\end{subfigure}
\caption{Tropical curves corresponding to points in the moduli spaces $M$, $M_1$, $M_2$. The leg lengths are $\ell$, $(\ell-r)/2$, and $\ell$ respectively, where in the latter case we require $\ell < r$, i.e.\ $r - \ell >0$. }
\label{fig:birat_eg}
\end{figure}
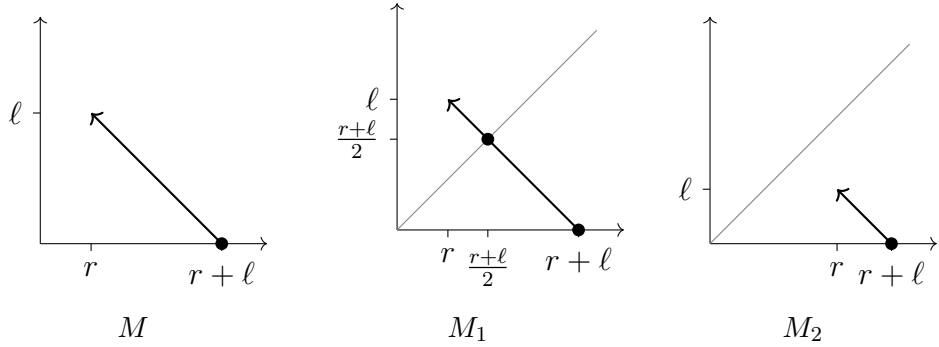

We begin with proving the finiteness of $\pi^*\Lambda$; first a lemma. 
\begin{lemma}\label{lem:pushout_pic}The natural commutative diagram
\begin{equation}
\begin{tikzcd}
  \on{Pic}(\tilde X) \ar[dr, phantom, very near start, "\ulcorner"]   & \on{Pic}(X) \ar[l]\\
PL(\tilde X^t)\ar[u]& PL(X^t)\ar[u]\ar[l]
\end{tikzcd}
\end{equation}
is a pushout diagram. 
\end{lemma}
 Without the assumption that $X \to X^t$ has connected fibres this is false. The assumption that $X$ is smooth can likely be weakened. 
\begin{proof}
Identifying Weil and Cartier divisors on $X$ and $\tilde X$, we have an exact sequence
\begin{equation}
0 \to \bigoplus_{1 \le i \le n} \bb Z E_i \to \on{Pic}(\tilde X) \to \on{Pic}(X) \to 0
\end{equation}
where $E_1, \dots, E_n$ are the irreducible components of codimension 1 of the exceptional locus of $\tilde X \to X$ (i.e.\ the locus where the map fails to be an isomorphism). Since $X \to X^t$ has connected fibres, the $E_i$ correspond bijectively to the rays in $\tilde X^t$ which map to higher-dimensional cones in $X^t$. Since both $X^t$ and $\tilde X^t$ are smooth, their rings of PL functions are naturally identified with the free abelian groups on the rays. 
\end{proof}

\begin{lemma}\label{lem:finite_disc_data}
	Fix balanced discrete data $\Lambda$ on $X^t$. Then there are only finitely many balanced discrete data $\tilde \Lambda$ on $\tilde X^t$ compatible with $\Lambda$; in other words, the set $\pi^*\Lambda$ from \ref{thm:tropical_birat_inv} is finite. 
\end{lemma}
\begin{proof}
	The genera and number of markings are determined. There are finitely many choices for each $\tilde \sigma_i$. After any such choice the $\tilde a_i$ are uniquely determined. For balanced data the global curve class is determined by the rest of the data by \ref{lem:pushout_pic}. 
\end{proof}

\begin{remark}\label{rem:hPP_push}
Let $f\colon X \to Y$ be a map of CSWB. If $f$ is a proper map, of relative log dimension 0 and of relative DM type, then a pushforward map of hPPs is constructed in \cite[Prop 67]{RPSS_log_taut}, and a formula is given in \cite[Prop 76]{RPSS_log_taut}. This map is compatible with the proper pushforward on the Chow group. We call such maps $f$ \emph{pushable}. We note that subdivisions, root stacks, and stratum inclusions are pushable, and compositions of pushable maps are pushable. 
\end{remark}

\subsubsection{Maps of cubes}
To prove \ref{thm:tropical_birat_inv} we will first require a more precise local description of the map we are pushing along, as a map of cubes. 
\begin{definition}
\label{def:cube}
Let $\sigma$ be a cone and $f\in \sigma^\vee$ be a non-negative linear function on $\sigma$, or $\infty$. Define the cone 
\begin{equation}
\sigma_f = \{(s, a) \in \sigma \times \bb Q_{\ge 0} : a \le f(s)\}, 
\end{equation}
with its forgetful map to $\sigma$. 
The dual of $\sigma_f$ is the sharpening of
\begin{equation}
\left( \sigma^\vee \oplus \bb N\right)  \langle (f, -1) \rangle \subseteq \sigma^\vee \oplus \bb Z; 
\end{equation}
if $f$ is non-zero then this monoid is already sharp.

For $\Sigma$ a cone stack and $f$ a non-negative strict piecewise linear function on $\Sigma$, or $\infty$, we define $\Sigma_f$ as the gluing of $\sigma_f$ for $\sigma \in \Sigma$.
\end{definition}
The cone stack $\Sigma_f$ comes with a projection map to $\Sigma$. We can think of it as an `interval of length $f$ over $\Sigma$;' in particular the fibre over a point $x \in \Sigma$ is the interval $[0, f(x)]$. From this perspective, given non-negative piecewise linear functions $f_1, \dots, f_n$ on $\Sigma$ we can think of the fibre product 
\begin{equation}
\Sigma_{f_1, \dots, f_n} \coloneqq \Sigma_{f_1} \times_\Sigma \Sigma_{f_2} \times_\Sigma \cdots \times_\Sigma \Sigma_{f_n}
\end{equation}
as an $n$-dimensional cube over $\Sigma$. 

Suppose now that we write each $f_i$ as a sum $f_i = f_{i, 1} + \cdots + f_{i, m_i}$ of non-negative linear functions. Then each interval $\Sigma_{f_i}$ is naturally the union of the $\Sigma_{f_{i, j}}$; more precisely there is a surjective map of cone complexes
\begin{equation}
\Sigma_{f_{i, 1}} \sqcup \dots \sqcup \Sigma_{f_{i, m_i}} \to \Sigma_{f_i}. 
\end{equation}
Similarly the $n$-cube $\Sigma_{f_1, \dots, f_n}$ is naturally a union of smaller $n$-cubes $\Sigma_{f_{1, j_1}, \dots, f_{n, j_n}}$; more precisely there is a surjective map of cone stacks
\begin{equation}\label{eq:n_cube_as_union}
\bigsqcup_{(j_1, \dots, j_n): 1 \le j_i \le m_i} \Sigma_{f_{1, j_1}, \dots, f_{n, j_n}} \to \Sigma_{f_1, \dots, f_n}; 
\end{equation}
we call such a map a \emph{decomposition} of the cube $\Sigma_{f_1, \dots, f_n}$. All of these $n$-cubes can naturally be equipped with the structure of CSWB, where we declare the boundary of the interval $\Sigma_f$ to be the endpoints of the interval: $\{(s, a) \in \Sigma \times \bb Q_{\ge 0} : a = 0 \text{ or } a = f(s)\}$, and we declare the interior of an $n$-cube to be the product of the interiors of its factors. One then quickly checks that a decomposition \ref{eq:n_cube_as_union} is a map of CSWB, and moreover is a pushable map. 

\subsubsection{Locally modelling as maps of cubes}

We will construct a moduli space $\bb M$ together with a map $f\colon \Mptmt_\Lambda(X) \to \bb M$ such that $f$ is locally modelled on $n$-cubes $\sigma_{f_1, \dots, f_n} \to \sigma$, and $$\pi\colon \bigsqcup_{\tilde \Lambda \in \pi^*\Lambda} \Mptmt_{\tilde \Lambda}(\tilde X) \to \Mptmt_\Lambda(X)$$ is locally modelled on decompositions of the cubes as in \ref{eq:n_cube_as_union}. 

Simply put, $\bb M$ is the analogue of the moduli space $\Mptmt_\Lambda(X)$ but where we forget the lengths of legs; the map $\Mptmt_\Lambda(X) \to \bb M$ is then given as $\sigma_{L_1, \dots, L_n} \to \sigma$ where $L_i$ is the distance in the direction of the $i$th leg to the boundary of its image cone; see \ref{const:ri}. Notice that stability, balancing and the `$\mu$' subdivision are all independent of the leg lengths, so this makes sense. Alternatively we can see $\bb M$ as the subdivision of the punctured tropical moduli space of \cite{Abramovich2020Punctured-logar} where we make each PL `distance to the boundary' function $r_i$ of \ref{const:ri} strict. The cone stack $\bb M$ is taken to have empty boundary. 

Note that the distance $L_i$ may be infinite, in the case where all contact orders are non-negative. This case is already treated in \cite{Abramovich2018Birational-inva}, and so below we focus mainly on the case of finite lengths, not making explicit when some terms may be infinite if $L_i$ is so.

Given a cone $\tau'$ in $\Mptmt_\Lambda(X)$ with image $\tau$ in $\bb M$, either $\tau'$ is completely contained in the boundary (in which case any homological PP vanishes on $\tau'$ by definition), or is equal to the $n$-cube $\tau_{L_1, \dots, L_n}$ described in the previous paragraph (since the cube is the fibre over $\tau$, and every proper face of the cube lies in the boundary). To prove \ref{thm:tropical_birat_inv} we are thus reduced to proving an equality of  hPP functions on the $n$-cube over some cone $\tau$ of $\bb M$. 

Now we define a map $\tilde{\bb M}\to \bb M$ as the coarsest subdivision where the distances to proper faces (in the sense of \ref{const:ri}) in $\tilde X$ are strict, and in which the combinatorial pattern of which edges cross which walls of $\tilde X^t$ are constant on each cone.

We fix a cone $\tau$ in $\bb M$ and a cone $\sigma$ in $\tilde{\bb M}$ mapping to it. Let $L_1, \dots, L_n$ be linear functions on $\sigma$ giving the distances to the boundary of the image cone in $X^t$, so that the pullback of $\sigma$ along $\Mptmt_\Lambda(X) \to \bb M$ is $\sigma_{L_1, \dots, L_n}$.

In other words, inside the diagram of spaces 
\begin{equation}\label{diagram_spaces}
\begin{tikzcd}
  &   \bigsqcup_{\tilde \Lambda} \Mptmt_{\tilde \Lambda}(\tilde X) \ar[dl, "\pi"] \ar[d, "f"]\\
\Mptmt_\Lambda(X) \ar[d]  & \Mptmt_\Lambda(X) \times_\bb M \tilde {\bb M} \ar[d]\ar[l]\\
\bb M & \ar[l] \tilde{\bb M}
\end{tikzcd}
\end{equation}
we have cones 

\begin{equation}
\begin{tikzcd}
  &   f^{-1}\sigma_{L_1, \dots, L_n} \ar[dl] \ar[d]\\
\tau_{L_1, \dots, L_n} \ar[d]  & \sigma_{L_1, \dots, L_n} \ar[d]\ar[l]\\
\tau & \ar[l] \sigma
\end{tikzcd}
\end{equation}
where $\sigma \to \tau$ is the inclusion of a sub-cone, and the source of each vertical arrow is the pullback of its target under the maps in \ref{diagram_spaces}.

A key observation is that the map $f^{-1}\sigma_{L_1, \dots, L_n} \to  \sigma_{L_1, \dots, L_n}$ has a special form; it is a decomposition of cubes, after taking a lattice refinement of the source. More precisely, for each $i$ write $L_i = L_{i, 1} + \cdots + L_{i, m_i}$ with $L_{i, j}>0$ where the $L_{i,j}$ are the lengths of the components that the $i$th leg is cut into by the subdivision $\tilde X^t \to X^t$. Then the map $f^{-1}\sigma_{L_1, \dots, L_n} \to  \sigma_{L_1, \dots, L_n}$ is a lattice refinement of the induced decomposition of cubes as in \ref{eq:n_cube_as_union}.

Now let $U = \bigoplus_{i=1}^n \bigoplus_{j=1}^{m_i} \R_{\geq 0} \cdot k_{i,j}$, admitting a natural map from $\sigma$ where each point is sent to the corresponding collection of $L_{i,j}$. For each $1 \le i \le n$ let $k_i = \sum_j k_{i,j}$, define a cube $V = U_{k_1, \dots, k_n}$, with decomposition $W = \bigsqcup_{(j_1, \dots, j_n): 1 \le j_i \le m_i} U_{k_{1, j_1}, \dots, k_{n, j_n}} \to V$. 

Then we can extend the above diagram to one where all small squares are pullbacks: 
\begin{equation}
\begin{tikzcd}
  &   f^{-1}\sigma_{L_1, \dots, L_n} \ar[dl] \ar[d] \ar[r]& W\ar[d]\\
\tau_{L_1, \dots, L_n} \ar[d]  & \sigma_{L_1, \dots, L_n} \ar[d]\ar[l] \ar[r]& V\ar[d]\\
\tau & \ar[l] \sigma \ar[r]& U
\end{tikzcd}
\end{equation}

Let $U_i = \bigoplus_{j=1}^{m_i} \R_{\geq 0} \cdot k_{i,j}$, so that $U = \prod_i U_i$, and let $V_i$ be the interval of length $\sum_j k_{i,j}$ over $U_i$, so that $V = \prod_i V_i$. On each $V_i$ we define an hPP function $p_i = b(\sum_j k_{i,j} - b)$ where $b$ is the coordinate on the fibre of $V_i \to U_i$. Then $p\coloneqq \prod_i p_i$ is an hPP function on $V$. From the definition of the virtual fundamental class we see
\begin{lemma}
The pullback of $p$ to $\sigma_{L_1, \dots, L_n}$ is equal to the pullback of the class $[\Mpt^t_\Lambda(X^t)]^\vir$. 
\end{lemma}

Now let $U_{i,j} = \R_{\geq 0} \langle k_{i,j} \rangle$, with $W_{i,j}$ the interval of length $k_{i,j}$, so that 
\begin{equation}
W_i \coloneqq \bigsqcup_{1 \le j \le m_i} W_{i, j}
\end{equation}
is a lattice refinement of a decomposition of $V_i$, and the decomposition $W \to V$ is given by
\begin{equation}
W = \prod_i W_i. 
\end{equation}
Define an hPP function $q_{i,j} = b(k_{i,j} - b)$ on $W_{i,j}$, which assemble to an hPP function $q_i$ on the decomposition $W_i$. Define $q = \prod_i q_i$ on $W$. Again we have
\begin{lemma}
The pullback of $q$ to $f^{-1}\sigma_{L_1, \dots, L_n}$ is equal to the pullback of the class $[\Mpt^t_{\tilde \Lambda}(\tilde X^t)]^\vir$. 
\end{lemma}

Compatibility of pushforward and pullback for Kresch's Chow groups is proven in \cite{Kresch1999Cycle-groups-fo}. The pushforward on hPPs is compatible with that on Chow, and the map from hPPs to Chow is injective by \ref{sec:push_hpp}, hence push-pull compatibility also holds for hPPs. Using this, \ref{thm:tropical_birat_inv} follows from the next lemma. 
\begin{lemma}
\label{lem:sumofcubesisbigcube}
The hPP $q$ on $W$ pushes forward to $p$ on $V$. 
\end{lemma}
\begin{proof}
Both $p$ and $q$ are products over $i$, and pushforward is compatible with (non-fibred) product, so we reduce to the case $n=1$. Since $p_1$ and $q_1$ are functions of degree 2, it suffices to check the equality on 2-dimensional cones of $U_1$. Both $p_1$ and $q_1$ vanish on any 2-dimensional cone that does not lie over a ray in $U_1$, so we reduce to checking this equality over a ray of $U_1$. But over any ray of $U_1$ only one of the $k_{1, j}$ is non-zero, and over such a ray $V_1 = W_1$ and $p_1 = q_1$. 
\end{proof}

\subsection{Logarithmic birational invariance}
Here we lift the tropical birational invariance statement to the logarithmic setting.

As before, we let $\tilde{X}/X$ be a subdivision, corresponding to the cartesian square
\begin{equation}
\label{eq:blowupsquare}
\begin{tikzcd}
\tilde{X} \arrow[r] \arrow[d] & \tilde{X}^t \arrow[d] \\
X \arrow[r] & X^t
\end{tikzcd}
\end{equation}
where $X \to X^t$ is strict and smooth with connected geometric fibres, and $X^t$ and $\tilde X^t$ are smooth. 
The diagram \ref{eq:potabstract} applied with $V/W$ any of the four morphisms above induces relative perfect obstruction theories for each of the morphisms
\[
\begin{tikzcd}
\Mpt_{\tilde{\Lambda}}(\tilde{X}) \arrow[r,"p_1"] \arrow[d,"f"] & \Mpt^t_{\tilde{\Lambda}}(\tilde{X}) \arrow[d,"g"] \\
\Mpt_{\Lambda}(X) \arrow[r,"q_1"] & \Mpt^t_\Lambda(X)
\end{tikzcd}
\]
As the square \ref{eq:blowupsquare} is cartesian, the square on moduli spaces is both cartesian and ``virtually cartesian'', in the sense that $\E_{p_1} = g^* \E_{q_1}, \E_f = q_1^* \E_g$. By \ref{thm:push_pull_compatibility} this means in particular that $f_* p_1^! = q_1^! g_*$.

\begin{theorem}\label{thm:log_bir_inv}
\begin{equation}
\sum_{\tilde \Lambda}f_*[\Mpt_{\tilde{\Lambda}}(\tilde{X})]^\vir = [\Mpt_{\Lambda}(X)]^\vir. 
\end{equation}
Here the (finite) sum is over balanced discrete data $\tilde \Lambda$ compatible with $\Lambda$. 
\end{theorem}
\begin{proof}
Consider the diagram
\[
\begin{tikzcd}
\bigsqcup_{\tilde \Lambda}\Mpt_{\tilde{\Lambda}}(\tilde{X}) \arrow[r,"p_1"] \arrow[d,"f"] & \bigsqcup_{\tilde \Lambda}\Mpt_{\tilde{\Lambda}}(\tilde{X},\tilde{X}^t) \arrow[d,"g"] \arrow[r,"p_2"] & \bigsqcup_{\tilde \Lambda}\Mpt_{\tilde{\Lambda}}^t(\tilde{X}) \arrow[d,"h"] \\
\Mpt_{\Lambda(X)} \arrow[r,"q_1"] & \Mpt_\Lambda(X,X^t) \arrow[r,"q_2"] & \Mpt_\Lambda^t(X,X^t)
\end{tikzcd}
\]
where both squares are cartesian. 
Recall that by definition 
\begin{equation}
[\Mpt_{\Lambda}(X)]^\vir = q_1^! q_2^! [\Mpt_\Lambda^t(X)]^\vir
\end{equation}
and 
\begin{equation}
[\Mpt_{\tilde{\Lambda}}(\tilde{X})]^\vir = p_1^!p_2^![\Mpt_{\tilde{\Lambda}}^t(\tilde{X})]^\vir. 
\end{equation}
By \ref{thm:tropical_birat_inv} we know that 
\begin{equation}
h_*\left[\bigsqcup_{\tilde \Lambda}\Mpt_{\tilde{\Lambda}}^t(\tilde{X})\right]^\vir = [\Mpt_\Lambda^t(X)]^\vir, 
\end{equation}
so given the observation above that, by \ref{thm:push_pull_compatibility}, we have $f_* p_1^! = q_1^! g_*$, it suffices to show that $g_*p_2^! = q_2^!h_*$. This follows from the fact that $q_2$ is smooth and the right hand square in the above diagram is cartesian.
\end{proof}

\section{Gluing formula}
\label{sec:gluing}

In this section we will give a recursive formula for the pullback of pierced logarithmic Gromov--Witten invariants along gluing maps. In this section we assume that the target $X$ is a proper smooth log smooth SNC log scheme.

\begin{situation}\label{sit:gl}
We fix discrete data $\Lambda$ with genus $g$ and $n$ marked points. We choose a gluing map $\gl$, either the separating gluing map
\[
	\gl: \Mpt_{g_1,n_1 + 1} \times \Mpt_{g_2,n_2+1} \to \Mpt_{g,n}
\]
or the non-separating gluing map
\[
	\gl: \Mpt_{g-1,n+2} \to \Mpt_{g,n}
\]
\end{situation}
We will give several formulas for the pullback along $\gl$ of the virtual fundamental class of $\Mpt_{\Lambda}(X)$. At the end of this subsection we briefly explain and point to the various tropical and logarithmic theorems we prove.

\begin{remark}
In \cite{Spelier2025SplittingFormulaLogDR} the non-separating gluing pullback is computed for the logarithmic double ramification cycle, roughly giving a loop axiom in this situation. 
\end{remark}

The key geometric input for the recursive formula for classical Gromov--Witten invariants in the separating case is the following fibre diagram
\[
	\begin{tikzcd}[column sep = tiny]
		& {\Mbar_{g,n}(X;\beta)}\ar[dl] & D  \ar[l]\ar[rr]\ar[dl]\ar[d]&& \bigsqcup_{\beta_1 + \beta_2 = \beta} {\Mbar_{g_1,n_1+1}(X;\beta_1) \times \Mbar_{g_2,n_2+1}(X;\beta_2)} \ar[d]\\
		\Mfrak & {\Mfrak_1 \times \Mfrak_2}\ar[l, "\gl"] & X\ar[rr, "\Delta_X"] && {X \times X}
	\end{tikzcd}
\]
and similarly for the non-separating case.

This diagram follows directly from the coproduct universal property for the glued curve. Note that the disjoint union has only finitely many non-zero terms, as $\beta_1,\beta_2$ both need to be effective (apply the N\'eron-Severi Theorem). Classically, one then has access to the Gysin pullbacks $\gl^!$ and $\Delta_X^!$, and one can prove the equality \begin{equation}\label{eq:classical}\gl^![\Mbar_{g,n}(X;\beta)]^\vir = \sum_{\beta_1 + \beta_2 = \beta} \Delta_X^! {[\Mbar_{g_1,n_1+1}(X;\beta_1)]^\vir \times [\Mbar_{g_2,n_2+1}(X;\beta_2)]^\vir}\end{equation} for the separating gluing map and \begin{equation}\label{eq:classicalnonsep}\gl^![\Mbar_{g,n}(X;\beta)]^\vir = \Delta_X^! {[\Mbar_{g-1,n+2}(X)]^\vir}\end{equation} for the non-separating gluing map.

Logarithmically, using the universal property of the gluing of pierced log curves we construct a similar diagram of log fibre squares in \ref{sec:glue:discdata}.
\begin{equation}\label{eqn:loggluing}
	\begin{tikzcd}[column sep = tiny]
		& {\Mpt_{\Lambda}(X)}\ar[dl] & D  \ar[l]\ar[rr]\ar[dl]\ar[d]&& \bigsqcup_{(\Xi_+,\Xi_-) \in \gl^*\Lambda} {\Mpt_{\Xi_+}(X) \times \Mpt_{\Xi_-}(X)} \ar[d]\\
		\Mpt & {\Mpt_1 \times \Mpt_2}\ar[l, "\gl"] & E\ar[rr, "\Delta_X"] && {X \times X}
	\end{tikzcd}
\end{equation}
and similarly in the non-separating case. Here $\gl^* \Lambda$ is a finite set of discrete data, defined in \ref{sec:glue:discdata}.

There are three main differences in this logarithmic diagram, and hence in our proof strategy.
\begin{enumerate}
	\item The splitting of the discrete data is slightly different; $\gl^* \Lambda$ (defined below in \ref{def:pullbackdiscretedata}) parametrises possible ways of splitting the curve class \emph{and} assigning tangency orders. In the separating case the splitting of the curve class will determine the tangency orders through balancing. In the non-separating case there are many possible splittings for the contact order. We treat this in \ref{sec:glue:discdata}.
	\item The diagram \ref{eqn:loggluing} does not capture the whole story; we need to take into account the corresponding diagrams on the level of tropical moduli spaces $\Mbar_{\Lambda^t}(X^t)$ and $\Mbar_{\Lambda^t}^t(X^t)$. We write down the full diagram we need in \ref{sec:glue:diagram}.
	\item The right hand side fiber square in \ref{eqn:loggluing} is a fiber square of log stacks, and not necessarily a fiber square on the level of underlying algebraic stacks. This means that even defining $\Delta_X^!$ is not trivial.
\end{enumerate}

We present several solutions that give formulas in different cases: work purely tropically in \ref{sec:glue:trop}, lift this to a logarithmic formula in \ref{sec:glue:log}, give a refined formula in the equivariant setting in \ref{subsec:equiv}, and in the separating gluing case work with explicit log modifications in \ref{subsubs:sepcasetrop,sec:glue:logsep} to circumvent issue (3) above.

\subsection{Splitting of the discrete data}
\label{sec:glue:discdata}

In this section we describe the possible discrete data appearing in $\gl^* \Mpt_\Lambda(X)$. In order to treat the separating case and non-separating case the same, we write $\Mpt_\Gamma$ for the source of the gluing map, either $\Mpt_{g_1,n_1+1} \times \Mpt_{g_2,n_2+1}$ or $\Mpt_{g-1,n+2}$.

The gluing maps of \ref{sit:gl} descend naturally to the tropical moduli spaces. The space
$\gl^* \Mpt_{\Lambda}^t(X^t)$ admits a map
\[
	s_0: \gl^* \Mpt_{\Lambda}^t(X^t) \to \Z^k,
\]
recording the slope of the glued edge (oriented from the first marking to the second marking). This map is locally constant, and hence we have
\[
	\gl^* \Mpt_{\Lambda}^t(X^t) = \sqcup_{b \in \Z^k} s_0^{-1}(b).
\]
as cone stacks with boundary. For $b \in \Z^k$, we consider the intersection of boundary divisors $\bigcap_{i: b_i \ne 0} D_i$. Because $X$ is strict normal crossings this intersection is a disjoint union of smooth strata corresponding to cones $\sigma$ of $X^t$. Hence, keeping track of which of these smooth strata we land in, we obtain a locally constant refinement of $s_0$ mapping to the set of contact orders $\CO(X)$ from \ref{def:contact_orders}
\[
	s_1: \gl^* \Mpt_{\Lambda}^t(X^t) \to \CO(X).
\]

Recall that $\Disc(X)$ is the set of possible discrete data, from \ref{def:discrete-data}, collecting the genus, number of markings, curve class, and contact orders.

\begin{definition}
\label{def:pullbackdiscretedata}
Take $\gl$ the separating gluing map. We define the map
\[
s : \gl^* \Mpt_{\Lambda}(X) \to \Disc(X) \times \Disc(X)
\]
by sending an element $(C,(p_1,\dots,p_n),f)$ that is the gluing of $(C_1,(p_1,\dots,p_{n_1},q_1),f_1)$ and $(C_2,(p_{n_1+1},\dots,p_n,q_2),f_2)$ to $(\Xi_+,\Xi_-)$ with $\Xi_i$ the discrete data associated to $f_i$. Explicitly, we have that the contact order at the marking $q_i$ is given by $(\sigma,(-1)^{i+1} b)$ where $(\sigma,b) =  s_1(\trop(C,(p_1,\dots,p_n),f))$.

Take $\gl$ the non-separating gluing map. We define the map
\[
s : \gl^* \Mpt_{\Lambda}(X) \to \Disc(X)
\]
by sending an element $(C,(p_1,\dots,p_n),f)$ that is the gluing of $(C_0,(p_1,\dots,p_n,q_1,q_2),f_0)$ to $\Xi$, the discrete data associated to $f_0$. Explicitly, we have that the contact order at the marking $q_i$ is given by $(\sigma,(-1)^{i+1} b)$ where $(\sigma,b) =  s_1(\trop(C,(p_1,\dots,p_n),f))$.

In both cases, we denote by $\gl^* \Lambda$ the image of $s$. For $\Xi \in \gl^* \Lambda$ (either a pair of discrete data, or discrete data), write
\[
	\Mpt_{\Lambda|\Xi}(X)
\]
for $s^{-1}(\Xi)$.

If $\gl$ is the separating gluing map, given $\Xi = (\Xi_+,\Xi_-) \in \gl^*\Lambda$ we write $\Mpt_{\Xi}(X)$ for $\Mpt_{\Xi_+}(X) \times \Mpt_{\Xi_-}(X)$.
\end{definition}

In the same way, we define the corresponding moduli spaces $\Mpt_{\Lambda|\Xi}(X^t)$ and $\Mpt_{\Lambda|\Xi}^t(X^t)$ and $\Mpt_{\Xi}(X)$ and $\Mpt_\Xi^t(X^t)$.

These sets of discrete data $\gl^* \Lambda$ look very different for the separating and the non-separating gluing map, as we now describe.

\begin{remark}
\label{rem:sepglLambda}
In the case of the separating gluing map, the balancing condition and the splitting of the curve class $\beta$ into $\beta_+ + \beta_-$ will uniquely determine the tangency vector $b \in \Z^k$ of the glued edge with each boundary divisor. Hence an element $(\Xi_+,\Xi_-) \in \gl^* \Lambda$ is uniquely determined by the splitting into $\beta_++\beta_-$ and the choice of a connected component of $\bigcap_{i: b_i \ne 0} D_i$.
\end{remark}

For the non-separating case, we present the following example.
\begin{example}
Take $X = \P^1$ with log structure given by $0 \in \P^1$, and $\Lambda$ given by $g = 1, n = 2, \beta = 0$, with the two tangencies to $0$ being $(3,-3)$. Then for the non-separating gluing map $\Mpt_{0,4}\to \Mpt_{1,2}$ we have that $\gl^* \Lambda$ has elements $\Xi_{b}$ for $-2 \leq b \leq 2$ with $\Xi_b$ having the four tangencies to $0$ being $(3,-3,b,-b)$.
\end{example}

\begin{proposition}
\label{prop:pullbackdisjointunion}
The pullback $\gl^* \Mpt_\Lambda(X)$ splits as a disjoint union
\[
	\bigsqcup_{\Xi \in \gl^* \Lambda} \Mpt_{\Lambda|\Xi}(X)
\]
\end{proposition}
\begin{proof}
It suffices to show the function $s$ is locally constant on $\gl^* \Mpt_\Lambda(X)$. We know $s_1$ is locally constant on $\gl^* \Mpt_{\Lambda}^t(X^t)$. In the non-separating case $s$ is uniquely determined by $s_1$, as the curve class of $\Xi$ is always the same as the curve class of $\Lambda$. In the separating case, we know $f_{i,*} [C_i]$ is locally constant, hence the statement follows. 
\end{proof}

With this notation, we obtain the following diagram
\begin{equation}\label{eqn:loggluing2}
	\begin{tikzcd}[column sep = tiny]
		& {\Mpt_{\Lambda}(X)}\ar[dl] & \bigsqcup_{\Xi \in \gl^*\Lambda} \Mpt_{\Lambda|\Xi}(X) \ar[l]\ar[rr]\ar[dl]\ar[d]&& \bigsqcup_{\Xi \in \gl^*\Lambda} \Mpt_{\Xi}(X) \ar[d]\\
		\Mpt_{g,n} & \Mpt_\Gamma \ar[l, "\gl"] & X\ar[rr, "\Delta_X"] && {X \times X}
	\end{tikzcd}
\end{equation}

By properness of $\Mpt_{\Lambda}(X)$, we immediately obtain the following finiteness result.
\begin{proposition}
The set $\gl^* \Lambda$ is finite.
\end{proposition}

This argument by properness is not constructive. However, we can make the argument constructive. In the separating case, the description in \ref{rem:sepglLambda} explicitly shows $\gl^* \Lambda$ is finite. In the non-separating case, we can obtain an explicit upper bound on the new contact orders in $\Xi$, by using the balancing equation to show the tropical moduli space $\Mpt_{\Lambda|\Xi}^t(X^t)$ vanishes for large contact orders.
\begin{proposition}
Let $\gl$ be the non-separating gluing map, and take $\Xi \in \gl^* \Lambda$. Write $a_1,\dots,a_n,b,-b \in \Z^k$ for the tangencies of the $n+2$ points against the $k$ boundary divisors.
The moduli space $\Mpt_{\Lambda|\Xi}^t(X^t)$ is empty if there is an $i \in \{1,\dots,k\}$ such that \[|b_i| > |D_i \cdot \beta| + \sum_{j: a_{ji} > 0} |a_{ji}|.\]
\end{proposition}
\begin{proof}
The proof is the same as the proof of \cite[Lemma~4.15]{Spelier2025SplittingFormulaLogDR}.
\end{proof}

\subsection{The gluing diagram}
\label{sec:glue:diagram}

In this section we refine the diagram \ref{eqn:loggluing}, by including the moduli spaces $\Mpt_{\Lambda}(X^t)$ and $\Mpt_{\Lambda}^t(X^t)$. We fix $\Xi \in \gl^* \Lambda$.

\begin{definition}
Let $M$ denote the fibre product
\[
	\begin{tikzcd}
		M \ar[d]\ar[r]& \ar[d] \Mpt_\Xi(X) \\
		\Mpt_{\Lambda|\Xi}(X^t) \ar[r] & \Mpt_\Xi(X^t) \\
	\end{tikzcd}
\]
\end{definition}
This is the locus of maps to $X$ with discrete data $\Xi$ together with a gluing of the tropicalisation of the maps. Again, if $\gl$ is the separating gluing map, then $\Xi$ is a pair $(\Xi_+,\Xi_-)$ of discrete data, and $\Mpt_\Xi(X)$ is shorthand for $\Mpt_{\Xi_+}(X) \times \Mpt_{\Xi_-}(X)$.

With this definition, we can construct \ref{eqn:biggluing}. 

\begin{figure}
    \centering
    \begin{tikzcd}[column sep=tiny]
	&\Mpt^t&\Mpt_\Gamma^t \ar[l, swap, "\gl_\trop"]& X^t \ar[rr, "\Delta_{X^t}"] && X^t \times X^t &&\\
	&&\Mpt^{t}_{\Lambda}(X^t) \ar[ul] & \Mpt^{t}_{\Lambda|\Xi}(X^t) \ar[u]\ar[rr] \ar[l]\ar[ul] \arrow[ull, phantom, "\boxtimes_1"] \arrow[urr, phantom, "\square_2"]&& \Mpt^{t}_\Xi(X^t)\ar[u]&&\\
	&& {\Mpt_\Lambda(X^t)} \ar[u, "s"] & \Mpt_{\Lambda|\Xi}(X^t) \ar[l]\ar[rr]\ar[u, "s_\Gamma"] \arrow[urr, phantom, "\boxtimes_3"]
	&& \Mpt_\Xi(X^t) \ar[u, "s_\Xi"]&&\\
	&&&&M\ar[ul] \ar[dr]\arrow[ur, phantom, "\boxtimes_4"]&&&\\
	&& \Mpt_\Lambda(X)\ar[uu, "P"]\ar[dl] & \Mpt_{\Lambda|\Xi}(X)  \ar[ur]\ar[l]\ar[uu]\ar[rr]\ar[dl]\ar[d] \arrow[uul, phantom, "\boxtimes_1"]\arrow[dll, phantom, "\boxtimes_1"]\arrow[drr, phantom, "\square_5"]&& \Mpt_\Xi(X) \ar[uu, "P_{\Xi}"]\ar[d]&&\\
	&\Mpt & {\Mpt_\Gamma}\ar[l, "\gl"] & X\ar[rr, "\Delta_X"] && {X \times X}&&
\end{tikzcd}
    \caption{The gluing diagram. The left side corresponds to splitting a curve by its tropical type and the right side corresponds to gluing a pair of curves at marked points with the same image in $X$. }
    \label{eqn:biggluing}
\end{figure}
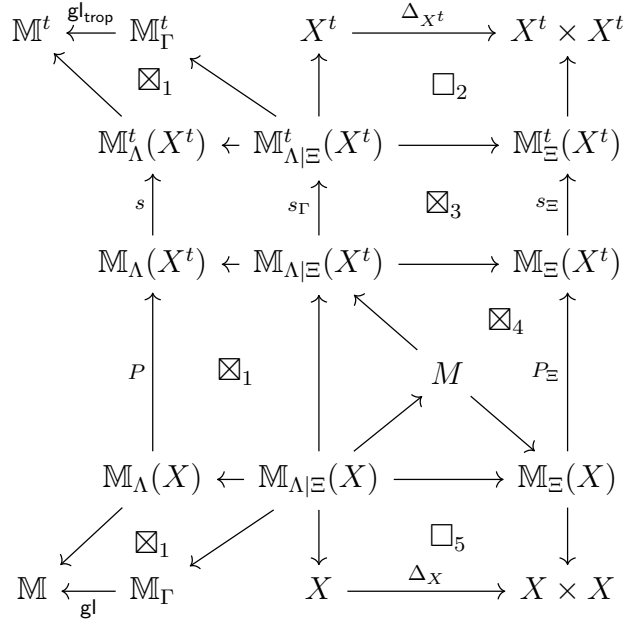

\begin{proposition}
After taking disjoint unions over contact orders $\Xi$ in the right and middle columns, all the marked squares in \ref{eqn:biggluing} are fibre squares, and all the fibre squares marked $\boxtimes$ are fibre squares on the level of underlying algebraic stacks. 
\end{proposition}
\begin{proof}
The squares marked $\boxtimes_1, \square_2, \boxtimes_3$ and $\square_5$ are fibre squares exactly by the universal property of gluing (tropical) curves \cite[Theorem~5.10,Theorem~B.12]{Holmes2023LogarithmicCohomologicalFT}. The square marked $\boxtimes_4$ is a fibre square by definition.

A log fibre product $A \times_C B$ is equal to the fibre product of underlying algebraic stacks if at least one of the maps $A \to C, B \to C$ is integral and saturated. In particular, this is the case if at least one of the maps is strict.

For the squares marked $\boxtimes_1$, we note that the (tropical) gluing map is the composition of a strict closed embedding, and (the tropicalisation of) the integral and saturated map $(\pt,\N)^2 \to (\pt,\N)$. For the squares marked $\boxtimes_3$ and $\boxtimes_4$, both the maps $s_1 \times s_2$ and $P_1 \times P_2$ are strict.
\end{proof}

\begin{remark}
The unmarked square in \ref{eqn:biggluing} is \emph{not} a fibre square. This is because the diagram
\[
\begin{tikzcd}
\mathbb{M}_\Gamma \arrow[r, "\gl"] \arrow[d] & \mathbb{M} \arrow[d] \\
\mathbb{M}_\Gamma^\trop \arrow[r, "\gl_\trop", swap] & \mathbb{M}^\trop
\end{tikzcd}
\]
is not a fibre square. However, all these maps are log lci, and hence give a pullback on log Chow groups by \ref{def:loglcipullback}.
\end{remark}

\begin{remark}\label{rk:not_fibre}
The two squares $\square_2,\square_5$ in \ref{eqn:biggluing} are \emph{not} necessarily a fibre square on the level of underlying algebraic stacks. In the separating case, $\square_2$ is a scheme-theoretic fibre square if the two evaluation maps are integral and saturated. They are integral if the log structure on $X$ is valuative (equivalently, the boundary is smooth), but even then are often not saturated. This introduces the usual multiplicities in the classical degeneration formula \cite{Li2002A-degeneration-}.

In the non-separating case, $\square_2$ is almost never a scheme-theoretic fibre square (unless all tangencies at the two nodes are $0$), even up to root stacks.
\end{remark}

\subsection{Tropical gluing}
\label{sec:glue:trop}
In this section we prove a tropical gluing formula, writing $\gl^! [\Mpt_{\Lambda}^t(X^t)]^\vir$ in terms of pullbacks of smaller virtual fundamental classes. We do this for both the separating and non-separating gluing map.

Recall that classically, one has the formula
\[\gl^![\Mbar_{g,n}(X;\beta)]^\vir = \sum_{\beta_1 + \beta_2 = \beta} \Delta_X^! [\Mbar_{g_1,n_1+1}(X;\beta_1)]^\vir \boxtimes [\Mbar_{g_2,n_2+1}(X;\beta_2)]^\vir,\] and similarly for the non-separating pullback.

We first state a tropical analogue of this theorem, involving two Gysin pullbacks. We define these Gysin pullbacks in \ref{sec:log_vir_pull} and prove the theorem at the end of this section. For the separating pullback and $\Xi = (\Xi_+,\Xi_-) \in \gl^* \Lambda$, we write \[[\Mpt_{\Xi}^t(X^t)]^\vir = [\Mpt_{\Xi_+}^t(X^t)]^\vir \boxtimes [\Mpt_{\Xi_-}^t(X^t)]^\vir.\]

\begin{theorem}
\label{thm:tropicalgluingsep}
We have equalities in $\LogCH_*(\bigsqcup_{\Xi \in \gl^* \Lambda} \Mpt_{\Lambda|\Xi}^t(X^t))$: 
\[s_+ s_- \gl^! [\Mpt_{\Lambda}^t(X^t)]^\vir = \sum_{\Xi \in \gl^* \Lambda} c_+^{\Xi} c_-^{\Xi} \Delta_{X^t}^! [\Mpt_{\Xi}^t(X^t)]^\vir\]
where $s_{\pm}$ are everywhere positive piecewise polynomial functions, and $c_\pm^{\Xi}$ are positive integers, both defined below. Furthermore, this equality determines $\gl^! [\Mpt_{\Lambda}^t(X^t)]^\vir$ uniquely.
\end{theorem}

We now fix a $\Xi \in \gl^* \Lambda$, and look at the clopen component $\Mpt_{\Lambda|\Xi}^t(X^t)$ of $\gl^* \Mpt_{\Lambda}^t(X^t)$.

We write $b \in \Z^k$ for the vector of tangencies of the first marking we are gluing (equivalently, the slopes along the glued edge, oriented from the first glued marking to the second glued marking).

On $\Mpt_{\Xi}$ we write $\ell_+, \ell_-$ for the lengths of the two markings we are gluing. Similarly, we write $r_\pm$ for the piecewise linear functions defined in \ref{const:ri} at these two markings, and $\ev_\pm$ for the evaluation map corresponding to these two markings.

\begin{remark}
\label{rem:tangencyallpositive}
In the case where the tangency vector $b$ is all non-negative, then $r_+$ is the constant function $1$ instead of a piecewise linear function, and similarly for $r_-$. For simplicity, we assume in the following that $r_\pm$ are actually piecewise linear. 
\end{remark}

We write $\tilde{X}^t$ for a Radon subdivision at $b$, defined in \ref{sec:radonsubdiv}. The main property important for us is that the strict piecewise linear functions $r_\pm$ on $\Mptmt_{\Xi_i}(X^t)$ as defined in \ref{const:ri} are pullbacks along the evaluation maps of piecewise linear functions on $X$ that are strict on $\tilde{X}$. In particular, there are $c_\pm^{\Xi} \in \Z_{>0}$ such that $c_\pm^{\Xi} r_\pm = \ev_{\pm}^* s_\pm$ where $s_\pm$ are the characteristic functions of the two Radon rays $\rho_\pm$, corresponding to the negative resp.\ positive parts of $b$.

\begin{proof}[Proof of \protect{\ref{thm:tropicalgluingsep}}]
By \ref{prop:pullbackdisjointunion}, we can fix a single $\Xi \in \gl^* \Lambda$. Recall that $\Mpt_{\Lambda|\Xi}^t(X^t)$ denotes the component of $\gl^! \Mpt^t_\Lambda(X^t)$ whose discrete data on the partially normalised curve is $\Xi$. Then we need to prove that the virtual fundamental class on $\Mpt_{\Lambda|\Xi}^t(X^t)$ induced by restricting $\gl^! [\Mpt_{\Lambda}^t(X^t)]^\vir$ to $\Mpt_{\Lambda|\Xi}^t(X^t)$ equals the corresponding term on the right hand side.

By \ref{thm:prrshomPP}, we can interpret both terms as homological piecewise polynomials. We first compute the gluing pullback. We can factor the gluing map as
\[
	\gl: \Mpt_{g_1,n_1+1}^t \times \Mpt_{g_2,n_2+1}^t \xrightarrow{\gl_0} T \xrightarrow{\gl_1} \Mpt_{g,n}^t 
\]
(and similarly in the non-separating case) where $\gl_1$ is the base change along the log smooth map $\R_{\geq 0}^2 \to \R_{\geq 0}$ and $\gl_0$ is the base change along the strict codimension $2$ regular immersion $\R_{>0}^2 \to \R_{\geq 0}^2$. Using \ref{prop:potpullbackfunctorial} we see $\gl^!$ is multiplication by $\ell_+\ell_-$, the lengths of the two glued legs.

Recall that $[\Mpt_{\Lambda}^t(X^t)]^\vir = \prod_{i=1}^n \ell_i r_i$. Then we obtain
\[
	\gl^! [\Mpt_{\Lambda}^t(X^t)]^\vir = \ell_{+} \ell_{-} \prod_{i=1}^n \ell_i r_i
\]
restricted to this component.
On the right hand side, we get $$\Delta_{X^t}^!( r_{+}r_{-}\ell_{n+1}\ell_{n+2} \cdot \prod_{i=1}^n \ell_i r_i).$$ The function $\Delta_{X^t}^!$ is given by pullback of functions, and the result follows immediately from the equality $c_\pm^{\Xi} r_\pm = \ev_{\pm}^* s_\pm$.

The fact that $\gl^! [\Mpt_{\Lambda}^t(X^t)]^\vir$ is uniquely determined by the equality follows from $s_+ s_-$ being everywhere positive.
\end{proof}

\subsubsection{Tropical gluing in the separating case}
\label{subsubs:sepcasetrop}
The extra factor $s_+ s_-$ on the left hand side is roughly because we needed a reduced pullback along the map $\Delta_{X^t}$, instead of $\Delta_{X^t}^!$. We explain this now.

The Radon subdivision is simplicial by \ref{prop:radonsubdiv}, and hence we can write $\tilde{X}$ for the corresponding smooth toroidal orbifold. We write $E_\pm \subset \tilde{X}$ for the divisors corresponding to the two Radon rays in $\tilde{X}^t$. These correspond to the negative part of the tangency and the positive part of the tangency respectively. We write $E = E_{\Xi} =  E_+ \cap E_-$.

Now the two markings we glue are already forced by their contact orders to lie in (a blow down of) $E^t$. Then the pullback $\Delta_{X^t}^!$ forces the two points to additionally be equal; intersection-theoretically, by intersecting with the zero section of the tangent complex of $X^t$. But if they're already known to lie in $E^t$, then we already know that we end up inside the tangent complex of $E^t$ sitting inside of the tangent complex of $X^t$, and hence we need only to intersect with the zero section of the tangent complex of $E^t$. Hence we would like to substitute $\Delta_{X^t}^!$ with $\Delta_{E^t}^!$.

There is a problem here: $E$ is smooth but fails to be \emph{log} smooth, and hence there is no known procedure for defining $\Delta_{E}^!$ as a bivariant operator on logarithmic intersection theory. Tropically, one could \emph{define} $\Delta_{E}^!$ as \[\frac{1}{s_+s_-} \Delta_{X^t}^!: \PP_*( \Mpt_{\Xi}^t(X^t) ) \to \PP_*(\Mpt_{\Lambda|\Xi}^t(X^t))[s_\pm^{-1}];\]
then one would arrive at the following version of \ref{thm:tropicalgluingsep}:
\[\gl^! [\Mpt_{\Lambda}^t(X^t)]^\vir = \sum_{\Xi \in \gl^* \Lambda} c_+^{\Xi} c_-^{\Xi} \Delta_{E_{\Xi}^t}^! [\Mpt_{\Xi}^t(X^t)]^\vir.\]

One can lift this to the logarithmic world in the equivariant setting of \ref{subsec:equiv}, where dividing by $s_+ s_-$ is still allowed. However, outside of this setting, $s_+, s_-$ pull back to nilpotent elements in $\LogCH^*(\M_\Lambda(X))$. We will now work, for the separating gluing map, towards a tropical gluing formula that will always lift to the logarithmic setting. The key to this is the following lemma.

\begin{lemma}
\label{lem:deltaet}
Consider a blowup $\hat{\Mpt}^t = \hat{\Mpt}_+^t \times \hat{\Mpt}^t_-$ of $\Mpt_\Xi^t(X^t) = \Mpt_{\Xi_+}^t(X^t) \times \Mpt_{\Xi_-}^t(X^t)$ such that the evaluation maps $\ev_\pm$ induce integral and saturated maps $\ev_\pm: \Mpt_\pm^t \to E^t$.

Consider the diagram of logarithmic fibre squares
\[
\begin{tikzcd}
 X^t \ar[r]  & \ar[d] X^t\times X^t  \\
 E^t \ar[r] \ar[u] & \ar[u] E^t \times E^t \\
 \hat{\Mpt}^t \times_{E^t \times E^t} E^t \ar[r, "\delta"] \ar[u] & \hat{\Mpt}^t \ar[u,"\ev_+ \times \ev_-"]
\end{tikzcd}
\]
Then the bottom square is a scheme-theoretic fibre square, and we have $s_+ s_- \Delta_{\ul{E^t}}^! = \Delta_{X^t}^!$.
\end{lemma}
\begin{proof}
Each of $\ev_+$ and $\ev_-$ is integral and saturated. Since $\hat{\Mpt}^t$ splits as a product, so is the map $\ev_+ \times \ev_-$.\footnote{This is what fails for the non-separating gluing map: both evaluation maps can be integral and saturated, but that does not mean the tuple $(\ev_+, \ev_-)$ is integral and saturated.}

Then the excess intersection formula gives that $c_{\text{top}}(N) \Delta_{\ul{E^t}}^! = \Delta_{X^t}^!$ where $N$ is the normal bundle of $E^t$ in $X^t$. But $E^t$ is the codimension $2$ strict regular embedding cut out by $s_\pm > 0$, hence $c_{\text{top}}(N) = s_+ s_-$.
\end{proof}

This immediately gives us the following corollary of \ref{thm:tropicalgluingsep}.
\begin{corollary}
Let $\gl$ denote the separating gluing map. Assume we have integral and saturated evaluation maps $\ev_{\pm}: \Mpt_{\Xi}^t(X^t) \to E_\Xi$ for all $\Xi \in \gl^* \Lambda$. Then we have
\[
\gl^! [\Mpt_{\Lambda}^t(X^t)]^\vir = \sum_{\Xi \in \gl^* \Lambda} c_+^{\Xi} c_-^{\Xi} \Delta_{\ul{E_\Xi^t}}^! [\Mpt_{\Xi}^t(X^t)]^\vir
\]
\end{corollary}

However, we are not always in this case, and in \ref{rem:uglycase} we shall see that blowing up $X$ does not necessarily resolve this. However, we will use \ref{lem:deltaet} to prove a different tropical gluing formula in the separating case. This gluing formula is more explicit: it takes place in $\CH_*\left(\bigsqcup_{\Xi} \Mptmt_{\Lambda|\Xi}(X^t)\right)$ and computes $\gl^! [\Mptmt_{\Lambda}(X^t)]^\vir$ directly, instead of its product with $s_+ s_-$. This formula will lift to a logarithmic gluing formula no matter the setting, but also necessitates more log blowups, to make the evaluation maps integral and saturated.

We keep the notation from before. We again fix $\Xi \in \gl^* \Lambda$. Note that the two evaluation maps $\ev_\pm: \Mpt_{\Xi_i}^t(X^t) \to X^t$ need not be integral and saturated. We can take simplicial log alterations $\pi: \hat{X}^t/X^t$ and $\hat{\Mpt}_{\Xi_i}^t(X^t)/\Mpt_{\Xi_i}^t(X^t)$ such that we have a diagram
\[
\begin{tikzcd}
	 \hat{\Mpt}_{\Xi_i}^t(X^t) \ar[r, "\ev_\pm"] \ar[d] & \ar[d] \hat{X^t} \\
  \Mpt_{\Xi_i}^t(X^t)	\ar[r] & X^t 
\end{tikzcd}
\]
with the map $\ev_\pm$ integral and saturated. By refining $\hat{X}$ further, we may assume $\hat{X}$ is its own Radon subdivision.

\begin{figure}
    \centering
    \begin{tikzpicture}[scale=.8]
        \draw[->] (0, 0) to (0, 5.5);
        \draw[->] (0, 0) to (5.5, 0);
        \draw[->] (0, 0) to (-1, -1);

        \draw[blue, thick] (0, 0) rectangle (2, 3);
        \draw[blue, thick] (0, 3) rectangle (2, 5);
        \draw[blue, thick] (0, 3) to (2, 5);
        \draw[blue, thick, ->] (2, 5) to (2.5, 5.5);
        \draw[blue, thick, ->] (0,5) to (-.5,5.5);
        \draw[blue, thick, ->] (0,3) to (-.5,2.7);
        \draw[blue, thick, ->] (0,0) to (-.5,-.5);
        \draw[blue, thick, ->] (2,3) to (2.5,3);
        \draw[blue, thick, ->] (2,0) to (2.5,-.5);

        \node[left] at (0, 4){$t$};
        \node[below] at (1, 0){$t$};
        \node[right] at (2, 5){$\ell$};
        \node[left] at (0, 1.5){$s$};
        \node[right] at (1.9, 5.95){$(t + 2\ell, t + 2\ell + s)$};

        \node[right] at (1.45, 3.25){\color{red}$v$};
        \filldraw[red] (2,3) circle (3pt);

        \draw[->, dashed] (0, 0) to (5.5, 5.5);
    \end{tikzpicture}
    \caption{The tropical type of \ref{rem:uglycase}. The image of the evaluation map $(t + \ell, t + \ell + s)$ lands in the cone $x \geq 0, y \geq x$. Let $X_1^t \to X^t$ be the subdivision along the dashed diagonal. In the cone of the moduli space $\Mpt_{\Lambda_1}(X_1^t)$ where $v$ lies on the dashed line (i.e., the cone where $s = t$), the evaluation map $(t + \ell, t + \ell + s)$ has image cone spanned by $(1,1),(1,2)$.}
    \label{fig:noblowupfixesevaluationmaps}
\end{figure}
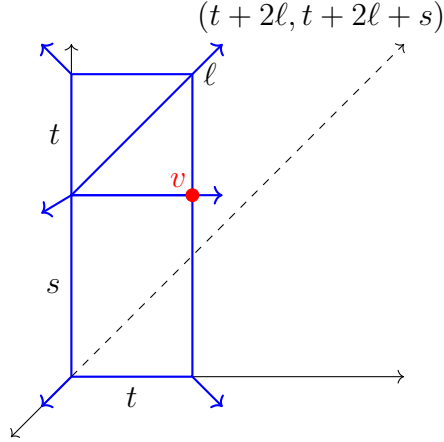

\begin{remark}
\label{rem:uglycase}
One cannot always find a log blowup $\pi: \hat{X}/X$ such that the evaluation map $\Mpt_{\Lambda'}(\hat{X}) \to \hat{X}$ is integral and saturated for every $\Lambda' \in \pi^* \Lambda$.
Endow $X = \P^2$ with its toric log structure and consider the tropical type in $X^t$ depicted in blue in \ref{fig:noblowupfixesevaluationmaps}. This defines a cone $\sigma$ in the moduli space of dimension $8$, namely $2$ plus the number of markings. 
Consider the upper right leg, a marked point with contact order $2$ with each of the $x$- and $y$-axes. The corresponding tropical evaluation map sends such a tropical curve to the point $(t + 2\ell, t + 2\ell + s)$. As $t, \ell, s \geq 0$, this evaluation map lands in the cone $x \geq 0, y \geq x$. In particular, the subcone where the red vertex $v$ lands on the $x$-axis $\R_{> 0} \times (1,0)$ gets sent to the dashed diagonal $\R_{> 0} \times (1,1)$.

To make the evaluation map integral and saturated, one must hence subdivide $X^t$ along the dashed diagonal. Let $X_1 \to X$ be the corresponding blowup. Let $\Lambda_1$ be the discrete data on $X_1$ Given by pulling back this subdivision to the above tropical type. 
Then in the space $\Mpt_{\Lambda_1}(X_1^t)$ one now has the new subcone of $\sigma$ where $v$ lies on the dashed diagonal. Then the image of the evaluation map $(t + \ell, t + \ell + s)$ restricted to this cone is the cone generated by $(1,1),(1,2)$. Hence this subdivision is not enough; we must additionally subdivide in the ray $(1,2)$. Then again, by considering the cone where $v$ lies on the new ray, we see must make a subdivision $X_2 \to X_1$ to add the ray $(1,3)$, and by iterating, $(1,n)$ for any $n \in \N$. Hence no finite blowup will do to make this evaluation map integral and saturated.

\end{remark}

Now we will do the setup in order to prove our theorem \ref{thm:tropicalgluingseparating}. The notation and proof of the formula will be similar to \ref{sec:bir_inv}.

On $\Mpt_{\Lambda|\Xi}^t(X^t)$, we have the length of the glued edge $\ell$, and it is subdivided into two edges of length $\ell_\pm$. We write $\Mpt$ for the (component with discrete data $\Xi$ of) the pullback of $\Mpt_{\Lambda}^t(X^t)$ along the strict gluing map with target $\Mpt_{g,n}$.

Then $\ell$ descends to $\Mpt$, and in fact $\Mpt_{\Lambda|\Xi}^t(X^t) = \Mpt_\ell$, the line segment of length $\ell$ over $\Mpt$, as defined in \ref{def:cube}.

Now the subdivision $\hat{X} \to X$ cuts the edge into parts of length $f_1,\dots,f_m$, and hence induces a decomposition of line segments $\bigsqcup_{i=1}^m \Mpt_{f_i} \to \hat{\Mpt}_\ell$ with $m = m_\Xi$. Likewise, we obtain a decomposition $\bigsqcup_i \Mpt^t_{\Xi;i} \to \hat{\Mpt}^t_{\Xi}$.

Now for each $\Mpt^t_{\Xi;i}$ we have the two Radon rays we denote with $\rho_{\pm i}$, with corresponding PL functions $s_{\pm i}$. We write $E_{\pm i}$ for the corresponding divisors in $\hat{X}$, and $E_i$ for the intersection $E_{+i} \cap E_{-i}$.

\begin{remark}
As in \ref{rem:tangencyallpositive}, in some cases, one of the Radon rays becomes a vertex, and instead of the corresponding PL function $s_{+ i}$ we have the corresponding piecewise constant function $s_{+i} = 1$. For simplicity in exposition, we assume this does not happen, and the $s_{\pm i}$ are actually piecewise linear.
\end{remark}

Now we have the integral and saturated evaluation maps $\ev_{\pm}: \hat{\Mpt}^t_{\Xi} \to \hat{X}$. The $\Mpt^t_{\Xi;i}$ are strict closed substacks of $\hat{\Mpt}^t_\Xi$, and hence the induced maps
\[
\ev_\pm: \Mpt^t_{\Xi;i} \to E_i \times E_i
\]
are integral and saturated. Here we have used that we are working with the separating gluing map: the point being that $\Mpt_{\Xi;i}^t$ is a product over the two separated pieces, and hence if both evaluation maps are integral and saturated, so is their product.

Endow $\Mpt^t_{\Xi;i}$ with the virtual fundamental class 
\[[\Mpt^t_{\Xi;i}]^\vir = r_{+i} r_{-i} \ell_{+i}\ell_{-i} \prod_{j=1}^n \ell_j r_j.\]

Before writing down the formula, we need some final notation. Fix $\Xi \in \gl^* \Lambda, i \in \{1,\dots,m_{\Xi}\}$. Then we let $\pi_{\Xi;i}$ denote the map $\Delta^* \Mpt_{\Xi;i}^t \to \Mpt_{\Lambda|\Xi}^t(X^t)$. Additionally, we let $\pi^*(\Xi;i)$ denote the discrete data on $\hat{X}$ attained by curves in $\Mpt^t_{\Xi;i} \times_{\Mpt_{\Xi}^t(X^t)} \Mpt^t(\hat{X}^t)$; for $\Xi' \in \pi^*(\Xi;i)$ we write $\pi_{\Xi'}: \Mpt^t_{\Xi'}(\hat{X}^t) \to \Mpt_{\Xi;i}^t$.

\begin{theorem}
\label{thm:tropicalgluingseparating}
We have \begin{align*}\gl^! [\Mptmt_\Lambda(X^t)]^\vir &= \sum_{\Xi \in \gl^* \Lambda} \sum_{i=1}^{m_\Xi} c^{\Xi;i}_+ c^{\Xi;i}_- \pi_{\Xi;i,*}\Delta_{E_{\Xi;i}}^! [\Mpt^t_{\Xi;i}]^{\vir} \\ 
&= \sum_{\Xi \in \gl^* \Lambda} \sum_{i=1}^{m_\Xi} c^{\Xi;i}_+c^{\Xi;i}_- \pi_{\Xi;i,*}\Delta_{E_{\Xi;i}}^! \sum_{\Xi' \in \pi^* (\Xi;i)} \pi_{\Xi',*} [\Mptmt_{\Xi'}(\hat X^t)]^\vir.\end{align*}
inside $\CH_*(\gl^*\Mptmt_\Lambda(X^t))$.
\end{theorem}
\begin{proof}
The proof is much the same as that of \ref{thm:tropicalgluingsep}; we highlight where it differs.

We restrict as usual to a single component $\Mpt^t_{\Lambda|\Xi}(X^t)$ of $\gl^* \Mpt^t_{\Lambda}(X^t)$.

We again compute the pullback along the gluing map as
\[
	\gl^! [\Mpt_{\Lambda}^t(X^t)]^\vir = \ell_{+} \ell_{-} \prod_{i=1}^n \ell_i r_i
\]
restricted to this component.

By \ref{lem:deltaet} and \ref{thm:tropicalgluingsep} combined we obtain $c^{\Xi;i}_+c^{\Xi;i}_- \Delta_{E_{\Xi;i}}^! [\Mpt^t_{\Xi;i}]^{\vir} = \ell_{+i}\ell_{-i} \prod_{j=1}^n \ell_j r_j$. By \ref{lem:sumofcubesisbigcube} we compute the pushforward of this to $\Mpt^t_{\Lambda|\Xi}(X^t)$ to be $\ell_{-}\ell_+ \prod_{j=1}^n \ell_j r_j$, finishing the first statement.

The second statement again follows by \ref{lem:sumofcubesisbigcube}.
\end{proof}

\subsection{Logarithmic gluing}
\label{sec:glue:log}

In this section, we lift the tropical theorem \ref{thm:tropicalgluingsep} to a logarithmic statement. We need to assume that every 
map from an irreducible curve to $X$ whose class vanishes is in fact constant. This holds for example if $X$ is projective or $\bb Q$-factorial. 

Just as in the tropical case, for the separating pullback and $\Xi = (\Xi_+,\Xi_-) \in \gl^* \Lambda$, we write \[[\Mpt_{\Xi}(X)]^\vir = [\Mpt_{\Xi_+}(X)]^\vir \boxtimes [\Mpt_{\Xi_-}(X)]^\vir.\]

\begin{theorem}
\label{thm:loggluingclean} 
We have 
\[s_+s_- \gl^! [\Mpt_{\Lambda}(X)]^\vir = \sum_{\Xi \in \gl^* \Lambda} c_1^{\Xi} c_2^{\Xi}\Delta_X^! [\Mpt_{\Xi}(X)]^\vir\]
in $\LogCH_*(\gl^* \Mpt_{\Lambda}(X))$.
\end{theorem}

To prove this, we consider \ref{eqn:biggluing}, and in particular the following subdiagram
\begin{equation}
\label{eq:smallbiggluing}
	\begin{tikzcd}
		\Mpt_{\Lambda|\Xi}(X^t)   \ar[rr] & & \Mpt_{\Xi}(X^t)\\
		& \ar[lu] M \ar[rd] \\
		\Mpt_{\Lambda|\Xi}(X) \ar[rr] \ar[uu, "P"]  \ar[ru] & & \Mpt_{\Xi}(X)  \ar[uu,"P_\Xi"]
	\end{tikzcd}
\end{equation}
Recall that the ``square'' with lower left corner $M$ is a fibre square, both logarithmically and on the level of underlying algebraic stacks. We have constructed perfect obstruction theories for the maps $P$ and $P_\Xi$ in \ref{sec:piercedlogGW} and for the horizontal maps we have taken $\Delta_{X^t}^!, \Delta_X^!$. By pullback in the fibre square, we also get perfect obstruction theories for the other maps in the fibre square, and then by definition the virtual pullbacks in the fibre square commute.

The proof of \ref{thm:loggluingclean} will follow from the following four steps:
\begin{enumerate}[label=\textbf{Step \arabic*}]
	\item Pull back the tropical formula \ref{thm:tropicalgluingsep} to an equality in $\CH_*(M)$.
	\item Construct a perfect obstruction theory for the map $\Mpt_{\Lambda|\Xi}(X) \to M$
	\item Show the perfect obstruction theories in the left triangle of \ref{eq:smallbiggluing} commute.
	\item Show the perfect obstruction theories in the bottom triangle of \ref{eq:smallbiggluing} commute.
    
\end{enumerate}

In Step 2, the perfect obstruction theory will turn out to be pulled back from the log diagonal $\Delta^\log: X \to X \times_{X^t} X$. Then in Step 3 we essentially give a relative version of the classical proof in \cite{Behrend1997GromovWitten}. Step 4, after unravelling the definitions, uses as a key idea that the cotangent complex of $\Delta^\log$ is invariant under taking blowups and strata.

\textbf{Step 1:}

We start with the fibre square in \ref{eq:smallbiggluing}. Consider again for a moment the full diagram \ref{eqn:biggluing}. The maps 
\begin{align*}
\Mpt_{\Lambda|\Xi}(X^t) &\to \Mpt_{\Lambda|\Xi}^{t}(X^t) \\ 
\Mpt_{\Xi}(X^t) &\to \Mpt_{\Xi}^{t}(X^t)
\end{align*} 
are base-changes of the tropicalisation maps on prestable curves, hence are smooth, both with cotangent complex $\Omega_{\Mpt_\Gamma}^\log$. Hence the pullbacks in the fibre square\footnote{To make this a fibre square we take a disjoint union over $\Xi$ which is suppressed in \ref{eqn:biggluing} for reasons of readability. } marked $\boxtimes_3$ in \ref{eqn:biggluing} commute. The same holds for the pullbacks in the unlabelled square in \ref{eqn:biggluing}, because they are pullbacks along the log lci maps in the commutative (non-cartesian) square
\[
\begin{tikzcd}
    \Mpt_\Gamma \ar[r] \ar[d] & \Mpt \ar[d]\\
    \Mpt^t_\Gamma \ar[r] & \Mpt^t; 
\end{tikzcd}
\]
functoriality of log lci pullbacks is proven in special cases in \cite{Barrott2019Logarithmic-Cho}, \cite{Molcho2021A-case-study-of}, and work in progress of Molcho; we prove a general version in \ref{thm:loggysinlogchowpullback}.

Then from \ref{thm:tropicalgluingsep} we find \[s_+s_-\gl^! [\Mpt_{\Lambda}(X^t)]^\vir = \sum_{\Xi \in \gl^* \Lambda} c_+^{\Xi} c_-^{\Xi}\Delta_{X^t}^![\Mpt_{\Xi}(X^t)]^\vir.\]
We can further pull this back along the perfect obstruction theory for the map $P_\Xi$ in \ref{eqn:biggluing}, to obtain
\[
	s_+s_-P_\Xi^! \gl^! [\Mpt_{\Lambda}(X^t)]^\vir = \sum_{\Xi \in \gl^* \Lambda} c_+^{\Xi} c_-^{\Xi}\Delta_{X^t}^![\Mpt_{\Xi}(X^t)]^\vir
\]
inside $\CH_*(M)$.

\textbf{Step 2:}

Next we construct a perfect obstruction theory for the map $\Mpt_{\Lambda|\Xi}(X) \to M$ in \ref{eq:smallbiggluing}. Recall that the space $M$ was defined as the fibre product
\[
\begin{tikzcd}
	M \ar[d]\ar[r] \lpbstrict	& \ar[d] \Mpt_{\Xi}(X) \\
	\Mpt_{\Lambda|\Xi}(X^t) \ar[r] & \Mpt_{\Xi}(X^t)  \\
\end{tikzcd}
\]
but we note that it also fits into another fibre square.

\begin{proposition}
The square (where the right vertical map comes from the evaluation maps)
\[
	\begin{tikzcd}
	\Mpt_{\Lambda|\Xi}(X) \ar[r] \ar[d] \lpbstrict & M \ar[d]  \\
	X \ar[r, "\Delta^\log"] & X \times_{X^t} X 
\end{tikzcd}
\]
is a fibre square in the category of log stacks and on the level of underlying algebraic stacks
\end{proposition}
\begin{proof}
This squares fits in the commutative diagram
\[
\begin{tikzcd}
& M \arrow[rd] \arrow[dd] & \\
\Mpt_{\Lambda|\Xi}(X) \arrow[rr] \arrow[dd] \arrow[ru] & & \mathbb{M}_{\Xi}(X) \arrow[dd] \\
& X \times_{X^t} X \arrow[rd] & \\
X \arrow[rr] \arrow[ru] & & X \times X. 
\end{tikzcd}
\]
We see from \ref{eqn:biggluing} that the front square is a log fibre square. The right square is also, using again \ref{eqn:biggluing} and the fact that $X \times_{X^t} X = (X \times X) \times_{X^t \times X^t} X^t$. Hence the left square is a log fibre square. Because the log diagonal $X \to X \times_{X^t} X$ is strict, it is also a fibre square on the level of underlying algebraic stacks.
\end{proof}

In particular, as the map $\Delta^\log: X \to X \times_{X^t} X$ is strict and regular, with conormal bundle $\Omega_X^\log$, we obtain a virtual pullback $\Delta^{\log,!}$ for the map $\Mpt_{\Lambda|\Xi}(X) \to M$.

\textbf{Step 3:}

We review a well-known lemma. Let $C/S$ be a possibly disconnected family of prestable curves over $S$ with up to two connected components in each fibre. Let $x_1, x_2 : S \to C$ be a pair of smooth sections (at least one through each connected component) and $C'$ the pushout along these two sections. Write $x : S \to C'$ for the node created out of $x_1, x_2$. Consider a map $f' : C' \to X$ with $X$ log smooth. Name related maps as in \ref{eqn:partialnormn}. 

\begin{equation}\label{eqn:partialnormn}
	\begin{tikzcd}
		C \ar[r, "\nu"] \ar[dr, "\pi", swap] \ar[rr, bend left, "f"] 		&C' \ar[r, "f'", swap] \ar[d, "\pi'"] 		&X 		\\
		&S
	\end{tikzcd}
\end{equation}

\begin{lemma}\label{lem:behrendsesconestacks}
	Let $E$ be a vector bundle on $C'$ with locally free sheaf of sections $\cal E$. There is a short exact sequence of vector bundle stacks
	\[
		0 \to \WR_{\pi'} E \to \WR_{\pi} \nu^* E \to x^* E \to 0
	\]
	and dual distinguished triangle of complexes 
	\[
		0 \to R \pi'_* \pra{\cal E \otimes \omega_{\pi'}} \to R\pi_* \pra{\nu^* \cal E \otimes \omega_\pi} \to x^* \cal E \overset{+1}{\to}.
	\]
\end{lemma}

\begin{proof}

	We have a short exact sequence of sheaves 
	\[
	 	0 \to E \to \nu_* \nu^* E \overset{(+-)}{\longrightarrow} x_* x^* E \to 0
	\]
	on $C'$. The first short exact sequence results from taking the Weil restriction $\WR_{\pi'}$ along $\pi'$. The distinguished triangle is the expression of this sequence in terms of complexes of sheaves. 
\end{proof}

\begin{proposition}
Consider the subdiagram of \ref{eqn:biggluing}
\[
	\begin{tikzcd}
		\Mptst_{\Lambda|\Xi}(X^t)   & \\
		& \ar[lu, "p_\Xi", swap] M \\
		\Mptst_{\Lambda|\Xi}(X) \ar[uu, "p"] \ar[ru, "\delta^\log", swap]&
	\end{tikzcd}
\]
where $p$, $p_\Xi$, and $\delta^\log$ have perfect obstruction theories pulled back from $P, P_\Xi,$ and $\Delta^\log$. 
The Gysin pullbacks are compatible 
\[P^! = \Delta^{\log,!} \circ (P_\Xi)^!.\]
\end{proposition}

\begin{proof}
We can omit the stability condition, as the stable loci are open and pulled back from $\Mpt_{\Lambda|\Xi}(X^t)$. The maps in question are \emph{strict}, and the sequence
\[
	\Mpt_{\Lambda | \Xi}(X) \overset{\delta^\log}{\longrightarrow} M \overset{p_\Xi}{\longrightarrow} \Mpt_{\Lambda|\Xi}(X^t)
\]
results in a distinguished triangle of cotangent complexes 
\[
	\ccx{p_\Xi}|_{\Mpt_{\Lambda | \Xi}(X)} \to \ccx{p} \to \ccx{\delta^\log} \overset{+1}{\to}.
\]
We will endow this sequence with a compatibility datum \cite[Definition 4.5]{manolache-pull}.

Extract the subdiagram of \ref{eqn:biggluing}:
\[
\begin{tikzcd}
	\Mpt_{\Lambda | \Xi}(X) \ar[r, "\delta^\log"] \ar[d] \lpb 		&M \ar[r] \ar[d] \lpb 		&\Mpt_{\Xi}(X) \ar[d] 		\\
	X \ar[r] 		&X \times_{X^t} X \ar[r] \ar[d] \lpbstrict 		&X \times X \ar[d] 		\\
			&X^t \ar[r]		&X^t \times X^t.
\end{tikzcd}
\]

The perfect obstruction theories of $\Mpt_{\Lambda | \Xi}(X), M$ over $\Mpt_{\Lambda | \Xi}(X^t)$ are defined in the usual way: restrict the log tangent bundle of the target $X$ along the universal map and push forward. Compare the universal curves $\cal C, \cal C_1 \sqcup \cal C_2$ over $\Mpt_{\Xi}(X)$ and $\Mpt_{\Lambda | \Xi}(X)$ in the diagram 
\[
\begin{tikzcd}
			&		&		&X 		\\
	\cal C \ar[d] \ar[urrr, bend left=15]		&\cal C|_{\Mpt_{\Lambda | \Xi}(X)} \ar[l] \ar[dr, "\pi'", swap] \ar[urr, bend left = 10, "f'"] 		&\cal C_1 \sqcup \cal C_2|_{\Mpt_{\Lambda | \Xi}(X)} \ar[l, "\nu", swap] \ar[r] \ar[d, "\pi"] \ar[ur, bend left = 5, "f"] 		&\cal C_1 \sqcup \cal C_2 \ar[u] \ar[d] 		\\
	\Mpt_{\Lambda}(X)		&		&\Mpt_{\Lambda | \Xi}(X) \ar[r] \ar[d] \ar[ll] \lpb 		&M \ar[d] 	 		\\
			&		&X \ar[r] 		&X \times_{\af{X}} X
\end{tikzcd}
\]

Apply \ref{lem:behrendsesconestacks} to the triangle made by $\pi, \pi', \nu$ to obtain the top row:
\[
\begin{tikzcd}
	R\pi_* \pra{f^* \lkah{X} \otimes \omega_{\pi}} \ar[r] \ar[d] 		&R \pi'_* \pra{{f'}^* \lkah{X} \otimes \omega_{\pi'}} \ar[r] \ar[d] 		&\lkah{X} \ar[r, "+1"] \ar[d, dashed, "\sim"] 		&\phantom{a} 		\\
	\ccx{p_\Xi}|_{\Mpt_{\Lambda | \Xi}(X)} \ar[r]		&\ccx{p} \ar[r] 		&\ccx{\delta^\log} \ar[r, "+1"]		&\phantom{a}
\end{tikzcd}
\]
The resulting dashed arrow is an isomorphism. This diagram is a compatibility datum for the perfect obstruction theories of $p, p_\Xi,$ and $\delta^\log$, so the resulting pullbacks satisfy $P^! = \Delta^{\log,!} \circ P_\Xi^!$ by \cite[Theorem 4.8]{manolache-pull}. 
\end{proof}

\textbf{Step 4:}

\begin{proposition}
\label{prop:deltadeltadeltaPOT}
Consider the diagram
\[
\begin{tikzcd}
& M \arrow[rd, "\delta_{X^t}"]  & \\
\Mpt_{\Lambda|\Xi}(X) \arrow[rr, "\delta_X"] \arrow[ru, "\delta^\log"] & & \mathbb{M}_{\Xi}(X)  \\	
\end{tikzcd}
\]
The maps $\delta^\log$, $\delta_X$ and $\delta_{X^t}$ have virtual pullbacks $\Delta^{\log,!}$, $\Delta_{X^t}^{!}$, and $\Delta_X^!$ respectively. We have an equality
\[
	\Delta^{\log,!} \circ \Delta_{X^t}^{!} = \Delta_X^!.
\]
\end{proposition}
\begin{proof}
This follows from the commutative diagram
\[
\begin{tikzcd}
& X \times_{X^t} X \arrow[rd] & \\
X \arrow[rr, "\Delta_X"] \arrow[ru, "\Delta^{\log}"] & & X \times X
\end{tikzcd}
\]
and the isomorphism $X \times_{X^t} X = (X \times X) \times_{X^t \times X^t} X^t$. 
\end{proof}

\begin{proof}[Proof of \ref{thm:loggluingclean}]
The theorem now follows immediately from \ref{thm:tropicalgluingsep} and the commutativity of the pullbacks in \ref{eqn:biggluing}.
\end{proof}

\subsection{Equivariant gluing formula}
\label{subsec:equiv}

A disadvantage of \ref{thm:loggluingclean} is that the PL functions $s_\pm$ are nilpotent, so multiplication by them loses information. In this section we describe some situations where we can lift $s_\pm$ to non-nilpotent elements in a suitable equivariant Chow ring. 

We begin with some generalities about lifting classes to equivariant Chow. Let $M$ be an algebraic stack with log structure and $f$ a sPL function on $M$. Let $g\colon \bb G_m \times M \to M$ be an action of $G = \bb G_m$. The following are equivalent: 
\begin{itemize}
    \item A $g$-linearisation of $\ca O_M(f)$, i.e.\ a descent datum for $f$ to $M/G$;
    \item A factorisation of $f\colon M \to \theta$ via $M/G$. 
\end{itemize}

If this happens, then there is an element $c_1(\ca O_{M/G}(f)) \in \CH(M/G)$ which pulls back to $c_1(\ca O_{M}(f)) \in \CH(M)$. This lift $c_1(\ca O_{M/G}(f))$ can retain significantly more information than $c_1(\ca O_{M}(f))$. For example, if $M$ is a point with log structure $P$ and $f \in P_{>0}$ then $c_1(\ca O_{M}(f)) = 0 \in \CH(M) = \bb Q$, on the other hand actions of $G$ on $M$ are classified by monoid maps $g\colon P \to \bb Z$, where the image $g(p)$ of $p \in P$ determines the character via which $\bb G_m$ acts on the fibre over $p$ of $P \oplus k^\times \to P$. The above descent condition is equivalent to the image $g(f) \in \bb Z$ being non-zero, and in this case $c_1(\ca O_{M/G}(f)) = g(f) t \in \CH(B\bb G_m) = \bb Q[t]$, in particular it is not a zero-divisor. 

Log blowups of $M$ are determined by ideals in $\ghost_M$, hence any $\bb G_m$-action on $M$ lifts to any log blowup $\tilde M \to M$. In particular, if $f$ is any (not necessarily strict) PL function on $M$ then we can ask whether it descends to $\tilde M/\bb G_m$ where $\tilde M \to M$ is some log blowup on which $f$ is strict, and if this is the case then $f$ induces an element of $\CH^1(\tilde M/\bb G_m)$, hence an element of $\LogCH(M/\bb G_m)$. 

We return now to the setup of \ref{thm:loggluingclean}, and suppose in addition that we have an action of $T = \bb G_m^2$ on $X$ such that $s_+$ and $s_-$ descend to the quotient $X/T$. 

The natural map $X \to X^t$ factors via $X/T$. The group $T$ acts by composition on spaces of maps to $X$, and we can obtain a variant on the bottom two rows of \ref{eqn:biggluing}: 
\begin{equation}
\label{eqn:gluing_equiv}
\begin{tikzcd}[column sep=tiny]
	&& \Mpt_\Lambda(X)/T\ar[dl] & \Mpt_{\Lambda|\Xi}(X)/T  \ar[l]\ar[rr]\ar[dl]\ar[d]\arrow[dll, phantom, "\boxtimes_1"]\arrow[drr, phantom, "\square_5"]&& \Mpt_\Xi(X)/T \ar[d]&&\\
	&\Mpt & {\Mpt_\Gamma}\ar[l, "\gl"] & X/T\ar[rr, "\Delta_X^T"] && (X \times X)/T&&\\
    &&&X \ar[rr, "\Delta_X"]\ar[u] &&X \times X\ar[u]&&
\end{tikzcd}
\end{equation}
Here the action of $T$ on $X \times X$ is the diagonal one, and in the separating case the same holds for the action on $ \Mpt_{\Xi_+}(X) \times \Mpt_{\Xi_-}(X)$. 

\begin{theorem}
\label{thm:loggluingequiv} 
We have 
\[s_+s_- \gl^! [\Mpt_{\Lambda}(X)/T]^\vir = \sum_{\Xi \in \gl^* \Lambda} c_+^{\Xi} c_-^{\Xi}\Delta_X^{T,!} [\Mpt_{\Xi}(X)/T]^\vir\]
in $\LogCH_*(\Mpt_{\Lambda | \Xi}(X)/T)$, where the $T$ actions are as described in \ref{eqn:gluing_equiv}.

Since $s_+$ and $s_-$ descend to the quotient $X/T$ we have natural maps $s_\pm\colon X/T \to \theta$; suppose moreover that the natural map from $T$ to the inertia of $\theta$ is surjective. Then
\[\gl^! [\Mpt_{\Lambda}(X)]^\vir = Q(\sum_{\Xi \in \gl^* \Lambda} \frac{1}{s_+s_- }c_+^{\Xi} c_-^{\Xi}\Delta_X^{T,!} [\Mpt_{\Xi}(X)/T]^\vir)\]
in $\bigoplus_{\Xi \in \gl^*\Lambda} \LogCH(\Mpt_{\Lambda | \Xi}(X))$, where 
\[Q \colon \LogCH(\Mpt_{\Lambda | \Xi}(X)/T) \to \LogCH(\Mpt_{\Lambda | \Xi}(X))\] 
is the pullback map.
\end{theorem}

\begin{proof}
The first displayed equation follows essentially exactly as in the proof of \ref{thm:loggluingclean}, but with the lower part of the diagram \ref{eqn:biggluing} replaced by \ref{eqn:gluing_equiv}. To deduce the second displayed equation, we formally invert (localise at) the classes $s_+s_-$ in $\LogCH(\Mpt_{\Lambda | \Xi}(X))$; the non-trivial content of the second formula is that the localisation map 
\[\LogCH(\Mpt_{\Lambda | \Xi}(X)) \to \LogCH(\Mpt_{\Lambda | \Xi}(X))\left[\frac{1}{s_+s_-}\right]\]
is injective, equivalently that $s_+s_-$ is not a zero-divisor. This can be checked locally; it is enough to check it for $s_{\pm, \Xi}$, which is pulled back from $X$. The condition in the statement of the theorem that $T$ surjects to the inertia of $\theta$ implies that the same holds for the map $s\colon \Mpt_{\Lambda | \Xi}(X) \to \theta$ (since this map factors via $X$ in a $T$-equivariant way), which in turn implies that $s_+s_-$ is not a zero-divisor in $\LogCH(\Mpt_{\Lambda | \Xi}(X))$. 
\end{proof}

\begin{remark}[Extensions and generalisations]
Above we use an action of $T$ on $X$ to induce an action of $T$ on $\Mpt(X)$, and lift $s_\pm$ to equivariant parameters. In the end we only need the action of $T$ on $\Mpt(X)$, and this is sometimes possible without an action of $T$ on $X$. An extreme example is where $X$ is not connected, but the curve class forces all maps to factor via a component of $X$ which carries a $T$ action. More usefully, the curve class and contact orders may force stable maps to $X$ to factor via some subscheme which carries a $T$ action, and if this $T$ action can be lifted to the normal bundle of the subscheme in $X$ it is again possible to lift the above constructions to the equivariant setting. We do not pursue this further at this point. 
\end{remark}

\subsection{The splitting guing map}
\label{sec:glue:logsep}
In this section we lift the tropical gluing statement \ref{thm:tropicalgluingseparating} to the log world. We use the same notation as in \ref{subsubs:sepcasetrop} for the tropical moduli spaces, e.g. $\Mpt_{\Xi;i}^t$, with as usual the logarithmic moduli spaces $\Mpt_{\Xi;i}$ being defined as the pullback $\Mpt_\Xi(X) \times_{\Mpt_\Xi^t(X^t)} \Mpt_{\Xi;i}^t$.

\begin{theorem}
\label{thm:loggluingseparating}
We have \begin{align*}\gl^! [\Mptm_\Lambda(X)]^\vir &= \sum_{\Xi \in \gl^* \Lambda} \sum_{i=1}^{m_\Xi} c^{\Xi;i}_+ c^{\Xi;i}_- \pi_{\Xi;i,*}\Delta_{E_{\Xi;i}}^! [\Mpt_{\Xi;i}]^{\vir} \\ 
&= \sum_{\Xi \in \gl^* \Lambda} \sum_{i=1}^{m_\Xi} c^{\Xi;i}_+c^{\Xi;i}_- \pi_{\Xi;i,*}\Delta_{E_{\Xi;i}}^! \sum_{\Xi' \in \pi^* (\Xi;i)} \pi_{\Xi',*} [\Mptm_{\Xi'}(\hat X)]^\vir.\end{align*}
inside $\CH_*(\gl^*\Mptm_\Lambda(X))$.
\end{theorem}
\begin{proof}
The proof is much the same as the proof of \ref{thm:loggluingclean}. All our moduli spaces can be pulled back to the logarithmic setting, and hence the obstruction theories also pull back. Only step 4 differs. We fix a $\Xi$ and $i$, and consider $E = E_{\Xi;i}$.

In \ref{prop:deltadeltadeltaPOT} we showed the obstruction theories $\Delta_{X^t}^!, \Delta_X^{\log,!}$ and $\Delta_{X}^!$ in the triangle
\[
\begin{tikzcd}
& M \arrow[rd, "\delta_{X^t}"]  & \\
\Mpt_{\Lambda|\Xi}(X) \arrow[rr, "\delta_{X}"] \arrow[ru, "\delta^\log"] & & \Mpt_{\Xi}(X)  \\	
\end{tikzcd}
\]
commute. Here we need to show something very similar. Namely, we pullback the above triangle along the map $\Mpt_{\Xi;i} \to \Mpt_\Xi(X)$. Then we have the obstruction theories $\Delta_X^{\log,!}, \Delta_{\ul{E^t}}^!, \Delta_{\ul{E}}^!$, and need to show that these commute.
This triangle is also a pullback of the commutative diagram
\[
\begin{tikzcd}
& E \times_{E^t} E \arrow[rd] & \\
E \arrow[rr, "\Delta_{E}"] \arrow[ru, "\Delta_E^{\log}"] & & E \times E
\end{tikzcd}
\]
The only content is to show that the obstruction theory $\Delta_{X}^{\log,!}$ is the same as the obstruction theory $\Delta_{E}^{\log,!}$.

The log diagonal of $E$ has log cotangent complex given by the log cotangent bundle of $E$ in degree $-1$. Since log cotangent bundles are invariant under blowup and taking strata, this gives the same perfect obstruction theory $\Delta^{\log,!}$ on $\Mpt_{\Lambda|\Xi}(X) \to M$ as before.
\end{proof}

We now push \ref{thm:loggluingseparating} on virtual fundamental classes forward to a theorem about the logarithmic Gromov--Witten classes of \ref{def:logGWinvariants}. We observe that in the classical setting, for a smooth DM-stack $Y$ the map $\Delta_{Y}^!$ is the same as intersecting with $[Y] \in \CH_*(Y \times Y)$. Under the assumption that $Y$ admits a Chow--Kunneth decomposition, or we are working in Borel--Moore homology where such a decomposition is automatic, this gives us the following theorem.

\begin{theorem}
\label{thm:loggluingkunneth}
Assume each $E_{\Xi;i}$ admits a K\"unneth decomposition $\Delta_{E_{\Xi;i}} = \sum_{k} \delta_{\Xi;i;k+} \boxtimes \delta_{\Xi;i;k;-}$ of its diagonal. Then we have the equality
\begin{multline*}
    \gl^! \LogGW(X;\Lambda, \gamma) =\\  \sum_{\Xi \in \gl^* \Lambda} \sum_{i=1}^{m_\Xi}  c_+^{\Xi;i} c_-^{\Xi;i} \sum_k\sum_{\Xi' \in \pi^*(\Xi;i)} \LogGW(\hat{X};\Xi_+', \gamma_+ \cup \delta_{\Xi;i;k+}) \boxtimes\\ \LogGW(\hat{X};\Xi_-', \gamma_- \cup \delta_{\Xi;i;k-})
\end{multline*}
\end{theorem}

\appendix

\section{Radon subdivision}
\label{sec:radonsubdiv}
In this section we generalise the star subdivision construction. We quickly recall the definition of the star subdivision of a simplicial cone, and the special properties it has.

\begin{definition}
Let $\sigma$ be a simplicial cone, generated by $v_1,\dots,v_n$. Take $c \in \sigma\setminus 0$, and let $\rho$ be the corresponding ray. The star subdivision $\tilde{\sigma}$ is the subdivision with cones $\tau,\tau + \R_{\geq 0}c$ for each subcone $\tau \subseteq \sigma$ with $ c \notin \tau$. For $c = 0$, we define the star subdivision of $\sigma$ in $c$ to be $\sigma$.
\end{definition}

This satisfies the following property.

\begin{lemma}
\label{lem:starsubdiv}
Let $\sigma$ be a cone, and take $c \in \sigma \setminus 0$. Write $\phi: \sigma \to \R$ for the piecewise linear function $x \mapsto \max \{ \lambda: x - \lambda c \in \sigma\}$. Then $\phi$ is strict piecewise linear on the star subdivision in $c$, and $\phi$ is a multiple of the strict piecewise linear function $\phi_\rho$ corresponding to the ray $\rho$. 
\end{lemma}

We aim to extend this to the setting where $c \in \sigma^\gp \setminus 0$. We will use this for $c$ the vector of contact orders at a single marking. 

\begin{definition}
\label{def:radonsubdivcone}
Let $\sigma$ be a simplicial cone, generated by $v_1,\dots,v_n$. Take $c \in \sigma^\gp \setminus 0$. Let $\sum_i \lambda_i v_i = c$ be the unique linear relation between $(v_i)_i,c$. Let $R_0,R_1 \in \sigma$ denote the points $R_0 = \sum_{i : \lambda_i > 0} \lambda_i v_i$ and $R_1 = \sum_{i : \lambda_i \leq 0} -\lambda_i v_i$, and let $\rho_0 = \R_{\geq 0} R_0, \rho_1 = \R_{\geq 0} R_1$ denote the \emph{Radon rays}. Let $R$ be the image of $R_0, R_1$ in $\R^n/\R \cdot c$. 

Then the \emph{Radon subdivision} $\tilde{\sigma}$ of $\sigma$ in $c$ is defined to be the pullback of the star subdivision of $\sigma/\R_{\geq 0} c$ in $R$. For $\Sigma \subset \R^n$ a cone complex, we define the Radon subdivision conewise.
\end{definition}

We first remark that for $\pm c \in \sigma$, we have that all the $\lambda_i$ are all of the same sign, and hence $R_0 = c, R_1 = 0$ or $R_0 = 0, R_1 = c$. In this case the Radon subdivision is simply the star subdivision in $\pm c$.

\begin{example}
Take $\sigma = \R_{\geq 0}^3$ and $c = (1,1,-1)$. Then $R_0 = (1,1,0)$ and $R_1 = (0,0,1)$. The Radon subdivision has $2$ maximal cones, $\angle{e_i,R_0,R_1}$ for $i \in \{1,2\}$.
\end{example}

\begin{example}
Take $\sigma = \R_{\geq 0}^4$, and $c = (1,1,-1,-1)$. Then $R_0 = (1,1,0,0), R_1 = (0,0,1,1)$. The Radon subdivision has $4$ maximal cones, each of the form $\angle{e_i,e_j,R_0,R_1}$ for $\{i,j\} \ne \{1,2\},\{3,4\}$.
\end{example}

\begin{proposition}
\label{prop:radonsubdiv}
Let $\sigma$ be a simplicial cone, generated by $v_1,\dots,v_n$. Take $c \in \sigma^\gp \setminus 0$. The Radon subdivision $\tilde{\sigma}$ is simplicial, with rays generated by $v_1,\dots,v_n,R_0,R_1$. Write $\phi_0,\phi_1$ for the piecewise linear functions $x \mapsto \max \{\lambda: x - \lambda c \in \sigma\}$ and $x \mapsto \max \{\lambda: x + \lambda c \in \sigma\}$ (if $c \in \sigma$ or $-c \in \sigma$, take $\phi_1 = 1$ or $\phi_0 = 1$ respectively). Then $\phi_0$ and $\phi_1$ are strict on $\tilde{\sigma}$, and they are a positive multiple of the piecewise linear functions $\phi_{\rho_0}, \phi_{\rho_1}$ respectively.
\end{proposition}
\begin{proof}
If $c \in \sigma$ or $-c \in \sigma$, this follows from \ref{lem:starsubdiv}. For simplicity, assume now that $\pm c \not\in \sigma$. Write $W$ for the vector space $\R^n/\R \cdot c$ and $q$ for the quotient map $\R^n \to W$. Note that the image $\tau \coloneqq q(\sigma)$ need not be simplicial, but every non-simplicial cone has contains $R$, and has one more ray than its dimension. Hence the star subdivision $\tilde{\tau}$ of $\tau$ in the point $R$ is simplicial. Now consider the pullback, the Radon subdivision $\tilde{\sigma}$.

We claim that for $i = 1, \dots,n$, the inverse image $q^{-1}(q(v_i)) \cap \sigma$ is either the ray generated by $v_i$, or the cone generated by $R_0, R_1$ if $v_i \in \{R_0,R_1\}$. To prove this claim, without loss of generality assume the inverse image is generated by $v_i$ and some other point $w = v_i + \mu c \in \R^n$ with $\mu \ne 0$. We see we can write $w  = \sum_j \mu_j v_j$ with $\mu_j \geq 0$ and $\mu_i = 0$. The equality $v_i - \sum_{j} \mu_j v_j = -\mu c$ shows $v_i \in \rho_0 \cup \rho_1$, in which case we see $q^{-1}(\R_{\geq 0} R) \cap \sigma = \langle R_0, R_1 \rangle$.

This shows the rays in $\tilde{\sigma}$ are exactly generated by $v_1,\dots,v_n,R_0,R_1$. We see the inverse image of a $d$-dimensional cone $\rho$ in $\tilde{\tau}$ spanned by $A \subset \{q(v_1),\dots,q(v_n),R\}$ is either isomorphic to $\rho$, or to the $(d+1)$-dimensional cone spanned by the $d+1$ rays mapping to $A$. Hence the Radon subdivision $\tilde{\sigma}$ is simplicial. The above claim that $q^{-1}(q(v_i)) \cap \sigma $ is either $v_i$ or $\langle R_0, R_1 \rangle$ also shows that $\phi_0,\phi_1$ are non-zero exactly on the rays $\rho_0,\rho_1$, hence $\phi_0$ and $\phi_1$ are multiples of $\phi_{\rho_0}, \phi_{\rho_1}$ respectively.
\end{proof}

\begin{definition}
\label{def:radonsubdiv}
Let $\Sigma$ be a cone complex. Let $\sigma \in \Sigma$ be a cone, take $c \in \sigma^\gp \setminus 0$ such that for $\tau \subsetneq \sigma$ we have $c \not\in \tau^\gp$. Let $S \subset \Sigma$ denote the star of $\sigma$. Then we define the \emph{Radon subdivision} $\tilde{S}/S$ in $c$ to be on any $\sigma' \supset \sigma$ the Radon subdivision of $\sigma'$ in $c$ as per \ref{def:radonsubdivcone}. We say a subdivision $\tilde{\Sigma}/\Sigma$ is a Radon subdivision if it is simplicial, and furthermore $\tilde{\Sigma}|_S = \tilde{S}$.
\end{definition}

\section{Compatibility of proper pushforward and virtual pullback}\label{sec:push_pull_compatibility}

In this section all algebraic stacks are assumed locally of finite type over a fixed ground field $k$. Recall that for us $\CH_*$ denotes Kresch's Chow groups with $\bb Q$-coefficients; for occasional references to the integral versions (which occur only in this section) we will write $\CH_{\bb Z,*}$. 

In \cite[Theorem 4.1]{manolache-pull} Manolache proves the compatibility of virtual pullback with projective pushforward for a diagram of algebraic stacks; the restriction to projective rather than proper morphisms comes about because, at the time, there was no known pushforward for proper representable morphisms of algebraic stacks. 

Since then, Bae, Schmitt and Skowera have constructed such a pushforward \cite[Appendix B]{BSS-I}. More precisely, let $p\colon X \to Y$ be a proper morphism of Artin stacks.  If $p$ is representable (by algebraic spaces) then there is a pushforward $p_*\colon \CH_{\bb Z, *}(X) \to \CH_{\bb Z, *}(Y)$, and if $p$ is of relatively DM type then there is a pushforward $p_*\colon \CH_*(X) \to \CH_*(Y)$. Moreover, these proper pushforwards are shown \cite[Proposition B.18]{BSS-I} to be compatible with flat pullback and with refined Gysin maps for representable regular local immersions. It is natural to expect that this proper pushforward is compatible with virtual pullback; in this appendix we confirm that expectation.

\begin{theorem}\label{thm:push_pull_compatibility}
Consider a fibre diagram of algebraic stacks 
\begin{equation}
\begin{tikzcd}
  F'' \arrow[d, "q"] \arrow[r, "f''"] \pb & G'' \ar[d, "p"] \\
F' \ar[r, "f'"] \ar[d, "g"] \pb & G' \ar[d, "h"]\\
  F \ar[r, "f"] & G
\end{tikzcd}
\end{equation}
with $f$ of DM type and let $\bb E \to \bb L_{f}$ be a (2-term) POT for $f$. Assume that $F'$ and $F''$ are stratified by global quotients in the sense of \cite{Kresch1999Cycle-groups-fo}. Assume that $p$ is proper. If $p$ is representable then for all $\alpha \in \CH_{\bb Z}(G'')$ we have $f^!p_*\alpha = q_*f^!\alpha$; if $p$ is of relatively DM type then the same holds for all $\alpha \in \CH(G'')$. 
\end{theorem}

\begin{proof}
The proof is essentially identical to that of Manolache, except that in 3 places she refers to \cite{Kresch1999Cycle-groups-fo} for the compatibility of flat and Gysin pullbacks with projective push-forward, and we simply replace those references with ones to \cite[Proposition B.18]{BSS-I} where the same statements are proven for the proper pushforward. For the convenience of the reader we spell out some of the details. To avoid distinguishing between the representable and DM-type cases, we will for the rest of the argument simply write $\CH$ where we mean ``$\CH_\bb Z$ if $p$ is representable, or $\CH = \CH_\bb Q$ if $p$ is relatively DM type". 

Suppose first that $f$ is a closed immersion (not necessarily a regular embedding). 
Writing $C_{f'}$ for the intrinsic normal cone of ${f'}$, Kresch constructs a map $\sigma' \colon \CH(G' ) \to \CH(C_{f'})$. We begin (paralleling Manolache) by proving that the diagram 
\begin{equation}
\begin{tikzcd}
  \CH(G'') \arrow[d, "p_*"] \arrow[r, "\sigma '' "] & \CH(C_{f''}) \ar[d, "Q_*"] \\
\CH(G') \ar[r, "\sigma' "]& \CH(C_{f'})
\end{tikzcd}
\end{equation}
commutes, where $Q$ is the composite of the closed immersion $C_{f''} \to q^*C_{f'}$ with the proper morphism $q^*C_{f'} \to C_{f'}$. 

Recall that $M^\circ_{f'}$ is the deformation to the normal cone, fitting into a diagram 
\begin{equation}
\begin{tikzcd}
\bb A^1 \times G' \arrow[r, "i" ] \arrow[d] & M^\circ_{f'} \ar[d] & C_{f'} \ar[l]\ar[d] \\
\bb A^1 \ar[r] & \bb P^1 & 0. \ar[l, "s"]
\end{tikzcd}
\end{equation}
The morphism $\sigma'$ is defined as the composite 
\begin{equation}
\CH_*(G') \to \CH_{* + 1}(G' \times \bb A^1) \xleftarrow{\sim} \frac{\CH_{* + 1}(M^\circ_{f'})}{\CH_{* + 1}(C_{f'})} \xrightarrow{s^!} \CH_*(C_{f'})
\end{equation}
where the first morphism is flat pullback, and the second follows from excision.

The induced map $M^\circ_{f''} \to M^\circ_f$ is proper, DM type (and representable if $p$ was), and restricts to a map $C_{f''} \to C_{f'}$. It therefore suffices to check that each of the squares in the following diagram commutes: 
\begin{equation}
\begin{tikzcd}
\CH_*(G'') \ar[r] \ar[d, "p_*"]&  \CH_{* + 1}(G'' \times \bb A^1) \ar[d, "p_*"]&  \frac{\CH_{* + 1}(M^\circ_{f''})}{\CH_{* + 1}(C_{f''})} \ar[l, "\sim"] \ar[r, "s^!"] \ar[d, "Q_*"]&  \CH_*(C_{f''})\ar[d, "Q_*"]\\
\CH_*(G') \ar[r] &  \CH_{* + 1}(G' \times \bb A^1) &  \frac{\CH_{* + 1}(M^\circ_{f'})}{\CH_{* + 1}(C_{f'})} \ar[l, "\sim"] \ar[r, "s^!"] &  \CH_*(C_{f'}). 
\end{tikzcd}
\end{equation}
Proceeding from left to right, the commutativity of the squares follow in turn from compatibility of proper pushforward with flat pullback, flat pullback, and Gysin pullback, all of which are proven in \cite[Proposition B.18]{BSS-I}. 

At this point we have completed the analogue of Step 1 in Manolache's argument (we warn the reader that this does \emph{not} mean that we have yet proven the theorem in the case where $f$ is a closed immersion). Her Step 2 (which completes the proof for general $f$, including closed immersions) now goes through completely unchanged in our setting. 
\end{proof}

\section{Pushing forward homological piecewise polynomial functions}
\label{sec:push_hpp}

Let $X$ be a finite cone stack with boundary, meaning that all but finitely many cones in the ambient Artin fan lie in the boundary. If $X$ is locally free then there is a natural isomorphism $\sPP_*(X) \longsimeq \CH_*(X)$ defined in \cite[Theorem 59]{RPSS_log_taut}. The goal of this section is to extend this as far as possible in the non-smooth case. 

To construct a map, choose a locally free log alteration $\tilde X \to X$. Pull back the strict piecewise polynomial to $\tilde X$ and push forward the resulting Chow class:
\begin{equation}\label{eqn:Psimapdef}
    \Psi : \sPP_*(X) \to \sPP_*(\tilde X) \isom \CH_*(\tilde X) \to \CH_*(X). 
\end{equation}
By the projection formula, the composite $\sPP_*(X) \to \CH_*(X)$ is independent of choice.

\begin{remark}\label{rem:Psicompatiblepushforwardpullback}
    Let $p : X \to Y$ be a strict, representable map of finite cone stacks with boundary which is a closed immersion (resp.\ smooth). Then 
    \[
        \Psi_Y \circ p_* = p_* \circ \Psi_X \qquad 
        (\text{resp.\ } \Psi_X \circ p^* = p^* \circ \Psi_Y).
    \]
    If $X, Y$ are already locally free, the statements hold by naturality of the isomorphisms (28) and Proposition 67 in \cite{RPSS_log_taut}. Otherwise, choose a locally free log alteration $\tilde Y \to Y$ and use $\tilde X = X \times_Y^{\ell} \tilde Y$ to reduce to that case. 
\end{remark}

We do not expect the map $\Psi$ is surjective, but this section will prove the following. 
\begin{proposition}\label{prop:Psiinj}
If $X$ is a finite cone stack, the map $\Psi : \sPP_*(X) \to \CH_*(X)$ of \ref{eqn:Psimapdef} is injective. 
\end{proposition}

\begin{proof}
We induct on the cardinality of $|X|$. The minimal points of $|X|$ are open: this can be checked on an \'etale cover by Artin cones, where open sets are those closed under generization. Let $U \in X$ be some such point, so $U = [B\GG_m^n/G]$ for $n \in \NN$ and $G$ a finite group. We have a diagram 
\begin{equation}\label{eqn:commdiaglocchernsPPCH}
    \begin{tikzcd}
    0 \ar[r] &  \CH_*(Z)  \arrow[r]  & \CH_*(X) \ar[r] & \CH_*(U) \ar[r] & 0 \\
    0 \ar[r] & \sPP_*(Z) \ar[r] \ar[u, "\Psi_Z"] & \sPP_*(X) \ar[r] \ar[u, "\Psi_X"] & \sPP_*(U) \ar[u, "\Psi_U"]
    \end{tikzcd}
\end{equation}
where the top row is exact by \cite[Proposition 60]{RPSS_log_taut} because the map $\CH_*(U, 1) \to \CH_*(Z)$ is zero, using arguments from \cite{BSS-II}. The vertical map $\Psi_Z$ is injective by our induction hypothesis. To show that $\Psi_X$ is injective, it suffices (by a small diagram chase) to show $\Psi_U$ is injective. We have thus reduced to the following \ref{lem:Psimapstackypointcase}. 
\end{proof}

\begin{lemma}\label{lem:Psimapstackypointcase}
    Let $G$ be a finite group acting on a sharp, f.s.\ monoid $P$ with $n = {\rm rank} P^\gp$ and write $X \coloneqq \bra{\af{P}/G}$ for the stack quotient. Let $U \coloneqq [B\GG_m^n/G] \in [\af{P}/G]$ be the unique stacky closed point.  The map $\Psi_U : \sPP_*(U) \to \CH_*(U)$ is injective. 
\end{lemma}
Recall that we use $\Q$ coefficients in $\sPP_*$ and $\CH_*$.
\begin{proof}
    We begin with the case where $G$ is trivial. Write $j : U \to X$ for the inclusion and choose a basis $x_1, \cdots, x_n \in P^\gp$. Write 
    \[
        S \coloneqq \sPP_*(X) = \Sym^* P^\gp \otimes \Q = \Q[x_1, \cdots, x_n], 
    \] 
    \[
    K \coloneqq {\rm Frac}(S) = \Q(x_1, \cdots, x_n)
    \] 
    for the ring $S$ of piecewise polynomial functions on $X$ and its fraction field $K$. Abbreviate $M_K \coloneqq M \otimes_S K$ for an $S$-module $M$. Let $h_i = c_1(\OO_X(x_i)) \in \CH^*(X)$ be the operational Chern classes of the corresponding line bundles on $X$. 

    We recall a few classical results.  
    \begin{itemize}
        \item Pushforward along the closed immersion $j : U \to X$ induces an isomorphism $j_* : \CH_*(U)_K \longsimeq \CH_*(X)_K$ of Chow groups tensored with $K$ \cite[proof~of~Theorem 1]{edidingrahambottlocalization}. 
        \item The fundamental classes differ by an ``equivariant multiplicity'' $e \in K$ in $\CH_*(X)_K$: $e \cdot j_*[U] = [X]$ \cite[Corollary 15]{Brion1998}. This multiplicity is nonzero $e \neq 0$ because $[X]$ is non-zero.
    \end{itemize}
    The map $S = \sPP_*(X) \to \CH_*(U)$ sending 
    \[f(x_1, \cdots, x_n) \mapsto f(h_1, \cdots, h_n)|_U \cap [U]\] 
    is an isomorphism \cite[Remark after Proposition 1]{edidingrahambottlocalization}. The map 
    \begin{equation}\label{eqn:normalizedchernisomorphism}
        \begin{split}
            \sPP_*(X)_K &\longsimeq \CH_*(X)_K       \\
            f(x_1, \cdots, x_n) &\mapsto f(h_1, \cdots, h_n) \cap [X]
        \end{split}
    \end{equation}
    can also be seen to be an isomorphism, as 
    \[f(h_1, \cdots, h_n) \cap [X] = e \cdot j_*(f(h_1, \cdots, h_n)|_U \cap [U])\] 
    by the projection formula and the above. 

    We identify the map $\Psi_X$ with \ref{eqn:normalizedchernisomorphism}. Let $p : B \to X$ be a log alteration by a locally free Artin fan. If $f \in S$ is a piecewise polynomial, then  
    \begin{equation*}
        \begin{split}
            \Psi_X(f) &\coloneqq p_*\Psi_B(f|_B) \\
                &=p_*(f(h_1, \cdots, h_n)|_B \cap [B])     \\
                &=f(h_1, \cdots, h_n) \cap [X]
        \end{split}
    \end{equation*}
    by the projection formula and birationality of $p$. 

    We have a commutative diagram by \ref{rem:Psicompatiblepushforwardpullback}
    \[
    \begin{tikzcd}
        \sPP_*(U) \ar[rr, "\Psi_U"] \ar[d]       &       &\CH_*(U) \ar[d, "j_*"]       \\
        \sPP_*(X) \ar[r]       &\sPP_*(X)_K \ar[r, "\sim"]        &\CH_*(X)_K.
    \end{tikzcd}
    \]  
    As the bottom horizontal composite is injective and $\sPP_*(U) \subseteq \sPP_*(X)$ is a subgroup, $\Psi_U$ is indeed injective. 

    Now let $G$ be any finite group. Recalling that we work with $\Q$-coefficients, passing to $B\bb G_m^n/G$ simply means taking $G$-invariants both for $\sPP_*$ and $\CH_*$. In particular, this preserves injectivity. 
\end{proof}

\begin{remark}
    Beware that $\Psi_U$ is not the map sending a piecewise polynomial $f \in \Sym^* P^\gp \otimes \Q$ to 
    $f(h_1, \cdots, h_n) \cap [U]$,
    using notation from the proof of \ref{lem:Psimapstackypointcase}. 
    This is because $j_*[U]$ and $[X]$ differ by $e \in K^*$. 
\end{remark}

\section{Logarithmic virtual pullback}
\label{sec:log_vir_pull}

All stacks in this section are locally of finite type over an arbitrary field $k$. The results hold for Chow groups with either $\Z$ or $\Q$ coefficients. 

In this appendix we will define a log Gysin pullback for morphisms endowed with a log perfect obstruction theory. Our definition is inspired by \cite{sammolchointersectiontheorypreprint,Bae2022Chow-rings-I,adeelkhanpb}.

Let $X$ be a quasicompact log algebraic stack. Let $\Alt(X)$ denote the poset of log alterations of $X$, and let $\cal C_X \subseteq \Alt(X)$ be the subposet of log alterations $\tilde X \to X$ pulled back from a map $\scr B \to \af{X}$ to the Artin fan $\af{X}$ of $X$ such that $\scr B$ admits a strict map to $\bra{\Aff^N/\GG_m^N}$ for some $N \in \NN$. 
The subposet ${\cal C}_X$ is cofinal \cite[Theorem 4.6.2]{Abramovich2017Boundedness-of-}, \cite[Appendix B]{drink}, \cite[Theorem 11]{kkmstoroidalembeddings}. By \cite[Proposition B.5]{drink}, a map of log algebraic stacks $T \to \tilde X$ with $\tilde X \in \cal C_X$ fits in a unique commutative square 
\[
\begin{tikzcd}
    T \ar[r] \ar[d]       &\tilde X \ar[d]       \\
    \af{T} \ar[r]      &\af{\tilde X}
\end{tikzcd}
\]
We shall frequently use this functoriality of Artin fans without comment. 

Let $f : X \to Y$ be a quasicompact morphism with $Y$ quasicompact. Write $\Alt(f)$ for the category with objects $\tilde f$ commutative squares 
\[
\begin{tikzcd}
    \tilde X \ar[r, "{\tilde f}"] \ar[d]        &\tilde Y \ar[d]       \\
    X \ar[r]       &Y
\end{tikzcd}
\]
where the vertical maps $\tilde X \to X, \tilde Y \to Y$ are log alterations. Write $\cal C \subseteq \Alt(f)$ for the subposet of $\tilde f : \tilde X \to \tilde Y$ such that $\tilde X, \tilde Y$ belong to ${\cal C}_X, {\cal C}_Y$ as above. Remark that $\cal C \subseteq \Alt(f)$ is cofinal, as are its images in $\Alt(X), \Alt(Y)$. 

Suppose $f : X \to Y$ is endowed with a log perfect obstruction theory $\Cl{f} \subseteq E$ and let $\tilde f : \tilde X \to \tilde Y$ lie in $\cal C \subseteq \Alt(f)$. We will first construct a Gysin pullback $f^!: \CH(\tilde{Y}) \to \CH(\tilde{X})$ and in \ref{thm:loggysinlogchowpullback} show that this induces a pullback on log Chow. Endow $\tilde f$ with the induced log perfect obstruction theory 
\begin{equation}\label{eqn:logaltpblogpot}
    \Cl{\tilde f} \subseteq \Cl{f}|_{\tilde X} \subseteq E|_{\tilde X}
\end{equation}
using the closed immersions of \cite[Remarks 2.3, 2.14]{herrthesis}. 

Choose Artin fans $\scr B, \scr C$ admitting strict maps to free Artin fans 
\[
	\scr B \to \bra{\Aff^N/\GG_m^N}, \qquad \scr C \to \bra{\Aff^{N'}/\GG_m^{N'}} \qquad N, N' \in \NN
\]
that fit in a commutative diagram 
\[
\begin{tikzcd}
    \tilde X \ar[r, "g"] \ar[dr] \ar[rr, bend left=30, "{\tilde f}"]       &\tilde W \ar[r] \ar[d] \lpbstrict       &\tilde Y \ar[d]       \\
            &\scr B \ar[r, "\psi", swap]         &\scr C
\end{tikzcd}
\]
with $\tilde X \to \scr B, \tilde Y \to \scr C$ strict. Such Artin fans exist and can be taken to be $\scr B = \af{\tilde X}$, $\scr C = \af{\tilde Y}$ because $\tilde f \in \cal C$. Write $\tilde W$ for the f.s.\ pullback and $g, \psi$ for the labelled arrows. 

Note that $g : \tilde X \to \tilde W$ is strict and $\tilde W \to \tilde Y$ is log \'etale. Log perfect obstruction theories on $\tilde f$ are equivalent to ordinary perfect obstruction theories on $g$ by the isomorphism \cite[Proposition 2.5]{herrthesis}
\[
    \Cl{\tilde X/\tilde Y} = C_{\tilde X/\tilde W}.
\]

Define a pullback map $f^!$ as the composite 
\begin{equation}\label{eqn:levelwiseloggysin}
    \CH(\tilde Y) \overset{\psi^!}{\longrightarrow} \CH(\tilde W) \overset{g^!}{\longrightarrow} \CH(\tilde X),
\end{equation}
where $g$ the Gysin map coming from the perfect obstruction theory $E$ and $\psi^!$ is described in the next remark. 

\begin{remark}\label{rmk:stackylcipb}
    Let $s : B \to C$ be a morphism of smooth algebraic stacks. Let $V \to C$ be a morphism from an algebraic stack $V$ and write $W \coloneqq V \times_C B$ for the pullback. We need a pullback map 
    \[
        s^! : \CH(V) \to \CH(W)
    \]
    akin to l.c.i.\ pullback, but $s$ need not be representable and its cotangent complex may have cohomology in positive degrees. 
    
    Factor through the graphs of the maps
    \[
    \begin{tikzcd}
        W \ar[r] \ar[d] \pb       &V \times B \ar[r] \ar[d] \pb         &V \ar[d]      \\
        B \ar[r, "{s'}", swap]       &C \times B \ar[r, "{s''}", swap]         &C.
    \end{tikzcd}
    \]
    Now $s' = \Gamma_s$ is representable and l.c.i.\ and $s''$ is smooth, so we can take the composite $s^! \coloneqq {s'}^! \circ {s''}^*$ of the flat and l.c.i.\ pullbacks. 
\end{remark}

We use these ``levelwise" log pullbacks \ref{eqn:levelwiseloggysin} to define a log Gysin map on log Chow before checking that these homomorphisms are independent of choices and well-defined.

\begin{theorem}\label{thm:loggysinlogchowpullback}
    Let $f : X \to Y$ be a quasicompact morphism of quasicompact log algebraic stacks. Suppose $f$ is equipped with a log perfect obstruction theory $\Cl{f} \subseteq E$, and pull it back to all log alterations $\tilde f \in \Alt(f)$ as in \ref{eqn:logaltpblogpot}. The maps \ref{eqn:levelwiseloggysin} are independent of the choice of $\scr B \to \scr C$ and compatible with restriction to other log alterations $\tilde f$, inducing a morphism on log Chow 
    \begin{equation}\label{eqn:potpullback}
        f^! \coloneqq \colim_{\tilde f \in \cal C \subseteq \Alt(f)} \tilde f^! : \LogCH_*(Y) \to \LogCH_*(X).
    \end{equation}
    If $\tilde f \in \Alt(f)$ is a log alteration of $f$, it induces the same map on log Chow $f^! = \tilde f^!$. 
\end{theorem}

\begin{definition}
    \label{def:potpullback}
    The \emph{log Gysin pullback homomorphism} $f^!$ is the pullback on log Chow groups from \ref{eqn:potpullback}.
\end{definition}

\begin{proof}[Proof of \ref{thm:loggysinlogchowpullback}]
    The equality $\tilde f^! = f^!$ is derived from how the log Gysin pullbacks are defined levelwise on all suitable log alterations $\cal C \subseteq \Alt(f)$. 

    Consider a map $\psi : \scr B \to \scr C$ of Artin fans admitting strict maps to free Artin fans fitting in a commutative square with $\tilde f$ as above. We get an induced factorization through the Artin fans of $\tilde X, \tilde Y$ using \cite[Proposition B.5]{drink}: 
    \[
    \begin{tikzcd} 
        \tilde X \ar[r] \ar[d]        &\tilde Y \ar[d]       \\
        \af{\tilde X} \ar[r] \ar[d]       &\af{\tilde Y} \ar[d]    \\
        \scr B \ar[r]      &\scr C.
    \end{tikzcd}    
    \]

    The morphisms $\af{\tilde X} \to \scr B, \af{\tilde Y} \to \scr C$ are in fact strict, as the locus where they are strict is an open on the source which must contain the closed points of the Artin fans as they are in the image of $\tilde X, \tilde Y$. Strict maps between Artin fans are \'etale, so the cotangent complex and intrinsic normal cone of $\scr B \to \scr C$ pulls back to that of $\af{\tilde X} \to \af{\tilde Y}$:
    \[
        \ccx{\scr B/\scr C}|_{\af{\tilde X}} = \ccx{\af{\tilde X}/\af{\tilde Y}}, \qquad C_{\scr B/\scr C}|_{\af{\tilde X}} = C_{\af{\tilde X}/\af{\tilde Y}}.
    \]
    Writing $\tilde W' = \tilde Y \times_{\af{\tilde Y}} \af{\tilde X}$, the induced morphism $\tilde W' \to \tilde W$ is \'etale. It results that the two pullbacks \ref{eqn:levelwiseloggysin} defined using $\scr B \to \scr C$ and $\af{\tilde X} \to \af{\tilde Y}$ are the same. 
    
    We must show that the morphisms \ref{eqn:levelwiseloggysin} are compatible with the pullback maps defining $\LogCH_*$. Let $\tilde f_i : \tilde X_i \to \tilde Y_i \in \cal C \subseteq \Alt(f)$ be a pair of log alterations of $f$. As this poset is filtered, we can assume there is a map between them 
    \[
    \begin{tikzcd}
        \tilde X_1 \ar[r, "{\tilde f_1}"] \ar[d]      &\tilde Y_1 \ar[d]         \\
        \tilde X_2 \ar[r, "{\tilde f_2}"] \ar[d]      &\tilde Y_2 \ar[d]         \\
        X \ar[r]       &Y.
    \end{tikzcd}
    \]

    As $\tilde f_i$ lie in $\cal C$, we may choose Artin fans $\scr B_i, \scr C_i$ which map strictly to some $\af{}^N$ and fit in cartesian diagrams
    \[
        \begin{tikzcd}
            \tilde X_1 \ar[r] \ar[d] \lpbstrict      &\scr B_1 \ar[d]       \\
            \tilde X_2 \ar[r] \ar[d] \lpbstrict      &\scr B_2 \ar[d]       \\
            X \ar[r]       &\af{X}
        \end{tikzcd} \qquad
        \begin{tikzcd}
            \tilde Y_1 \ar[r] \ar[d] \lpbstrict      &\scr C_1 \ar[d]       \\
            \tilde Y_2 \ar[r] \ar[d] \lpbstrict      &\scr C_2 \ar[d]      \\
            Y \ar[r]       &\af{Y}.
        \end{tikzcd}
    \]

    Form the diagram 
    \[
    \begin{tikzcd}
        \tilde X_1 \ar[r, "{g_1}"] \ar[d] \lpbstrict      &\tilde W_1 \ar[r] \ar[d]         &\tilde Y_1 \ar[d]         \\
        \tilde X_2 \ar[r, "{g_2}", swap]      &\tilde W_2 \ar[r]         &\tilde Y_2.
    \end{tikzcd}
    \]
    We claim the left square in this rectangle is an f.s.\ pullback, in which case it is also a pullback of underlying schemes because the bottom horizontal arrow is strict. It suffices to show that $\tilde W_1 \to \tilde W_2$ is pulled back from $\scr B_1 \to \scr B_2$, as $\tilde X_1 \to \tilde X_2$ is. But all of the faces except the top and bottom face of the cube are f.s.\ pullback squares, so it results: 
    \[
    \begin{tikzcd}
        \tilde W_1 \ar[rr] \ar[dr] \ar[dd]      &&\tilde Y_1 \ar[dr] \ar[dd]         \\
                &\tilde W_2         &&\tilde Y_2  \ar[dd] \ar[from=ll, crossing over]        \\
        \scr B_1 \ar[rr, "\psi_1", near start] \ar[dr, "\varphi", swap]        &&\scr C_1 \ar[dr, "\varphi'"]       \\
                &\scr B_2 \ar[rr, "\psi_2", swap] \ar[from=uu, crossing over]       &&\scr C_2.
    \end{tikzcd}
    \]

    Name the morphisms between the Artin fans $\psi_i, \varphi, \varphi'$ as in the bottom face of the cube. This bottom face is not cartesian, but its commutativity suffices to show the l.c.i.\ pullbacks agree
    \[
        \psi_1^! \circ {\varphi'}^! = \varphi^! \circ \psi_2^!
    \]
    and are both equal to the l.c.i.\ pullback along the composite $\scr B_1 \to \scr C_2$. So we have a commutative diagram of pullbacks 
    \[
    \begin{tikzcd}
        \CH(\tilde Y_2) \ar[r, "{\psi_2^!}"] \ar[d, "{{\varphi'}^!}", swap]      &\CH(\tilde W_2) \ar[r, "{g_2^!}"] \ar[d, "\varphi^!"]        &\CH(\tilde X_2) \ar[d, "{\varphi^!}"]         \\
        \CH(\tilde Y_1) \ar[r, "{\psi_1^!}", swap]      &\CH(\tilde W_1) \ar[r, "{g_1^!}", swap]         &\CH(\tilde X_1),
    \end{tikzcd}
    \]
    where the right square commutes because the perfect obstruction theory of $g_1$ is the restriction of $g_2$ and by standard compatibility of Gysin pullbacks \cite[Theorem 4.3]{manolache-pull}, \cite[Theorem 6.4]{Fulton1984Intersection-th}. Given classes $\alpha_i \in \CH(\tilde Y_i)$ satisfying $\alpha_1 = {\varphi'}^! \alpha_2$, this means the resulting classes 
    \[
        \beta_i \coloneqq g_i^! \circ \psi_i^! (\alpha_i)
    \]
    satisfy $\beta_1 = \varphi^! \beta_2$, which was to be shown. 
\end{proof}

\begin{proposition}
\label{prop:potpullbackfunctorial}
    If $X \overset{f}{\longrightarrow} Y \overset{g}{\longrightarrow} Z$ are a pair of morphisms equipped with log perfect obstruction theories $\Cl{f} \subseteq E$, $\Cl{g} \subseteq F$ and a compatibility datum \cite[Remark 3.4]{herrthesis}, \cite[Definition 4.5]{manolache-pull}, the resulting pullbacks are compatible in that 
        \[
            (g \circ f)^! = f^! g^!
        \]
    as morphisms from $\LogCH_*(Z)$ to $\LogCH_*(X)$. 
\end{proposition}

\begin{proof}
    As ${\tilde f}^! = f^!$ for log alterations $\tilde f \in \cal C \subseteq \Alt(f)$ by Theorem \ref{thm:loggysinlogchowpullback}, we can freely replace $X, Y, Z$ by log alterations which admit strict maps to free Artin cones $\af{}^N$. Their Artin fans are thus functorial by \cite[Proposition B.5]{drink}. It similarly suffices to check the levelwise maps \ref{eqn:levelwiseloggysin} $(g \circ f)^! = f^! g^!$ agree, as the same argument applies to each log alteration of $X, Y, Z$. 
    
    Write 
    \[
        W = Z \times_{\af{Z}}^{\ell s} \af{Y}, \quad
        V = Y \times_{\af{Y}}^{\ell s} \af{X}, \quad
        U = Z \times_{\af{Z}}^{\ell s} \af{X}
    \]
    and form the diagram
    \[
    \begin{tikzcd}
        X \ar[rr, "f'"] \ar[dr, "{f''}", swap]   &  &U \ar[rr] \ar[dd] \lpbstrict       &&W \ar[dd] \ar[dr] \ar[ddr, phantom, "\ulcorner \ell s", near start]        \\
            &V \ar[ur, "{g''}"] \ar[dr]  \ar[phantom, drrr, "{\!\!\!\!\!\!\!\!\!\!\!\!\ulcorner \ell s}", near start]      &&Y \ar[ur, "{g'}"] \ar[dr] \ar[from=ll, crossing over]    &       &Z \ar[d] \ar[from=ll, crossing over]      \\
            &&\af{X} \ar[rr, "\varphi", swap]         &&\af{Y} \ar[r, "\psi", swap]         &\af{Z}. 
    \end{tikzcd}
    \]
    Label arrows as written. 

    The log perfect obstruction theory for $g : Y \to Z$ gives an ordinary perfect obstruction theory for $g' : Y \to W$, and we equip $g'' : V \to U$ with the restriction of that perfect obstruction theory. The logarithmic compatibility datum for $X \to Y \to Z$ induces an ordinary compatibility datum for $X \to V \to U$ as argued in \cite[Theorem 3.12]{herrthesis}, so ${f''}^! {g''}^! = {f'}^!$. 

    The two pullbacks we must equate are defined as the composites: 
    \[
    \begin{tikzcd}[row sep = small, column sep = small]
        f^! g^! :   &\CH(Z) \ar[r, "\psi^!"]  &\CH(W) \ar[r, "{{g'}^!}"]  &\CH(Y) \ar[r, "\varphi^!"]  &\CH(V) \ar[r, "{{f''}^!}"] &\CH(X)       \\
        (g \circ f)^!:  &\CH(Z) \ar[r, "{(\psi \circ \varphi)^!}"] &\CH(U) \ar[r, "{f'}^!"] &\CH(X).
    \end{tikzcd}
    \]
    As $g''$ has the perfect obstruction theory pulled back from $g'$, the pullbacks agree ${g''}^! = {g'}^!$ where both defined. The Gysin pullbacks are compatible $\varphi^! {g'}^! = {g''}^! \varphi^!$ \cite[Theorem 4.3]{manolache-pull}, \cite[Theorem 6.4]{Fulton1984Intersection-th}. Then 
    \begin{align*}
        (g \circ f)^! &= {f'}^! (\psi \circ \varphi)^!  \\
            &={f''}^! {g''}^! \varphi^! \psi^! \\
            &={f''}^! \varphi^! {g'}^! \psi^! \\
            &=f^! g^!
    \end{align*}
    and we are done. 
\end{proof}

We need a result from \cite[Theorem 5.8]{sammolchointersectiontheorypreprint}, which also proves the converse.

\begin{definition}[{\cite[Definition 5.9]{sammolchointersectiontheorypreprint}}]
Let $f: X \to Y$ be integral and saturated. Then $f$ is \emph{log lci} if the truncation $\tau_{\leq -1}\lccx{X/Y}$ of the log cotangent complex is perfect.
\end{definition}

\begin{theorem}[{\cite[Theorem 5.8]{sammolchointersectiontheorypreprint}}]\label{thm:loglcimaps}
    Let $f: X \to Y$ be integral and saturated. Then $f$ is log lci if and only if $f$ factors as a strict regular closed immersion $X \to W$ followed by a log smooth map $W \to Y$.
\end{theorem}

Now we can define the Gysin pullback for log lci maps.

\begin{definition}
\label{def:loglcipullback}
For $f: X \to Y$ a log lci map, and any map $Z \to Y$, we write $f^!: \LogCH(Z) \to \LogCH(X \times_Y^\ell Z)$ for the Gysin pullback \ref{def:potpullback}, where the f.s.\ pullback $f' : X \times_Y^\ell Z \to Z$ is endowed with the log POT pulled back from $f$.
\end{definition}

\begin{remark}
    Three different definitions for Gysin pullbacks $f^!: \LogCH(Y) \to \LogCH(X)$ along log lci maps $f: X \to Y$ appear in \cite[Section~5]{sammolchointersectiontheorypreprint}; the first is the same as our definition in the case $Z = Y$.
\end{remark}

We can then equip a log lci morphism $f$ with a canonical log perfect obstruction theory $E = \Cl{f}$. Functoriality of log lci pullback follows from \ref{prop:potpullbackfunctorial} using the canonical compatibility datum provided by the distinguished triangle of (ordinary) cotangent complexes associated to the sequence 
\[
	X \to \Log Y \to \Log Z.
\]

For a log algebraic stack $X$, write 
\[
    {\rm log}K^\circ (X) = \colim_{\tilde X \in {\cal C}_X} K_\circ(\tilde X),
\]
where the colimit is over the pullback maps described in \cite[\S 4.1]{drink} just as in Chow. See also \cite[\S 2.4.6]{itologmotives}.
The above arguments apply equally well to define a log Gysin pullback in $K$-theory using the pullbacks defined in \cite{fengquvirtualpullbackkthy} or \cite{ypleeoriginalkthy}. 

\begin{corollary}
    Let $f : X \to Y$ be as in \ref{thm:loggysinlogchowpullback}. There is a pullback morphism 
    \[
        f^! : {\rm log}K^\circ (Y) \longrightarrow {\rm log}K^\circ (X)
    \]
    defined on log alterations in ${\cal C} \subseteq \Alt(f)$ by the analogous pullbacks to \ref{eqn:levelwiseloggysin} in $K$-theory. In the situation of \ref{thm:loggysinlogchowpullback}, the pullback morphisms agree:
    \[(g \circ f)^! = f^! g^! : {\rm log}K^\circ (Z) \longrightarrow {\rm log}K^\circ (X)\] 
\end{corollary}

\section{Example calculation}
\label{sec:example_calculation}
In this section we provide a sample calculation of \ref{thm:loggluingkunneth}. We let $X = \P^2$ with its toric log structure, and we denote its coordinates as $(x_0: x_1:x_2)$. We denote its three divisors as $D_0 = V(x_0), D_1 = V(x_1), D_2 = V(x_2)$, and record contact orders as $a \in \Z^3$. We write $e_0,e_1,e_2$ for the standard basis vectors of $\bb Z^3$.

We will give a calculation for the pullback of a genus $2$ log Gromov--Witten invariant $\LogGW(X; \Lambda, \gamma)$ with \[\Lambda = (g = 2, n = 4, \beta = kH, a)\] for a fixed $k > 0$, with contact orders $a_i = ke_{i-1}$ for $1 \leq i \leq 3$ and $a_4 = 0$, and \[
\gamma=\bigl([\mathrm{pt}]_{D_0},\,[\mathrm{pt}]_{D_1},\,1,\,[\mathrm{pt}]_{\bb P^2}\bigr).
\]
To start, we will computing the homological degree of $\LogGW(X;\Lambda,\gamma)$ for this choice of $X$ and any $\Lambda, \gamma$. The virtual dimension of $\Mpt_{\Lambda}^t(X^t)$ is $-n-p$ where $p$ is the number of markings with a negative tangency. Since $X$ is toric, its log tangent bundle $T^\log_X$ is trivial of rank $2$, with three generators $d\log x_i$ and relation $\sum_i d\log x_i = 0$. Then the relative perfect obstruction theory for $\Mpt_\Lambda(X)/\Mpt_\Lambda(X^t)$ is given by $\pi_* \Ocal_C^2$ of relative virtual dimension $-2(g-1)$. In total the virtual relative dimension of $\trop: \Mpt_\Lambda(X)/\Mpt_\Lambda^t(X^t)$ is $-2(g-1) + 3g-3+2n$, meaning we have the formula \begin{equation}\label{eq:dimRef}\mathrm{vdim} \LogGW(X;\Lambda,\gamma) = g-1+n-p - \sum_i \deg \gamma_i\end{equation} for the homological degree. For our choice of $\Lambda$, this gives homological degree $1-1+4-0-4 = 1$.

We then consider the following gluing map
\[
\gl\colon \Mpt_{1,\{p_1,q_+\}}\times\Mpt_{1,\{p_2,p_3,p_4,q_-\}}\longrightarrow
\Mpt_{2,4},
\]
and wish to compute $\gl^* \LogGW(X;\Lambda,\gamma)$ in homology using the formula in \ref{thm:loggluingkunneth}. We will prove the following formula.

\begin{proposition}
\label{prop:theexample}
The degree of the dimension $0$ class $\gl^! \LogGW(X;\Lambda,\gamma)$ is
\[
	\frac{-k^3}{24}\binom{k+2}{5}.
\]
\end{proposition}

\begin{remark}
This calculation could be lifted to homological log Chow groups instead of homology, as the diagonals whose K\"unneth decomposition we need also admit a Chow--K\"unneth decomposition. That would give a formula for the pushforward of $\gl^! \LogGW(X;\Lambda,\gamma)$ as a point class on the moduli space $\Mpt_{1,2} \times \Mpt_{1,4}$.
\end{remark}

We first need to determine which pairs of discrete data appear in $\gl^* \Mpt_\Lambda(X)$, that is, we need to determine the set $\gl^* \Lambda$. By \ref{rem:sepglLambda} the pairs $(\Xi_+, \Xi_-) \in \gl^* \Lambda$ are uniquely determined by the curve classes $\beta_+, \beta_-$ by the global balancing condition $\beta_+ D_j = \sum_i a_{ij}^+$.

\begin{lemma}
For $(\Xi_+, \Xi_-) \in \gl^* \Lambda$ we have $\beta_+ = 0$ or $\beta_+ = kH$.
\end{lemma}
\begin{proof}
We must have $\beta_+ = dH$ for some $0 \leq d \leq k$. Then the global balancing condition tells us that the contact order $b$ along the new marking $q_+$ is given by $(d-k,d,d)$. However, the triple intersection $D_0 \cap D_1 \cap D_2$ is empty, so at most two can be non-zero, so $d = 0$ or $d = k$.
\end{proof}

\begin{lemma}
The term in \ref{thm:loggluingkunneth} where $(\Xi_+,\Xi_-)$ satisfies $\beta_+ = kH$ is $0$.
\end{lemma}
\begin{proof}
We will show that the restriction of $\gl^* [\Mpt_\Lambda(X)]^\vir \cap \ev_4^* \gamma_4$ is $0$. Since $\beta_- = 0$, the entire second component is contracted to a point in (a blowup of) $X$. Since the four contact orders are $ke_1, ke_2, 0 , -ke_1-ke_2$, this point must map to both $D_1$ and $D_2$. But $\gamma_4 = [\pt]_\P^2$. Taking any point outside of $D_1 \cap D_2$ as a representative, we see this component of $\gl^* [\Mpt_\Lambda(X)]^\vir \cap \ev_4^* \gamma_4$ is supported on an empty moduli space, and is hence $0$.
\end{proof}

So it remains to compute the contribution from the component where $\beta_+ = 0$, and the contact order $b$ along the new edge (oriented from $q_+$ to $q_-$) is $-ke_0$. To spell out the discrete data, we have
\begin{align*}
\Xi_+&=\bigl(g=1,n=2,\ \beta=0,\ \{ke_0,\,-ke_0\}\bigr)\\
\Xi_-&=\bigl(g=1,n=4\ \beta=kH,\ \{ke_1,\,ke_2,\,0,\,ke_0\}\bigr).
\end{align*}
The tropical evaluation maps are already integral, hence there is no need to blowup, and we have $m_\Xi = 1$, and $E = E_{\Xi;i} = D_0 = \P^1$. The contribution $c_+^\Xi c_-^\Xi$ is $k\cdot 1$. After capping with respectively $\gamma_1$ and $\gamma_2,\gamma_3,\gamma_4$ the virtual dimensions of the two parts are by \ref{eq:dimRef} $0$ and $1$. The diagonal of $\P^1$ has Chow--K\"unneth decomposition $\mathrm{pt} \tensor \P^1 + \P^1 \tensor \mathrm{pt}$, and by dimension counts only the term $\P^1 \tensor \mathrm{pt}$ contributes. All in all the formula \ref{thm:loggluingkunneth} reduces to
\begin{equation}
\label{eq:kunnethcomputationintermediate}
	\gl^* \LogGW(X;\Lambda,\gamma) = k\LogGW(X;\Xi_+,\gamma_+) \boxtimes \LogGW(X;\Xi_-,\gamma_-)
\end{equation}
where
\[
	\gamma_+ = ([\mathrm{pt}]_{D_0},1), \gamma_- = ([\mathrm{pt}]_{D_1},\,1,\,[\mathrm{pt}]_{\P^2},[\mathrm{pt}]_{D_0}).
\]

We will now compute both terms separately.

\subsection{The pierced logarithmic Gromov--Witten invariant}
The first invariant has discrete data
\[\Xi_+=\bigl(g=1,n=2,\ \beta=0,\ \{ke_0,\,-ke_0\}\bigr),\]
and is hence a pierced logarithmic Gromov--Witten invariant (although in this simple case by \ref{prop:comparisonbnr} it is up to a factor $k$ equal to a refined log invariant from \cite{BNR}).
As a moduli space, $M_\Xi(X)$ is isomorphic to $\LogDRL_1(k,-k) \times D_0$, the product of the logarithmic double ramification locus and $D_0$ (remembering the image of the map). The obstruction theory however is somewhat different, given by $\pi_* T_\log X$, and hence the virtual structure is given by the log double double ramification cycle $\LogDDR_1((k,-k),(0,0)) \times D_0$. We compute the virtual fundamental class is $-\frac1k \LogDR_1(k,-k) \cdot \lambda_1 \times [D_0]$. Here the factor $1/k$ comes from the tropical virtual fundamental class $\ell_1 \ell_+ r_+$, where $r_+$ is $\frac{T}{k}$, with $T$ measuring the distance from the pierced leg to the wall.

\begin{proposition}
\label{prop:theexamplepierced}
We have \[\LogGW(X;\Xi_+,\gamma_+) = -\frac 1k \LogDR_1(k,-k) \lambda_1 = -\frac{k}{24} [\mathrm{pt}].\]
\end{proposition}
\begin{proof}
This follows from the classical computation $\DR_1(k,-k)= k^2 \psi_1 - \lambda_1$ and $\psi_1 \lambda_1 = \frac{1}{24}$.
\end{proof}

\subsection{The classical logarithmic Gromov--Witten invariant}

Now we will consider the second invariant. To keep the notation in this subsection compact and readable, we reorder the points, and rename them $q_0, q_1, q_2, q_3$. Then the discrete data is
\[
\Xi_-=\bigl(g=1,n=4\ \beta=kH,\ \{ke_0, ke_1, ke_2, 0\}\bigr)
\]
and insertions
\[\gamma_- = ([\mathrm{pt}]_{D_0}, [\mathrm{pt}]_{D_1},\,1,\,[\mathrm{pt}]_{\P^2}).\]

One could use \cite{MandelRuddat2020LogGWToric} to determine this invariant; however, we give a short, more self-contained proof, based on techniques from \cite{Ranganathan2023Logarithmic-Gromov-Witten}. There log Gromov--Witten invariants of toric varieties are expressed as intersections of logarithmic double ramification cycles multiplied by a piecewise polynomial function. We will first express our invariant in such a way, effectively summarising a special case of \cite[Theorem~B]{Ranganathan2023Logarithmic-Gromov-Witten}, and then compute the resulting expression.

We will first determine the virtual fundamental class multiplied with the point insertions at $q_0,q_1$.

First, the moduli space $M_{\Xi_-}(X)$ naturally maps to the logarithmic double double ramification locus $S = \LogDDRL((k,0,-k,0),(0,k,-k,0))$. Over $S$ there is a $\G_\log^2$-torsor of trivialisations, and $M_{\Xi_-}(X)$ is a blowup of this torsor. This torsor has a section, by asking for $\ev_0(q_i) = 1$ inside $D_i = \P^1$ for $i = 0,1$. Then we see $[M_{\Xi_-}(X)]^\vir \cap \ev_0^* \gamma_0 \cap \ev_1^* \gamma_1$ is represented by the pushforward of the virtual fundamental class $[S]^\vir$ along this torsor section.

Then we get a composition $\ev_3: S \to \G_\log^2$, and we need to compute the virtual degree of this map, i.e., $[S]^\vir \cap \ev_3^* \gamma_3$. For this, we first set up some notation. We denote the rays of $\P^2$ by $\rho_0, \rho_1, \rho_2$, and their corresponding PL functions by $\phi_0, \phi_1, \phi_2$. Then we take the representative of $\gamma_3$ given by the PP function $\phi_0 \phi_1$, we write $\phi = \ev_3^* \phi_0 \phi_1$ and the invariant we are after is $[S]^\vir \cap \phi$. That is exactly the result of \cite[Theorem~B]{Ranganathan2023Logarithmic-Gromov-Witten}; now we must compute this intersection product.

Tropically, the map $S \to \G_\log^2$ is given by the universal curve $S$ over $M = \LogDDRL^t((k,0,-k),(0,k,-k)) \subset \Mpt_{1,3}^t$ mapping to $\G_\log^2$ with slopes $(k,0),(0,k),(-k,-k)$ at the legs. Here the legs corresponding to $q_0,q_1$ lie on the rays $\rho_0,\rho_1$, but $q_2$ need not. An example of such a tropical curve, i.e. the tropicalisation of a fiber of $S/M$, is shown in \ref{fig:tropE}.

\tikzset{
  midarr/.style={postaction={decorate,decoration={markings,
      mark=at position #1 with {\arrow{Stealth[length=3.0mm]}}}}},
  midarr/.default=0.58,
  curve/.style={blue,line width=1.45pt},
  edgeE/.style={blue,line width=3.2pt},
  ray/.style={black,line width=0.6pt},
}

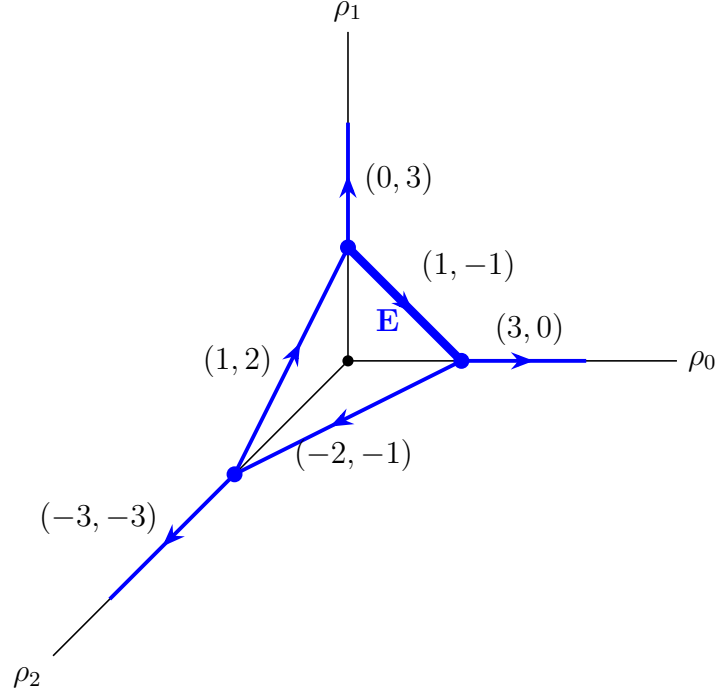
\begin{figure}
\centering
\begin{tikzpicture}[scale=1.5]

\coordinate (O)  at (0,0);
\coordinate (V1) at (0,1);
\coordinate (V2) at (1,0);
\coordinate (V3) at (-1,-1);
\coordinate (P1) at (0,2.1);
\coordinate (P2) at (2.1,0);
\coordinate (P3) at (-2.1,-2.1);

\draw[ray] (O) -- (0,2.9)   node[above] {$\rho_1$};
\draw[ray] (O) -- (2.9,0)   node[right] {$\rho_0$};
\draw[ray] (O) -- (-2.6,-2.6) node[below left] {$\rho_2$};

\draw[curve,midarr] (V1) -- (P1);
\draw[curve,midarr] (V2) -- (P2);
\draw[curve,midarr] (V3) -- (P3);

\draw[edgeE,midarr] (V1) -- (V2);
\draw[curve,midarr] (V2) -- (V3);
\draw[curve,midarr] (V3) -- (V1);

\fill[blue] (V1) circle (2.0pt);
\fill[blue] (V2) circle (2.0pt);
\fill[blue] (V3) circle (2.0pt);
\fill (O) circle (1.4pt);


\node[right=2pt] at (0,1.60)      {$(0,3)$};
\node[above=2pt] at (1.60,0)      {$(3,0)$};
\node[above left=-1pt] at (-1.6,-1.6) {$(-3,-3)$};

\node[blue,below left=0pt] at (0.55,0.55) {$\mathbf{E}$};
\node[above right=0pt]     at (0.55,0.60) {$(1,-1)$};
\node[below=4pt]        at (0.05,-0.5)    {$(-2,-1)$};
\node[left=3pt]         at (-0.5,0)    {$(1,2)$};

\end{tikzpicture}
\caption{A tropical curve mapping to $\P^2$, with the slopes at each edge given. This appears as the tropicalisation of a fiber of $S/M$.}
\label{fig:tropE}
\end{figure}

We see the piecewise polynomial $\phi_0\phi_1$ is non-zero on exactly the interior of the positive orthant. This vanishes everywhere, unless there is an edge connecting the vertices incident to $q_0, q_1$. The tropical curves that meet this interior are all given by curves of topological type as pictured in \ref{fig:tropE}, but with slopes $(a,-b),(a-k,-b),(a,k-b)$ along the inner triangle (starting at the marked edge $E$), where $a, b, c = k - a - b > 0$. Then $\phi$ vanishes everywhere except on the interior of the edge $E$ connecting the vertices $q_0, q_1$. Since this class is $0$-zero dimensional (it is supported on the singular locus of $S$, over the singular curves in $M$), its virtual fundamental class and fundamental class agree.

We now fix $(a,b,c)$ summing to $k$, and a logarithmic curve $(C,f) \in M$ whose tropicalisation $f^t: C^t \subset \R^2$ has slopes as described above. We first note that because the slopes along $E$ are $a,b$, the pullback is $ab[E]$. Next, fixing an identification $\psi: C^t \isom \R/\ell \R$ where $\ell$ of the loop and $\psi(q_2) = 0$ and using that $C^t$ embeds into $\R^2$ with given slopes, we find $\psi(q_0) = \frac{a}{k} \ell$ and $\psi(q_1) = -\frac{b}{k} \ell = \frac{a+c}{k} \ell$. Then $E$ is the edge between $\psi(q_0)$ and $\psi(q_1)$ of length $\frac{c}{k} \ell$, and the pushforward $\pi_* [E]$ along $S/M$ gives the $\frac{c}{k} [(C,f)] \in H_0(M)$.

It remains to count how many logarithmic curves in $M$ have tropicalisation with these slopes. As we saw, the slopes uniquely determine the image of the torsion points $q_0, q_1$ in the tropicalisation of the curve as $\frac{a}{k} \ell, -\frac{b}{k} \ell$ (when $q_2$ is taken of the origin as the log elliptic curve), and whenever these are the tropicalisations, there is piecewise linear function $C^t \to \R^2$ with these slopes. So the answer is the square of the number of lifts from a $k$-torsion point of $C^t$ to a $k$-torsion point of $C$, for $C$ the unique non-smooth log elliptic curve. The respective $k$-torsion groups of $C,C^t$ have cardinality $k^2,k$, so the total answer is $\left(\frac{k^2}{k}\right)^2 = k^2$.

All in all, we find the following proposition.
\begin{proposition}
\label{prop:theexample:classic}
We have \[\LogGW(X;\Xi_-,\gamma_-) = k \sum_{a+b+c = k} abc [\mathrm{pt}] = k\binom{k+2}{5} [\mathrm{pt}].\]
\end{proposition}

\begin{proof}[Proof of \ref{prop:theexample}]
This now follows directly from \ref{eq:kunnethcomputationintermediate} and \ref{prop:theexample:classic,prop:theexamplepierced}.
\end{proof}

\bibliographystyle{alpha}
\bibliography{prebib}

\end{document}